\documentclass[aos]{imsart}

\usepackage{amsthm,amsmath,amssymb}
\RequirePackage[colorlinks,citecolor=blue,urlcolor=blue]{hyperref}

\startlocaldefs

\DeclareMathOperator*{\as}{as-}
\DeclareMathOperator*{\aslim}{as-lim}
\DeclareMathOperator*{\asliminf}{as-lim\,inf}
\DeclareMathOperator*{\aslimsup}{as-lim\,sup}

\DeclareMathOperator*{\elim}{e-lim}
\DeclareMathOperator*{\aselim}{as-e-lim}
\newcommand{\bbd}{{\ooalign{$d$\cr $\mkern6.8mul$}}}

\usepackage{amssymb}
\usepackage{graphicx}
\usepackage{booktabs}
\usepackage{subcaption}

\theoremstyle{plain}
\newtheorem{theorem}{Theorem}

\newtheorem{proposition}{Proposition}

\theoremstyle{remark}
\newtheorem{remark}{Remark}
\newtheorem{example}{Example}
\newtheorem{assumption}{Assumption}

\endlocaldefs

\begin{document}

\begin{frontmatter}

% "Title of the paper"
%\title{Consistency of an Optimization Hierarchy for Fair Statistical Decision Problems}
\title{Optimization Hierarchy for Fair \\Statistical Decision Problems}
\runtitle{Fair Optimization}

\begin{aug}
% indicate corresponding author with \corref{}
\author{\fnms{Anil} \snm{Aswani}\corref{}\ead[label=e1]{aaswani@berkeley.edu}}\and
\author{\fnms{Matt} \snm{Olfat}\thanksref{t1}\ead[label=e2]{molfat@berkeley.edu}}
\thankstext{t1}{This material is based upon work supported by the National Science Foundation under Grant CMMI-1847666, and by the UC Berkeley Center for Long-Term Cybersecurity.} 
%\address{Industrial Engineering \\\qquad and Operations Research\\ Berkeley, CA 94720\\ \printead{e1} \\\phantom{E-mail:\ }\printead*{e2}}
\address{Industrial Engineering \\\qquad and Operations Research\\ Berkeley, CA 94720\\ E-mail:\ \href{mailto:aaswani@berkeley.edu}{\textup{aaswani@berkeley.edu}}\\\phantom{E-mail:\ }\href{mailto:molfat@berkeley.edu}{\textup{molfat@berkeley.edu}}}

\affiliation{University of California, Berkeley}
% \affiliation{Some University}

\end{aug}
%\author{\fnms{Anil} \snm{Aswani}\ead[label=e1]{molfat@berkeley.edu}}
%\address{\printead{e1}}\and\author{\fnms{Matt} \snm{Olfat}\ead[label=e2]{aaswani@berkeley.edu}}

\runauthor{Aswani and Olfat}

\begin{abstract}
Data-driven decision-making has drawn scrutiny from policy makers due to fears of potential discrimination, and a growing literature has begun to develop fair statistical techniques. However, these techniques are often specialized to one model context and based on ad-hoc arguments, which makes it difficult to perform theoretical analysis. This paper develops an optimization hierarchy{, which is a sequence of optimization problems with an increasing number of constraints,} for fair statistical decision problems. Because our hierarchy is based on the framework of statistical decision problems, this means it provides a systematic approach for developing and studying fair versions of hypothesis testing, decision-making, estimation, regression, and classification. We use the insight that qualitative definitions of fairness are equivalent to statistical independence between the output of a statistical technique and a random variable that measures attributes for which fairness is desired. We use this insight to construct an optimization hierarchy that lends itself to numerical computation, and we use tools from variational analysis and random set theory to prove that higher levels of this hierarchy lead to consistency in the sense that it asymptotically imposes this independence as a constraint in corresponding statistical decision problems. We demonstrate numerical effectiveness of our hierarchy using several data sets, and we use our hierarchy to fairly perform automated dosing of morphine.
\end{abstract}



\end{frontmatter}

\section{Introduction}

There is growing concern that improperly designed data-driven approaches to decision-making may display biased or discriminatory behavior. In fact, such concerns are justified by numerous examples of unfair algorithms that have been deployed in the real world \cite{angwin2016machine,barocas2016big,nature2016more,executive2016big}. In response, researchers have started to develop a number of approaches to encourage fairness in various statistical or machine learning problems \cite{calders2009building,chouldechova2017fair,dwork2012fairness,hardt2016equality,olfat2018spectral,zafar2017,zliobaite2015relation}. The problem of classification has received particular attention due to the ease of mapping class labels to positive and negative outcomes with which to characterize fairness, but recent work has also begun to explore fair statistical methods in the context of unsupervised learning \cite{chierichetti2017fair,olfat2018convex} and in more general decision-analytic frameworks \cite{ensign2017runaway,liu2018delayed}.

\subsection{Existing Approaches to Fairness}

The literature on fair statistics and learning can be classified into three categories: pre-processing steps, post-processing steps, and training regularization. The general setup of these approaches is that they seek to estimate a model that predicts a {response} variable using a vector of {predictor} variables, while trying to ensure that the model predictions are fair (we discuss quantitative measures of fairness in the next subsection) with respect to some {(potentially multiple)} variables that indicates a protected attribute (e.g., gender or race). Here we briefly review some of the existing approaches that have been developed for fairness. 

Pre-processing approaches transform the data before estimation, to remove any protected information that could cause unfairness. For instance, \cite{calmon2017optimized,zemel2013learning} take a nonparametric approach: They optimize over distributions to variationally transform the feature space. {However, the underlying optimization problem quickly becomes intractable because its computation scales as exponential in dimension.} Alternatively, \cite{olfat2018convex} take an adversarial outlook on pre-processing for fairness, and propose a semidefinite programming (SDP) formulation to calculate a ``fair principal component analysis (FPCA)" that can then be used to this end. Several groups have designed autoencoders, with a similar inspiration, oriented around deep classifiers \cite{beutel2017,edwards2015,madras2018,zhang2018}. 

In comparison, there is a smaller literature on post-processing for fairness. These methods take the output of a statistical technique, and process the output in order to improve fairness. A canonical example of this approach is \cite{hardt2016equality}, which designs a method for post-processing an arbitrary classifier in order to ensure fairness. While this method is flexible with regards to the type of classifier used, it achieves fairness by requiring different score function thresholds for different groups of protected classes. This violates a general principle called \textit{individual fairness} \cite{dwork2012fairness}, which says that similar individuals should be treated similarly. More significantly, \cite{woodworth2017learning} show that this method achieves suboptimal tradeoffs between accuracy and fairness.

Notably, both pre-processing and post-processing approaches are necessarily greedy since they unlink the process of estimation from ensuring fairness. This has motivated work on regularization approaches to fairness, which generally achieve lower generalization error while improving fairness. The regularization approaches most related to this paper include \cite{berk2017fairness,olfat2018spectral,woodworth2017learning,zafar2017,agarwal2018reductions,agarwal2019fair,oneto2019general,donini2018empirical}. In particular, \cite{zafar2017} control the correlation of a classifier score function and the protected attribute, which can be formulated as a linear constraint in the estimation problem. The method in \cite{olfat2018spectral} implements non-convex optimization techniques to further consider second-order deviations. However, a limitation of both is they are applicable only when protected attributes are binary. The approach of \cite{kamishima2012fairness,zemel2013learning} works for more general types of protected attributes, but it requires a heuristic to approximate a \emph{mutual information} (MI) measure of fairness as a constraint. Alternatively, \cite{goh2016satisfying} designs an iterative cutting-plane algorithm for fair support vector machine (SVM) that requires solving an SVM instance in each iteration. Moving away from classification, \cite{calders2013controlling,johnson2016impartial,agarwal2019fair} develop concepts of fairness in the case of regression, and \cite{berk2017convex} extends this to regularization techniques for ensuring different qualitative types of fairness in regression. {Empirical risk minimization formulations for classification \cite{agarwal2018reductions,donini2018empirical}, regression \cite{agarwal2019fair}, and general problems \cite{oneto2019general} have also been proposed.} Finally, recent work has sought to generalize these ideas towards fair decision-making \cite{ensign2017runaway,liu2018delayed}.

\subsection{Quantitative Measures of Fairness}
\label{sec:robcondind}

We have casually used the terms fairness and bias without formally defining them. Part of the difficulty is a considerable lack of clarity in the existing literature as to their meaning, with different works defining different quantitative measures of fairness. We believe the underlying (and unifying) idea behind all these measures is they approximate in some way a measure of independence between the output of the statistical procedure and the variable of protected attributes. In fact, this way of thinking about fairness was first noticed by \cite{kamishima2012fairness}. 

%To make our discussion concrete, we first discus notions of fairness for classification. Let $(X,Y,Z) \in \mathbb{R}^p\times\{\pm 1\}\times\{\pm 1\}$ be a jointly distributed random variable consisting of a vector of predictors, a binary class label, and a binary protected attribute. {Let $\delta(x)$ be a score for a classifier that operates on $X$, and note that we could also consider a score $\delta(x,z)$ that operates on both $X$ and $Z$ but that such a score violates individual fairness}. Suppose the classifier makes binary predictions $\widehat{Y}(x,t) = \mathrm{sign}(t-\delta(x))$ for a given threshold $t$ of the score. Since binary classifiers output a $\pm 1$ that can be mapped to desirable/undesirable decisions, one measure of fairness is 
To make our discussion concrete, we first discuss notions of fairness for classification. Let $(X,Y,Z) \in \mathbb{R}^p\times\{\pm 1\}\times\{\pm 1\}$ be a jointly distributed random variable consisting of a vector of predictors, a binary class label, and a binary protected attribute. Let $\delta(x)$ be a score for a classifier { that operates on $X$}, and suppose the classifier makes binary predictions $\widehat{Y}(x,t) = \mathrm{sign}(t-\delta(x))$ for a given threshold $t$ of the score. Since binary classifiers output a $\pm 1$ that can be mapped to desirable/undesirable decisions, one measure of fairness is 
\begin{equation}
\label{eq:dispimpact}
KS = \max_{t\in\mathbb{R}}\big|\mathbb{P}[\widehat{Y}(X,t)=+1|Z=+1]-\mathbb{P}[\widehat{Y}(X,t)=+1|Z=-1]\big|.
\end{equation}
{This measures how similar the probability of making a prediction of a given binary class is between the two groups specified by the protected attributed}, and it is often called \emph{disparate impact} \cite{hardt2016equality,olfat2018spectral}. Effectively, disparate impact measures the total disparity in outcomes between protected classes. 
%{This measures the similarity (as measured between the two groups specified by the protected attributed) of the binary predictions}, and it is often called \emph{disparate impact} \cite{hardt2016equality,olfat2018spectral}. Effectively, disparate impact measures the total disparity in outcomes between protected classes. 

This above measure of fairness can be too strict in some applications, as there may be unavoidable correlation between the classifier output and the protected label. For such cases, \cite{hardt2016equality} proposes \textit{equalized odds} as an alternative measure of fairness that instead constrains disparity in outcomes conditional on some informative variable. In the setting of binary classification, one possible informative variable is $Y\in\{\pm 1\}$ itself. This choice leads to the following quantitative measure of equalized odds fairness:
\begin{multline}
\label{eq:equaloddsbin}
EO = \max_{y\in\{\pm 1\}}\max_{t\in\mathbb{R}}\big|\mathbb{P}[\widehat{Y}(X,t)=+1|Z=+1,Y=y]-\\
\mathbb{P}[\widehat{Y}(X,t)=+1|Z=-1,Y=y]\big|.
\end{multline}
Restated, the quantity (\ref{eq:equaloddsbin}) measures the disparity in \emph{error rates} between the protected classes. An additional benefit is that a classifier with zero training error will also be fair with respect to this measure of fairness \cite{hardt2016equality}.

At an initial glance, the above measures of fairness do not look like manifestations of independence. Yet note the event $\{\widehat{Y}(X,t)=+1\}$ is equivalent to the event $\{\delta(X) \leq t\}$ since $\widehat{Y}(x,t) = \mathrm{sign}(t-\delta(x))$. This means that (\ref{eq:dispimpact}) is the Kolmogorov-Smirnov (KS) distance between the distributions of $\delta(X)|Z=+1$ and $\delta(X)|Z=-1$. Since (\ref{eq:equaloddsbin}) has a very similar interpretation, we will focus our discussion on (\ref{eq:dispimpact}). Thus when $KS = 0$ in (\ref{eq:dispimpact}), we have that 
\begin{equation}
G(t) := \mathbb{P}[\delta(X)\leq t|Z=+1] = \mathbb{P}[\delta(X)\leq t|Z=-1].
\end{equation}
This means that the joint distribution factorizes as
\begin{equation}
\mathbb{P}(\delta(X) \leq t, Z=z) = \mathbb{P}[\delta(X)\leq t|Z=z]\cdot\mathbb{P}(Z=z) = G(t)\cdot\mathbb{P}(Z=z),
\end{equation}
which means the two random variables are independent. Summarizing, we have $KS = 0$ in (\ref{eq:dispimpact}) if and only if $\delta(X)$ is independent of $Z$. The importance of such independence in relation to fairness was first noticed by \cite{kamishima2012fairness}. 

%is the larger of the two KS distances between the distributions of $\delta(X)|Z=+1,Y=y$ and $\delta(X)|Z=-1,Y=y$ for $y=\pm 1$.
 %In this paper, we will proceed with definition \cref{eq:dispimpactbin}, but our model generalizes easily to accommodate \cref{eq:equaloddsbin}. In particular, we consider a more general version of disparate impact, where our classifier is of the form $d(x,t)=\sign(\delta(x)-t)$ for some activation function $\delta$. Now, our notion of disparate impact becomes
%
%
%\noindent Note that \cref{eq:dispimpact} amounts to bounding the KS distance between $\delta(x)$ and $z$. This is a more rigorous requirement than \cref{eq:dispimpactbin}, but is more robust to generalization and to possibly adversarial choice of the thresholding parameter, $t$.

\subsection{Technical Challenges with Independence}

The above discussion suggests that a promising direction for generalizing fairness to a broader class of problems is to ensure independence (or rather some approximate notion of independence) between the output of a statistical technique and a random variable that measures attributes for which fairness is desired. In fact, the broader idea of quantifying independence using an empirical estimate has a long history in statistics \cite{breiman1985estimating,chen2005consistent,feuerverger1977empirical,pal2010estimation,szekely2007measuring,szekely2009brownian,MAL-060,gretton2005measuring}. One approach is to compute some generalized notion of correlation such as Renyi correlation, distance correlation, or the {Hilbert Schmidt Independence Criterion (HSIC).} Another approach is to use some distance like the KS distance, total variation distance, or mutual information between the empirical probability measures of the joint and product distributions.

However, incorporating empirical independence measures into statistical procedures is not straightforward. Many statistical procedures are computed by solving an optimization problem, and so such measures must be added as constraints. However, measures like Renyi correlation, { HSIC}, KS distance, total variation distance, and mutual information are all themselves the solutions of an optimization problem. (Mutual information is traditionally defined using a hard-to-compute integral, but a well-known variational characterization \cite{boucheron2013concentration} shows that it should more properly be thought of as the solution to an optimization problem for our discussion.) This means the resulting optimization problem for a fair statistical procedure defined in this way would have another optimization problem as a constraint; these types of problems are known as bilevel programs and are very difficult to numerically solve \cite{dempe2002,ouattara2018duality}. The numerical difficulties are compounded for those measures defined using an empirical c.d.f., which is always discontinuous. { HSIC and distance correlation are an exception to the above statement in that these quantities can be estimated by an explicit formula, and so an optimization problem with HSIC or distance correlation as a constraint is simply an optimization problem with a nonlinear constraint corresponding to the empirical estimate of the HSIC or distance correlation. There is in fact a history of using HSIC as a component of optimization problems for tasks such as feature selection \cite{song2007supervised} and clustering \cite{song2007dependence}.}

\subsection{Contributions and Outline}

This paper develops an optimization hierarchy for fair statistical decision problems. We first generalize in Section \ref{sec:fsdp} the framework of statistical decision problems \cite{lehmann2006testing} to include fairness. This provides a systematic approach for developing and studying fair versions of hypothesis testing, decision-making, estimation, regression, and classification. We use the above discussed insight relating fairness to statistical independence in order to propose in Section \ref{sec:foh} an optimization hierarchy that lends itself to numerical computation. Tools from variational analysis and random set theory are used to prove in Section \ref{sec:scfoh} that higher levels of this hierarchy lead to consistency in the sense that it asymptotically imposes independence as a constraint in corresponding statistical decision problems {for bounded random variables. Section \ref{sec:ubrv} generalizes these results to unbounded random variables, namely sub-Gaussian random variables and random variables with finite moments.} In Section \ref{sec:er}, we demonstrate numerical effectiveness of our hierarchy using several data sets, and we conclude by using our hierarchy to fairly perform automated dosing of morphine.

The distinguishing feature of our approach to ensuring independence is to use a moment-based characterization of independence that generalizes Kac's theorem \cite{bisgaard2006does,kac1936fonctions} to multivariate random variables. This has the key practical benefit over other approaches to measuring independence (such as \cite{kamishima2012fairness,zemel2013learning}) that all the resulting constraints in the corresponding optimization problems are smooth polynomials. This means we avoid the bilevel programming structure that arises from the use of other independence measures \cite{kamishima2012fairness,zemel2013learning}, and which makes numerical optimization very difficult. Because the moment constraints are smooth polynomials, this further allows us to leverage advances in convex optimization \cite{lasserre2010moments} and related heuristics such as the constrained convex-concave procedure \cite{smola2005kernel,tuy1995dc,yuille2002concave} for the purpose of numerically solving the resulting optimization problem. The tradeoff is that we have to include multiple (but a finite number of) constraints, one for each possible combination of moments between joint and product distributions.

Our framework also builds on preliminary work on the use of moment-based constraints for fair statistical methods \cite{olfat2018spectral,olfat2018convex,zafar2017}. These approaches were restricted to binary classification with binary protected classes, made use of only first- or second-order moments of only the classifier, were based on ad-hoc arguments and justifications, and lacked theoretical analysis of the resulting statistical methods. The past papers \cite{olfat2018spectral,olfat2018convex,zafar2017} leave open the larger question of how moment-based approaches to fairness can be generalized to continuous protected classes, multivariate protected classes, multivariate statistical decisions, and other classes of statistical problems beyond classification. Our work in this paper unifies these past approaches into a broader theoretical framework, {proves this framework provides asymptotic and finite-sample guarantees on fairness}, and successfully achieves a generalization of moment-based methods in order to handle continuous protected classes, multivariate protected classes, multivariate statistical decisions, and multiple classes of statistical decision problems, including fair versions of hypothesis testing, decision-making, estimation, regression, and classification.

%{Empirical risk minimization formulations for fair statistics have been recently proposed \cite{agarwal2018reductions,agarwal2019fair,donini2018empirical,oneto2019general}. The setups of these papers are similar to our framework, but differ in several important ways. Fairness is defined in \cite{agarwal2018reductions,donini2018empirical} using conditional probabilities, which because of the classification setup considered can be exactly rewritten as a conditional expectation. This allows the fairness constraints to be represented by a finite number of inequalities using sample averages in place of the conditional expectations. In contrast, our framework applies to problems such as regression where fairness as defined by statistical independence cannot be exactly rewritten as a conditional expectation. The work in \cite{agarwal2019fair} extends these ideas to regression by a performing a discretization that results in approximation of regression by a classification problem. The fairness constraints in this approach require discrete protected classes, whereas our framework is also able to handle continuous and vector-valued (consisting of both discrete and continuous) protected attributes. The formulation in \cite{oneto2019general} applies to general risk minimization problems, defines fairness in terms of conditional expectation, and proposes an approximation to ensure convexity of the resulting optimization problem. Our framework defines a different definition of fairness in terms of statistical independence.}

{Empirical risk minimization formulations for fair statistics have been recently proposed \cite{agarwal2018reductions,agarwal2019fair,donini2018empirical,oneto2019general}. These papers are similar to our framework, but differ in several important ways. Fairness is defined in \cite{agarwal2018reductions,donini2018empirical} using conditional probabilities, which because of the classification setup considered can be exactly rewritten as a conditional expectation. This allows the fairness constraints to be represented by a finite number of inequalities using sample averages in place of the conditional expectations. In contrast, our framework applies to problems such as regression where fairness as defined by statistical independence cannot be exactly rewritten as a conditional expectation. The work in \cite{agarwal2019fair} extends these ideas to regression by a performing a discretization that results in approximation of regression by a classification problem. The fairness constraints in this approach require discrete protected classes, whereas our framework is also able to handle continuous and vector-valued (consisting of both discrete and continuous) protected attributes. The formulation in \cite{oneto2019general} applies to general risk minimization problems, defines fairness in terms of conditional expectation, and proposes an approximation to ensure convexity of the resulting optimization problem. Our framework defines a different definition of fairness in terms of statistical independence.}

%Because we have to include multiple constraints, this significantly complicates the theoretical analysis of our optimization hierarchy. The limiting behavior of our framework requires a statistical analysis on the solution to an optimization problem in the limit of a countably-infinite number of random constraints involving empirical moments. Traditional results in statistics do not apply to set-valued functions \cite{aswani2019statistics}, which are one way to interpret constraints in an optimization problem \cite{rockafellar2009variational}. In fact, most attention in statistics on sets has been focused on estimating a single set under different measurement models \cite{devroye1980,geffroy1964,guntuboyina2012,korostelev1995,patschkowski2016,renyi1963,scholkopf2001}. A traditional approach to analysis is to metricize the set of sets using the Pompeiu–Hausdorff metric, but such an approach is intractable in a setting such as ours with random sets defined using an infinite number of non-convex constraints. Instead, we build on our past work on statistics with set-valued functions \cite{aswani2019statistics} and develop new theoretical statistics analysis arguments (based on variational analysis \cite{rockafellar2009variational,majumdar2014convex,zhao2019optimal} and random sets \cite{matheron1975,molchanov2006}) applicable to our setting.

Because we have to include multiple constraints, this significantly complicates the theoretical analysis of our optimization hierarchy. The limiting behavior of our framework requires a statistical analysis on the solution to an optimization problem in the limit of a countably-infinite number of random constraints involving empirical moments. Traditional results in statistics do not apply to set-valued functions \cite{aswani2019statistics}, which are one way to interpret constraints in an optimization problem \cite{rockafellar2009variational}. In fact, most attention in statistics on sets has been focused on estimating a single set under different measurement models \cite{devroye1980,guntuboyina2012,korostelev1995,patschkowski2016,scholkopf2001}. The traditional theoretical argument is to use the Pompeiu–Hausdorff distance to metricize the set of sets, but this approach is too difficult for use in our setting which has random sets defined using (in the limit) an infinite number of non-convex constraints. Instead, we build on our past work on statistics with set-valued functions \cite{aswani2019statistics}: We develop new theoretical arguments for statistics with random sets and set-valued functions, using variational analysis \cite{rockafellar2009variational,royset2019} and random sets \cite{matheron1975,molchanov2006}. {These techniques are of potential interest to other set-based statistical problems where empirically-successfully approaches without theoretical guarantees have been used \cite{zaheer2017deep}. Examples of such statistical problems include estimation tasks where the predictor variables are a set and the response variable is a scalar, such as galaxy red-shift estimation in cosmology \cite{rozo2014redmapper} and point-cloud classification in computer vision \cite{wu20153d}.}

\section{Preliminaries}
\label{sec:notation}

This section presents our notation. We also describe some useful (and needed) notation and definitions from variational analysis and random sets.  Most of the variational analysis definitions are from \cite{rockafellar2009variational}, and the stochastic set convergence notation is originally from \cite{aswani2019statistics}.

\subsection{Notation}

Let $M : \mathbb{R}^{dp} \rightarrow \mathbb{R}^{d\times p}$ be the function that reshapes a vector into a matrix by placing elements into the matrix columnwise from the vector. Similarly, we define $W := M^{-1}: \mathbb{R}^{d\times p}\rightarrow\mathbb{R}^{dp}$ to be its inverse.

We use $\mathbb{E}_n(\cdot)$ to denote expectation with respect to the empirical distribution. Recall this is the sample average of the random variable inside parenthesis. As examples, $\mathbb{E}_n(Z) = \frac{1}{n}\sum_{i=1}^nZ_i$ and $\mathbb{E}_n(ZX) = \frac{1}{n}\sum_{i=1}^nZ_iX_i$.

Consider a tensor $\varphi\in\mathbb{R}^{r_1\times\cdots\times r_q}$, and let $[r] = \{1,\ldots,r\}$. The norm $\|\varphi\|$ is the $\ell_\infty$ vector norm for the tensor considered as a vector. For two tensors $\varphi,\nu\in\mathbb{R}^{r_1\times\cdots\times r_q}$, we define their inner product $\langle \varphi,\nu\rangle$ to be the usual dot product for the tensors interpreted as vectors. 

For a tensor interpreted as a multilinear operator $\varphi(u_1,\ldots,u_q)$, we define the two subordinate norms 
\begin{equation}
\begin{aligned}
\|\varphi\|_\circ &= \max\big\{\|\varphi(\rlap{$\hspace{0.17em}u$}\hphantom{u_2},\ldots,\rlap{$\hspace{0.23em}u$}\hphantom{u_q})\|\ \big|\ \|\rlap{$\hspace{0.23em}u$}\hphantom{u_k}\|_2 = 1\big\}\\
\|\varphi\|_* &= \max\big\{\|\varphi(u_1,\ldots,u_q)\|\ \big|\ \|u_k\|_2 = 1 \text{ for } k\in[q]\big\}
\end{aligned}
\end{equation}
where $\|\cdot\|_2$ is the Euclidean norm for vectors. These are subordinate norms since $\|\varphi(u,\ldots,u)\| \leq \|\varphi\|_\circ\big(\|u\|_2\big)^q$ and $\|\varphi(u_1,\ldots,u_q)\| \leq \|\varphi\|_*\prod_{k=1}^q\|u_k\|_2$. When $\varphi(\cdot,\ldots,\cdot)$ is symmetric in its arguments, then $\|\varphi\|_\circ = \|\varphi\|_*$ \cite{banach1938homogene,bochnak1971polynomials}. 

\subsection{Variational Analysis} 
\label{sec:vaprelim}

Let $\overline{\mathbb{R}} = [-\infty,\infty]$ denote the extended real line. We define $\Gamma(\cdot, \mathcal{S}) : E \rightarrow\overline{\mathbb{R}}$ to be the indicator function 
\begin{equation}
\Gamma(u,\mathcal{S}) = \begin{cases} 0, &\text{if } u \in \mathcal{S}\\
+\infty, &\text{otherwise}\end{cases}
\end{equation}
where $E$ is some Euclidean space that will be clear from the context.

The outer limit of the sequence of sets $C_n$ is defined as
\begin{equation}
\textstyle\limsup_n C_n = \{x : \exists n_k \text{ s.t. } x_{n_k} \rightarrow x \text{ with } x_{n_k}\in C_{n_k}\},
\end{equation}
and the inner limit of the sequence of sets $C_n$ is defined as
\begin{equation}
\textstyle\liminf_n C_n = \{x : \exists x_n \rightarrow x \text{ with } x_n\in C_n\}.
\end{equation}
The outer limit consists of all the cluster points of $C_n$, whereas the inner limit consists of all limit points of $C_n$.  The limit of the sequence of sets $C_n$ exists if the outer and inner limits are equal, and when it exists we use the notation that $\textstyle\lim_n C_n := \limsup_n C_n = \liminf_n C_n$. 

A sequence of extended-real-valued functions $f_n : X\rightarrow\overline{\mathbb{R}}$ is said to epi-converge to $f$ if at each $x\in X$ we have
\begin{equation}
\begin{aligned}
\begin{cases}
\lim\inf_n f_n(x_n)\geq  f(x) & \text{for every sequence } x_n\rightarrow x\\
\lim\sup_n f_n(x_n)\leq f(x) &\text{for some sequence }x_n\rightarrow x
\end{cases}
\end{aligned}
\end{equation}
Epi-convergence is so-named because it is equivalent to set convergence of the epigraphs of $f_n$, meaning that epi-convergence is equivalent to the condition $\lim_n \{(x,\alpha)\in X\times\mathbb{R} : f_n(x) \leq \alpha\} = \{(x,\alpha)\in X\times\mathbb{R} : f(x) \leq \alpha\}$. We use the notation $\elim_n f_n = f$ to denote epi-convergence relative to $X$.

A sequence of extended-real-valued functions $f_n : X\rightarrow\overline{\mathbb{R}}$ is said to converge pointwise to $f$ if at each $x\in X$ we have that $\lim_n f_n(x) = f(x)$. We abbreviate pointwise convergence relative to $X$ using the notation $\lim_n f_n = f$.

\subsection{Specific Distributions}

{We define a multivariate random variable $U\in\mathbb{R}^p$ to be sub-Gaussian with variance parameter $\sigma^2$ if we have that $\mathbb{E}\exp(s\cdot\langle t, U - \mathbb{E}(U)\rangle) \leq \exp(\sigma^2 s^2/2)$ for all $t \in \mathbb{S}^{p-1}$, which is the unit sphere in $p$-dimensions. Thus a sub-Gaussian random variable also satisfies
\begin{equation}
\label{eqn:subgaualt}
\mathbb{E}\exp\big(s\cdot\langle t, U\rangle\big) \leq M\exp\big(\sigma^2 s^2\big)
\end{equation}
for all $t \in \mathbb{S}^{p-1}$, where $M\geq 1$ and $\sigma^2\geq 0$ are constants. We will use (\ref{eqn:subgaualt}) as our primary characterization of a sub-Gaussian distribution. An important implication of this characterization is that
\begin{equation}
\label{eqn:subgaumom}
\mathbb{E}\big(\langle t, U\rangle^{2k}\big) \leq M\sigma^{2k}\cdot(2k)!/k!
\end{equation}
for all $t \in \mathbb{S}^{p-1}$, which can be shown using the bound in (\ref{eqn:subgaualt}).

Sub-Gaussian distributions are ubiquitous. A Gaussian distribution $X$ with mean $\mu$ and variance $\sigma^2$ is denoted $X \sim \mathcal{N}(\mu,\sigma^2)$, a Bernoulli random variable $X$ with success probability $x\in[0,1]$ is denoted $X \sim \mathrm{Ber}(x)$, and a uniform random variable $X$ with support $[a,b]$ is denoted $X \sim \mathrm{Uni}(a,b)$. These are all elementary examples of sub-Gaussian random variables.}
\subsection{Random Sets}

Let $(\mathcal{U}, \mathfrak{F}, \mathbb{P})$ be a complete probability space, where $\mathcal{U}$ is the sample space, $\mathfrak{F}$ is the set of events, and $\mathbb{P}$ is the probability measure.  A map $S : \mathcal{U}\rightarrow\mathcal{F}$ is a random set if $\{u : S(u) \in\mathcal{X}\}\in\mathfrak{F}$ for each $\mathcal{X}$ in the Borel $\sigma$-algebra on $\mathcal{F}$ \cite{molchanov2006}.  Like the usual convention for random variables, we notationally drop the argument for a random set.  

When discussing stochastic convergence of random sets, we denote that a type of limit occurs almost surely by appending ``$\as$'' to the limit notation.  For instance, notation $\aslimsup_n C_n \subseteq C$ denotes $\mathbb{P}(\limsup_n C_n\subseteq C) = 1$, and notation $\asliminf_n C_n \supseteq C$ denotes $\mathbb{P}(\liminf_n C_n\supseteq C) = 1$.

\section{Fair Statistical Decision Problems}
\label{sec:fsdp}

We use the setting of statistical decision problems: Consider the random variables $(X,Y,Z)$ that have a joint distribution $\mathcal{P} \in \mathcal{D}$ where $\mathcal{D}$ is some fixed family of distributions. The interpretation is that $X$ gives descriptive information, $Y$ has information about some target, and $Z$ encodes protected information which we would like to be fair with respect to. We will not explicitly use $Y$ in this paper, but we note that it is implicitly included within other terms that we discuss.

The goal is to construct a function $\delta(\cdot,\cdot)$ called a \emph{decision rule}, which provides a decision $d = \delta(x,z)$. To evaluate the quality of a decision rule $\delta$, we define a \emph{risk function} $R(\delta)$.  (Though it is conventional to define the risk as $R(\mathcal{P}, \delta)$, we assume without loss of generality that the risk is of the form $R(\delta)$ because when the risk is $R(\mathcal{P}, \delta)$ then the proper choice of $R(\delta)$ recovers the Bayes $R(\delta) = \mathbb{E}_\mathcal{P} R(\mathcal{P},\delta)$ and minimax $R(\delta) = \max_{\mathcal{P}\in\mathcal{D}}R(\mathcal{P},\delta)$ procedures.) In this setup, an optimal decision rule is taken to be any function from $\arg\min_{\delta(\cdot,\cdot)} R(\delta)$. However, we can define a related optimization problem that chooses an optimal fair decision rule by solving
\begin{equation}
\label{eqn:ofdr}
\textstyle\delta^*(x,z) \in \arg\min_{\delta(\cdot,\cdot)}\big\{R(\delta)\ \big|\ \delta(X,Z) \perp \!\!\! \perp Z\big\},
\end{equation}
where the notation $\delta(X,Z) \perp \!\!\! \perp Z$ indicates independence of $\delta(X,Z)$ and $Z$.

The above abstract setup is useful because it allows us to reason about fairness for a wide class of problems using a single theoretical framework. This is demonstrated by the following (which is the first to our knowledge) example of a procedure for performing fair hypothesis testing:
\begin{example}
Consider a hypothesis testing setup where the null hypothesis is $H_0: \mathbb{E}(\Xi) = 0$ for the underlying distribution
\begin{equation}
\begin{bmatrix}\Xi \\ \Psi\end{bmatrix} \sim \mathcal{N}\Bigg(\begin{bmatrix}0 \\ 0\end{bmatrix}, \begin{bmatrix}1 & \rho \\ \rho & 1\end{bmatrix}\Bigg).
\end{equation}
Suppose $X = (\Xi_1,\ldots,\Xi_n)$ and $Z = (\Psi_1,\ldots,\Psi_n)$ consist of i.i.d. samples. Let $d_0$ be the decision to accept the null, and let $d_1$ be the decision to reject the null. The traditional hypothesis test with a significance level of $a$ corresponds to a decision rule $\delta$ that minimizes the risk function
\begin{equation}
\label{eqn:htrisk}
R(\delta) = \mathbb{P}_{H_1}(\delta = d_0) + \Gamma(\mathbb{P}_{H_0}(\delta = d_1) - a, \mathbb{R}_{\leq 0}),
\end{equation}
where $H_1 = \{\mathcal{P} \in\mathcal{D} : \mathcal{P} \neq H_0\}$ \cite{lehmann2006testing} {and $\Gamma(\cdot,\cdot)$ is the indicator function that was defined in Section \ref{sec:vaprelim}}. An optimal decision rule for this risk is
\begin{equation}
\label{eqn:drht}
\delta^* = \begin{cases}d_0, & \text{if } p \geq a\\
d_1, &\text{if } p < a\end{cases}
\end{equation}
where $p$ is a $p$-value \cite{lehmann2006testing}. An optimal decision rule that depends only upon $X$ corresponds to the use of a traditional $p$-value
\begin{equation}
\label{eqn:oldp}
\textstyle p = 2\Phi\Big(-\sqrt{n}\big|\frac{1}{n}\sum_{i=1}^n\Xi_i\big|\Big),
\end{equation}
with $\Phi(\cdot)$ being the standard normal c.d.f. Using the above framework, we can compute an optimal \emph{fair} decision rule for this risk. This corresponds to
\begin{equation}
\label{eqn:fairp}
\textstyle p = 2\Phi\Big(-\sqrt{\frac{n}{1-\rho^2}}\big|\frac{1}{n}\sum_{i=1}^n\big(\Xi_i - \rho\Psi_i\big)\big|\Big),
\end{equation}
which we can interpret as a \emph{fair} $p$-value. An interesting observation about this setup is that using (\ref{eqn:fairp}) results in a test with greater \emph{power} than using (\ref{eqn:oldp}).  {This means that the risk as measured by (\ref{eqn:htrisk}) of the decision rule (\ref{eqn:drht}) with (\ref{eqn:oldp}) is higher than the risk of the decision rule (\ref{eqn:drht}) with (\ref{eqn:fairp}). This example is interesting because it shows that using more variables, even protected ones, can improve the resulting decision rule by reducing its risk.}%This is interesting because it shows that imposing fairness constraints can lead to better statistical procedures in certain contexts. 
\end{example}

%In this case, the Bayes and minimax procedures conceptually lead to the corresponding usual setups. 
In many statistical contexts, $\mathcal{D}$ is singleton but unknown. We then instead choose the decision rule using a sample $(X_i,Y_i,Z_i)$ for $i=1,\ldots,n$, which is i.i.d. from the distribution $\mathcal{P}$. Towards this aim, we approximate the risk function $R(\delta)$ using an (random) approximate risk function $R_n(\delta)$ that depends upon the sample. However, computing a sample-based fair decision rule is not obvious because a statistically well-behaved, sample-based analog of the constraint $\delta(X,Z) \perp \!\!\! \perp Z$ from (\ref{eqn:ofdr}) has not been studied previously.

\section{Fair Optimization Hierarchy}

\label{sec:foh}

We next propose a framework for computing a fair decision rule by solving a sample-based analog of (\ref{eqn:ofdr}). We first describe our assumptions about the statistical and numerical properties of the problem. Next we present our framework and provide some intuition to justify the structure of our formulation. We conclude by discussing some of the favorable computational properties of our framework.

\subsection{Assumptions}

We first make some assumptions about our decision rule and random variables:
%In this paper, we propose the use of the following optimization problem 

\begin{assumption}
\label{ass:drule}
The decision rule belongs to a parametric polynomial family and can be written as
\begin{equation}
\delta(x,z) = B\cdot\omega(x,z),
\end{equation}
where $B\in\mathcal{B}$ is a matrix, $\mathcal{B}\subset\mathbb{R}^{d\times p}$ is a compact set, and $\omega(x,z) \in\mathbb{R}^p$ is a vector of monomials of the entries of the vectors $x,z$. More precisely, $B$ parametrizes the decision rule $\delta(x,z)$, and the function $\omega(x,z)$ is assumed to be known and fixed by our design such as through feature engineering. We define the random variable $\Omega = \omega(X,Z)$, so that $\delta(X,Z) = B\Omega$. 
\end{assumption}

\begin{remark}
In some settings, it may be desirable to have the fair decision rule depend upon only $X$ and not $Z$. The above includes this case by noting $\omega(x,z)$ is free to be chosen to include only monomials of the entries of $x$.
\end{remark}

\begin{remark}
{This assumption says the decision rules are linear with respect to some polynomial transformation of the $X$ and $Z$. Such a linear decision rule may not be competitive in terms of risk minimization as compared to more sophisticated models, but linear decision rules are commonly used in many application domains such as health care or economics and as such are important to theoretically study in the setting of fairness.}
\end{remark}

\begin{assumption}
\label{ass:2norm}
%Recall $B\in\mathcal{B}$ is a matrix that parametrizes the decision rule, for compact $\mathcal{B}\subset\mathbb{R}^{d\times p}$. We assume $\mathcal{B} \subseteq \{B \in \mathbb{R}^{d\times p} : \|W(B)\|_2 \leq \sqrt{\lambda}\}$.
Assume $\mathcal{B} \subseteq \{B \in \mathbb{R}^{d\times p} : \|W(B)\|_2 \leq \sqrt{\lambda}\}$ for $\lambda \geq 1$.
\end{assumption}

Our next assumption is about statistical properties of the approximate risk function. Since our primary interest in this paper is studying independence constraints, we directly make assumptions about the convergence of the approximate risk function. Showing that such convergence holds typically involves a separate statistical analysis specific to the problem at hand.

\begin{assumption}
\label{ass:convergence}
Note the function $R_n(B\cdot\omega(x,z))$ is the approximate risk function composed with the parametric decision rule in Assumption \ref{ass:drule}. We assume that this function can be written in the form
\begin{equation}
\label{eqn:objn}
h_n(B) := R_n(B\cdot\omega(x,z)) = f_n(B) + \Gamma(g_n(B), \{\mathbb{R}_{\leq 0}\}^\eta),
\end{equation}
where $f_n : \mathbb{R}^{d\times p}\rightarrow\mathbb{R}$ and $g_n : \mathbb{R}^{d\times p}\rightarrow\mathbb{R}^\eta$. Moreover, define the notation $h(B) = R(B\cdot\omega(x,z))$. We assume $\aselim h_n = \aslim h_n = h$ relative to $\mathcal{B}$.
%The function $h_n(B)$ epi-converges almost surely and converges pointwise almost surely relative to the set $\mathcal{B}$. More specifically, we assume that $\aselim R_n = R$ and $\aslim R_n = R$ relative to $\mathcal{B}$.
\end{assumption}

\begin{remark}
We should interpret the notation of (\ref{eqn:objn}) as simultaneously specifying an objective function $f_n(B)$ and a set of constraints $g_n(B) \leq 0$.
\end{remark}

\begin{remark}
This convergence assumption may look unfamiliar, but we note that it is weaker than the convergence results that are usually shown when proving consistency of estimators. In particular, almost sure uniform convergence of $h_n$ to $h$ implies the above assumption.
\end{remark}

The first three assumptions are primarily related to statistical properties. {It is instructive to consider examples that show how linear regression and linear classification problems match the assumptions above.}

\begin{example}
{Linear regression with $(X_i, Y_i) \in \mathbb{R}^p\times\mathbb{R}$ in our setup would mean we choose a linear decision rule $\delta(x) = Bx$ with $B \in \mathbb{R}^{1\times p}$. We could use a squared loss $R_n(B\cdot x) = \frac{1}{n}\sum_{i=1}^n(Y_i - BX_i)^2$ or the least absolute deviation loss $R_n(B\cdot x) = \frac{1}{n}\sum_{i=1}^n|Y_i-BX_i|$ for our regression. The nondifferentiability of the latter can be managed by introducing the variables $s_i$ and noting $R_n(B\cdot x) = \frac{1}{n}\sum_{i=1}^ns_i$ subject to the constraints $-s_i \leq Y_i-BX_i \leq s_i$. This matches the decomposition (\ref{eqn:objn}) of $R_n(\delta)$ into an objective with constraints. These loss functions can hence be minimized by many algorithms.}
\end{example}

\begin{example}
{Linear classification with $(X_i, Y_i) \in \mathbb{R}^p\times\{-1,+1\}$ in our setup would mean we choose a linear decision rule $\delta(x) = Bx$ with $B \in \mathbb{R}^{1\times p}$. We could use any classification-calibrated loss: Logistic regression uses $R_n(B\cdot x) = \frac{1}{n}\sum_{i=1}^n\log(1 + \exp(-Y_i\cdot BX_i))$. Because this logistic loss is convex and differentiable, it can be easily optimized. Support vector machine uses the hinge loss $R_n(B\cdot x) = \frac{1}{n}\sum_{i=1}^n\max\{0, 1-Y_i\cdot BX_i\}$. Its nondifferentiability is handled by introducing the variables $s_i$ and noting $R_n(B\cdot x) = \frac{1}{n}\sum_{i=1}^ns_i$ subject to constraints $s_i \geq 0$ and $s_i \geq 1-Y_i\cdot BX_i$. This matches the decomposition (\ref{eqn:objn}) of $R_n(\delta)$ into an objective with constraints, and this formulation can be easily minimized by many algorithms.}

{In each case, the linear classifier makes binary predictions $\widehat{Y}(x,t) = \mathrm{sign}(t-Bx)$ by applying a threshold $t$ to the decision rule $\delta(x) = Bx$. This interpretation of using the score function $Bx$ as the decision rule is theoretically justified because classification-calibrated losses (like the logistic loss or the hinge loss) composed with the score function are statistically consistent with respect to the 0-1 classification loss composed with the thresholded binary predictions $\widehat{Y}(x,0)$ \cite{bartlett2006convexity}, and because statistical independence of $\delta(x) = Bx$ and $Z$ implies independence between $\widehat{Y}(X,t)$ and $Z$.}
\end{example}

\begin{example}
\label{ex:01loss}
{We could consider the above linear classification setup using the 0-1 classification loss $R_n(B\cdot x) = \frac{1}{n}\sum_{i=1}^n H(-Y_i\cdot BX_i)$, where $H(\cdot) : \mathbb{R} \rightarrow\{0, 1\}$ is the step function defined as
\begin{equation}
\mathbf{1}(u) = \begin{cases} 0, &\text{if } u \leq 0\\
1, &\text{otherwise}\end{cases}.
\end{equation}
This loss is supported by our setup because Assumption \ref{ass:convergence} follows by applying standard uniform convergence results \cite{wainwright2017high}. (Uniform convergence is technically stronger than the type of convergence required in Assumption \ref{ass:convergence}.)  However, the resulting optimization problem is an integer program \cite{Liittschwager1978}. The idea behind the integer programming formulation is that it uses binary variables to keep track of whether or not each $Y_i\cdot BX_i$ is nonnegative.}
\end{example}

\subsection{Formulation}

We are now ready to present our framework. Given the above assumptions, we study use of the following sample-based optimal fair decision rule: The level-$(\mathfrak{g},\mathfrak{h})$ fair optimization (FO) is
\begin{equation}\label{eq:fo}
\begin{aligned}
\min_{B\in\mathcal{B}}\ &R_n(B\cdot\omega(x,z))\\
%\text{s.t. }&\textstyle\big\|\mathbb{E}_n\big(Z_i^{\otimes m}(B\Omega_i^{\vphantom{\otimes m}})^{\otimes q}\big)-\mathbb{E}_n\big(Z_i^{\otimes m}\big)\otimes\mathbb{E}_n\big((B\Omega_i)^{\otimes n}\big)\|\le\Delta_{n,m},\\
\text{s.t. }&\textstyle\big\|\mathbb{E}_n\big(Z^{\otimes m}(B\Omega)^{\otimes q}\big)-\mathbb{E}_n\big(Z^{\otimes m}\big)\otimes\mathbb{E}_n\big((B\Omega)^{\otimes q}\big)\big\|\le\Delta_{m,q},\\
&\qquad\text{for } (m,q)\in[\mathfrak{g}]\times[\mathfrak{h}]
%&\qquad\qquad\qquad\qquad\qquad\qquad\qquad\qquad \text{for } 1\le m\le k_1,1\le q\le k_2.
%\text{s.t. }&\textstyle\big\|\frac{1}{n}\sum_{i=1}^{n}(Z_i)^{\otimes m}(B\Omega_i)^{\otimes q}+\\
%\text{s.t. }&\textstyle\big\|\frac{1}{n}\sum_{i=1}^{n}(Z_i)^{\otimes m}(B\Omega_i)^{\otimes q}-\frac{1}{n}\sum_{i=1}^{n}(Z_i)^{\otimes m}\otimes\frac{1}{n}\sum_{i=1}^{n}(B\Omega_i)^{\otimes n}\big\|\le\Delta_{n,m},\\
%&\qquad\qquad\qquad\qquad \text{for } 1\le m\le k_1,1\le q\le k_2.
%\text{s.t. }&\textstyle\big\|\frac{1}{n}\sum_{i=1}^{n}(Z_i)^{\otimes m}(B\Omega_i)^{\otimes q}+\\
%&\qquad\qquad\textstyle-\frac{1}{n}\sum_{i=1}^{n}(Z_i)^{\otimes m}\otimes\frac{1}{n}\sum_{i=1}^{n}(B\Omega_i)^{\otimes n}\big\|\le\Delta_{n,m},\\
%&\qquad\qquad\qquad\qquad \text{for } 1\le m\le k_1,1\le q\le k_2.
\end{aligned}
\end{equation}
{where $\mathfrak{g},\mathfrak{h}\geq 1$ are integers and $\Delta_{m,q} \geq 0$ are nonnegative real numbers. We note that $\mathfrak{g}$, $\mathfrak{h}$, and $\Delta_{m,q}$ will generally be chosen to depend on $n$, but for simplicity we will not make this $n$-dependence explicit in our notation. Our optimization hierarchy for fair statistical decision problems is defined by the above formulation given in (\ref{eq:fo}), with the increasing number of constraints in the hierarchy parametrized by increasing values of $\mathfrak{g},\mathfrak{h}$.} We will study the constraints of the above problem and show that they are statistically well-behaved analogs of the independence constraint in (\ref{eqn:ofdr}).
%\begin{equation}
%\widehat{\delta}(X) \in\arg\min R_n(\delta)$.%(Strictly speaking, the approximation is $R_n(\delta, X_1,Y_1,Z_1,\ldots,X_n,Y_n,Z_n)$.)
%\end{equation}

\begin{remark}
{
The above formulation considers fairness in the sense of disparate impact. When the protected attributes are categorical, meaning $Z \in \mathcal{Z}$ for some finite-cardinality set $\mathcal{Z}$, then our formulation can be modified to consider fairness in the sense of equalized odds by replacing the constraints in the above formulation with the constraints
%The above formulation considers fairness in terms of disparate impact. When the protected attributes are categorical, meaning $Z \in \mathcal{Z}$ for some finite set $\mathcal{Z}$, then our formulation can be modified to consider fairness in terms of equalized odds by replacing the constraints in the above with
\begin{multline}
\textstyle\big\|\mathbb{E}_n\big[Z^{\otimes m}(B\Omega)^{\otimes q}|Z = z\big]-\mathbb{E}_n\big[Z^{\otimes m}|Z=z\big]\otimes\mathbb{E}_n\big[(B\Omega)^{\otimes q}|Z=z\big]\big\|\\
\le\Delta_{m,q}, \text{ for } (m,q)\in[\mathfrak{g}]\times[\mathfrak{h}] \text{ and } z \in \mathcal{Z}.
\end{multline}
Compared to the above formulation, here we take expectations with respect to the empirical distribution conditioned on each possible value in $\mathcal{Z}$.}
\end{remark}

Our first result provides intuition about the constraints in the FO optimization problem (\ref{eq:fo}). This result generalizes Kac's theorem \cite{bisgaard2006does,kac1936fonctions}, which characterizes independence of random variables using moment conditions, to the setting of random vectors. This generalization is novel to the best of our knowledge, and so we include its proof below for the sake of completeness.

\begin{theorem}
\label{thm:kac}
{Let $M_{(U,V)}(s,t) = \mathbb{E}\exp(\langle s,U\rangle + \langle t,V\rangle)$ be the moment generating function for the multivariate random variable $(U,V)$ where we have $U\in\mathbb{R}^p$ and $V\in\mathbb{R}^d$. If $M_{(U,V)}(s,t)$ is finite in a neighborhood of the origin, then} $U$ and $V$ are independent if and only if 
\begin{equation}
\label{eqn:mc}
\mathbb{E}\big(U^{\otimes m}V^{\otimes q}\big) = \mathbb{E}\big(U^{\otimes m}\big)\otimes\mathbb{E}\big(V^{\otimes q}\big)\ \mathrm{for}\ m,q\geq 1.
\end{equation}
\end{theorem}

\begin{proof}
{Let $M_U(s) = \mathbb{E}\exp(\langle s, U\rangle)$ and $M_V(t) = \mathbb{E}\exp(\langle t,V\rangle)$ be the moment generating functions for $U$ and $V$, respectively. Observe that these are defined for $s$ and $t$ in a neighborhood of the origin by the assumption in the hypothesis on $M_{(U,V)}(s,t)$.} Our proof begins with the well-known characterization of independence using moment generating functions, that is $U$ and $V$ are independent if and only if $M_{(U,V)}(s,t) = M_U(s)M_V(t)$. In particular, if (\ref{eqn:mc}) holds then we have
\begin{equation}
\label{eqn:pfkacth}
\begin{aligned}
M_{(U,V)}(s,t) &= \textstyle\sum_{m=0}^\infty\sum_{q=0}^\infty\frac{1}{m!\cdot q!}\cdot\mathbb{E}\big(\langle s,U\rangle^m\langle t,V\rangle^q\big)\\
&\textstyle= \sum_{m=0}^\infty\sum_{q=0}^\infty\frac{1}{m!\cdot q!}\cdot\langle\mathbb{E}\big(U^{\otimes m}V^{\otimes q}\big), s^{\otimes m}t^{\otimes q}\rangle\\
&\textstyle= \sum_{m=0}^\infty\sum_{q=0}^\infty\frac{1}{m!\cdot q!}\cdot\langle\mathbb{E}\big(U^{\otimes m}\big)\otimes\mathbb{E}\big(V^{\otimes q}\big), s^{\otimes m}t^{\otimes q}\rangle\\
&\textstyle= \sum_{m=0}^\infty\sum_{q=0}^\infty\frac{1}{m!\cdot q!}\cdot\langle\mathbb{E}\big(U^{\otimes m}\big), s^{\otimes m}\rangle\cdot\langle\mathbb{E}\big(V^{\otimes q}\big), t^{\otimes q}\rangle\\
&\textstyle= \sum_{m=0}^\infty\sum_{q=0}^\infty\frac{1}{m!\cdot q!}\cdot\mathbb{E}\big(\langle s,U\rangle^m\big)\cdot\mathbb{E}\big(\langle t,V\rangle^q\big)\\
&\textstyle= \sum_{m=0}^\infty\frac{1}{m!}\cdot\mathbb{E}\big(\langle s,U\rangle^m\big)\cdot\sum_{q=0}^\infty\frac{1}{q!}\cdot\big(\mathbb{E}\langle t,V\rangle^q\big) \\
&= M_U(s)M_V(t)
\end{aligned}
\end{equation}
This proves the reverse direction. {To prove the forward direction, we note it follows by applying componentwise for all $\sigma \in [p]^m$ and $\tau\in[d]^q$  the standard result that if $U$ and $V$ are independent, then $\mathbb{E}(\prod_{k=1}^m U_{\sigma_k}\cdot\prod_{k=1}^q V_{\tau_k}) = \mathbb{E}(\prod_{k=1}^mU_{\sigma_k})\cdot\mathbb{E}(\prod_{k=1}^q V_{\tau_k})$ when these expectations exist. Indeed, these expectations exist because of the hypothesis assumption on $M_{(U,V)}(s,t)$.}
\end{proof}

\begin{remark}
{This result requires that $M_{(U,V)}(s,t)$ exists in a neighborhood of the origin. Examples of distributions that satisfy this condition are those with a bounded support (almost surely), as well as those belonging to the sub-Gaussian, sub-exponential, or sub-gamma families of distributions. This encompasses a large number of the most common distributions.}
\end{remark}

{Next, we show a similar result that characterizes approximate independence of random variables using moment conditions. The benefits of this next result are that: it holds for (possibly unbounded) distributions that have finite moments, and it does not require the existence of $M_{(U,V)}(s,t)$ in a neighborhood of the origin. This means it applies to a larger class of distributions. Our characterization relating moment conditions to approximate independence is the first result of its kind, to our knowledge.}

%{However, we have to first specify how independence will be quantified. A natural idea is to consider a distance between the joint distribution of $(Z,\widehat{B}_n\Omega)$ and the product distribution of $Z$ and $\widehat{B}_n\Omega$. This idea is natural because when we have independence then it must be the case that the joint distribution equals the product distribution. Thus the pertinent detail is choosing a distance between distributions to use. Our next example shows a subtle issue in making this choice.}

{However, we have to specify how independence is quantified. A natural idea is to consider a distance between the joint distribution of $(Z,\widehat{B}_n\Omega)$ and the product distribution of $Z$ and $\widehat{B}_n\Omega$. This idea is natural because independence means that the joint distribution equals the product distribution. Thus the pertinent detail is choosing a distance between distributions to use. Our next example shows a subtle issue in making this choice.}

\begin{example}
{Consider a setting where $B\in\mathbb{R}$, where $\omega(x,z) = x$, and the distributions are $X \sim\mathrm{Uni}(-1,1)$ and $Z = X$. Then $\Omega = X$. Next let
\begin{equation}
\bbd(B) = \sup_{s,t} \big|\mathbb{P}_{(Z,B\Omega)}(Z \leq s, B\Omega \leq t) - \mathbb{P}_{\vphantom{(Z,B\Omega)}Z}(Z \leq s) \cdot \mathbb{P}_{\vphantom{(Z,B\Omega)}B\Omega}(B\Omega \leq t)\big|
\end{equation}
be the multivariate Kolmogorov-Smirnov distance between the joint and product distributions of $Z$ and $B\Omega$. Now note $\bbd(0) = 0$ because $Z$ is trivially independent of the constant $0\cdot\Omega \equiv 0$. Next observe that for any $B \neq 0$ we have $\bbd(B) = \bbd(1)$, but $\bbd(1) > 0$ since $Z = \Omega$. Hence for the sequence $B_n = n^{-1}$, we have that $B_n\Omega$ is asymptotically independent of $Z$ but that $\bbd(\lim_n B_n) = \bbd(0) = 0 \neq \lim_n \bbd(B_n) = \bbd(1) > 0$ . This means the multivariate Kolmogorov-Smirnov distance cannot quantify independence here.}
\end{example}

\begin{remark}
{Because the total variation distance is greater than or equal to the value of the multivariate Kolmogorov-Smirnov distance, the above example also applies to the total variation distance. Thus Pinsker's inequality implies the above example applies to the Kullback–Leibler (KL) divergence. This means the above example also applies to mutual information, which is defined as the KL divergence between the joint and product distributions.}
\end{remark}

\begin{remark}
{A multivariate version of this example can be constructed where the same issue occurs for a $B \neq 0$, where the example is constructed such that the issue occurs because $B$ does not have full column rank.}
\end{remark}

{The above examples show that several popular distances between distributions cannot be used for quantifying the degree of independence in our setting of fair optimization. This is perhaps not surprising given that the notion of convergence in distribution is weaker than many popular distances. Consequently, we need to consider topologically-weaker metrics on probability distributions, that are able to metricize convergence in distribution. 

One such distance is the Zolotarev metric defined using characteristic functions \cite{Zolotarev_1976,klebanov1984estimate,rachev2013methods}, and we will use this distance to quantify the degree of independence between two random variables. Let $U\in\mathbb{R}^p$ and $V\in\mathbb{R}^d$ be random vectors, and define $\mathfrak{i} = \sqrt{-1}$. Then for $s \in \mathbb{R}^{p}$, $t\in\mathbb{R}^{d}$, and $\zeta\in\mathbb{R}$; let $J(s,t,\zeta) = \mathbb{E}\exp(\mathfrak{i}\zeta\langle s, U\rangle + \mathfrak{i}\zeta\langle t, V\rangle)$ and $P(s,t,\zeta) = \mathbb{E}\exp(\mathfrak{i}\zeta\langle s, U\rangle)\cdot\mathbb{E}\exp(\mathfrak{i}\zeta\langle t, V\rangle)$ be the characteristic functions corresponding to the joint and product distributions, respectively, of $U$ and $V$. The Zolotarev metric between the joint and product distributions is given by
\begin{equation}
\label{eqn:mdef}
\mathbb{H}(U;V) = \sup_{(s,t)\in\mathbb{S}^{p+d-1}}\Bigg[\mathop{\mathrm{inf}\vphantom{\mathrm{sup}}}_{\vphantom{|\zeta|}T > 0}\max\Big\{\frac{1}{2}\sup_{|\zeta| \leq T}\big|J(s,t,\zeta) - P(s,t,\zeta)\big|, \frac{1}{T}\Big\}\Bigg].
\end{equation}
We call the quantity $\mathbb{H}(U;V)$ the \emph{mutual characteristic} of $U$ and $V$, and the choice of this name is meant to draw a direct analogy to mutual information.}

\begin{theorem}
\label{thm:kac2}
{Consider the random variable $(U,V)$ where $U\in\mathbb{R}^p$ and $V\in\mathbb{R}^d$. If $J_{\mathfrak{g},\mathfrak{h}} = \sup_{(s,t)\in\mathbb{S}^{p+d-1}}\mathbb{E}(\langle s, U\rangle^{\mathfrak{g}+1}\langle t,V\rangle^{\mathfrak{h}+1})$ is finite and 
\begin{equation}
\label{eqn:mc2}
\mathbb{E}\big(U^{\otimes m}V^{\otimes q}\big) = \mathbb{E}\big(U^{\otimes m}\big)\otimes\mathbb{E}\big(V^{\otimes q}\big)\ \mathrm{for}\ m,q\in[\mathfrak{g}]\times[\mathfrak{h}],
\end{equation}
then we have that
\begin{equation}
\label{eqn:kacapdi}
\mathbb{H}(U;V) \leq \textstyle\Big[\frac{J_{\mathfrak{g},\mathfrak{h}} + P_{\mathfrak{g},\mathfrak{h}}}{(\mathfrak{g}+1)!\cdot(\mathfrak{h}+1)!}\Big]^{1/(\mathfrak{g}+\mathfrak{h}+3)}
\end{equation}
where $P_{\mathfrak{g},\mathfrak{h}} = \sup_{s\in\mathbb{S}^{p-1}}\mathbb{E}(\langle s, U\rangle^{\mathfrak{g}+1})\cdot \sup_{t\in\mathbb{S}^{d-1}}\mathbb{E}(\langle t, V\rangle^{\mathfrak{h}+1})$.}
\end{theorem}

\begin{proof}
{We need to bound the modulus of $J(s,t,\zeta) - P(s,t,\zeta)$. As a first step, note that the difference of their Taylor polynomials satisfies
\begin{multline}
\textstyle\sum_{m=0}^\mathfrak{g}\sum_{q=0}^\mathfrak{h}\frac{1}{m!\cdot q!}\cdot\mathbb{E}\big(\langle s,U\rangle^m\langle t,V\rangle^q\big) +\\
\textstyle- \sum_{m=0}^\mathfrak{g}\frac{1}{m!}\cdot\mathbb{E}\big(\langle s,U\rangle^m\big)\cdot\sum_{q=0}^\mathfrak{h}\frac{1}{q!}\cdot\big(\mathbb{E}\langle t,V\rangle^q\big) = 0
\end{multline}
by the same reasoning used to show (\ref{eqn:pfkacth}). We note that the above summation is well-defined because of the finiteness assumption on $J_{\mathfrak{g},\mathfrak{h}}$ in the hypothesis of this theorem. Next we apply a standard argument (see for instance Section 26 of \cite{billingsley1995probability}) that first uses Jensen's inequality and then uses the elementary inequality $|\exp(i\zeta) - \sum_{m=0}^\mathfrak{g}(i\zeta)^m/m!| \leq |\zeta|^{\mathfrak{g}+1}/(\mathfrak{g}+1)!$ for the complex exponential. This argument implies that for $|\zeta| \leq T$ we have
\begin{equation}
\big|J(s,t,\zeta) - P(s,t,\zeta)\big| \leq\textstyle\frac{J_{\mathfrak{g},\mathfrak{h}} + P_{\mathfrak{g},\mathfrak{h}}}{(\mathfrak{g}+1)!\cdot(\mathfrak{h}+1)!}\cdot T^{\mathfrak{g}+\mathfrak{h}+2}.
\end{equation}
If we choose $T^{\mathfrak{g}+\mathfrak{h}+3}=(\mathfrak{g}+1)!\cdot(\mathfrak{h}+1)!/(J_{\mathfrak{g},\mathfrak{h}} + P_{\mathfrak{g},\mathfrak{h}})$, then the result follows by applying this bound to the definition (\ref{eqn:mdef}).}
\end{proof}

These two generalizations of Kac's theorem allow us to interpret the constraints of the FO problem (\ref{eq:fo}). {Using Theorem \ref{thm:kac}, we can interpret the constraints as a finite number ($\mathfrak{g}\cdot\mathfrak{h}$ many, for a level-$(\mathfrak{g},\mathfrak{h})$ FO problem) of sample-based analogs of the corresponding moment conditions (\ref{eqn:mc}) for independence. Using Theorem \ref{thm:kac2}, we can also interpret the constraints as sample-based analogs of the corresponding finite number of moment conditions that achieve approximate independence in the sense of (\ref{eqn:kacapdi}).}

\subsection{Computational Properties}

We next discuss some favorable computational properties of the FO problem (\ref{eq:fo}). {A key advantage of our framework is that the moment constraints are polynomials. This leads to three general approaches that can be used to numerically solve the FO problem. 

The first approach when the relevant functions are polynomials, which allow us to draw upon powerful tools for polynomial optimization \cite{lasserre2010moments}:}
%
%\begin{assumption}
%\label{ass:cont}
%In the notation of Assumption \ref{ass:convergence}, we assume that the functions $f_n : \mathbb{R}^{d\times p}\rightarrow\mathbb{R}$ and $g_n : \mathbb{R}^{d\times p}\rightarrow\mathbb{R}^\eta$ are polynomials on the set $\mathcal{B}$. We also assume that $h$ is a lower semicontinuous function on the set $\mathcal{B}$.
%\end{assumption}
%
%\begin{remark}
%The polynomial assumption is not restrictive because the celebrated Stone-Weierstrass theorem shows that if $f_n$ and $g_n$ are continuous then they can be approximated to arbitrary accuracy by polynomials, since the domain of the optimization problem is within a compact set $\mathcal{B}$.
%\end{remark}
%

%A key advantage of our framework is that the moment constraints are polynomials, and so we can leverage significant numerical and theoretical advances in order to solve such problems.

\begin{theorem}[Theorems 5.6, 5.7 of \cite{lasserre2010moments}]
Suppose Assumptions \ref{ass:drule}--\ref{ass:convergence} hold. {If, in the notation of Assumption \ref{ass:convergence}, we assume that the functions $f_n : \mathbb{R}^{d\times p}\rightarrow\mathbb{R}$ and $g_n : \mathbb{R}^{d\times p}\rightarrow\mathbb{R}^\eta$ are polynomials on the set $\mathcal{B}$}, then the level-$(\mathfrak{g},\mathfrak{h})$ FO problem (\ref{eq:fo}) can be solved to any desired accuracy by solving a convex optimization problem that can be explicitly constructed.
\end{theorem}

\begin{remark}
{The polynomial assumption is not restrictive because the celebrated Stone-Weierstrass theorem shows that if $f_n$ and $g_n$ are continuous then they can be approximated to arbitrary accuracy by polynomials, since the domain of the optimization problem is within a compact set $\mathcal{B}$. This means this approach can be used in principle for the squared loss, logistic loss, hinge loss (after the earlier reformulation), least absolute deviation loss (after the earlier reformulation), and many other functions.}
\end{remark}

{Though the convex optimization problems resulting from the explicit construction of \cite{lasserre2010moments} are often large, these resulting optimization problems can be numerically solved for many interesting instances \cite{majumdar2014convex,zhao2019optimal}. We briefly discuss the intuition behind this approach. The first insight is that any polynomial optimization problem $\min \big\{f_n(B)\ \big|\ g_n(B) \leq 0, B\in\mathcal{B}\}$ can be written as maximizing a scalar subject to nonnegative polynomial constraints
\begin{equation}
\max \big\{s\ \big|\ f_n(B) - s \geq 0, -g_n(B) \geq 0, s\in\mathbb{R}, B\in\mathcal{B}\big\}.
\end{equation}
The second insight is that nonnegative polynomials can be approximated on a bounded domain to arbitrary accuracy using sum-of-squares (SOS) polynomials \cite{berg1987multidimensional,lasserre2010moments}. Since our problems involve optimizing a vector that belongs to Euclidean space, SOS polynomials are literally the set of polynomials that are generated by squaring arbitrary polynomials and then adding them up. Specifically, the nonnegative polynomial constraints can be approximated by instead asking for the polynomials to equal a linear combination of a finite number of SOS polynomials. This is a tractable approximation because the resulting optimization problem is a convex semidefinite program, and the following solution can be made arbitrarily accurate by increasing the finite number of SOS polynomials used in the approximation.}

{The second approach applies to cases where the relevant functions are differentiable (but not necessarily polynomial), which allow us to use standard optimization algorithms. Specifically, the moment constraints for low levels of our FO hierarchy have structures that enable numerical solution using algorithms like the constrained convex-concave procedure \cite{smola2005kernel,tuy1995dc,yuille2002concave}.} We can say more about the FO problem for specific levels of the hierarchy, and we omit the proofs since they follow from the definition of the constraint:
%We can say more about the FO problem for specific levels of the hierarchy. We start with a result concerning the constraints in (\ref{eq:fo}) for $q = 1$.

\begin{proposition}
The constraints in the FO problem (\ref{eq:fo}) for $q = 1$ can be written as the following linear inequality constraints:
\begin{equation}
\begin{aligned}
\textstyle B\Big(\frac{1}{n}\sum_{i=1}^n\Omega_i\otimes(Z_i)^{\otimes m} - \frac{1}{n}\sum_{i=1}^n\Omega_i\otimes \frac{1}{n}\sum_{i=1}^n(Z_i)^{\otimes m}\Big)\leq &\Delta_{m,1}\\
\textstyle-B\Big(\frac{1}{n}\sum_{i=1}^n\Omega_i\otimes(Z_i)^{\otimes m} - \frac{1}{n}\sum_{i=1}^n\Omega_i\otimes \frac{1}{n}\sum_{i=1}^n(Z_i)^{\otimes m}\Big)\leq &\Delta_{m,1}\\
\end{aligned}
\end{equation}
where the inequality should be interpreted as being elementwise of the left (which is a tensor) with respect to the scalar $\Delta_{m,1}$ on the right.
\end{proposition}

This results says constraints with $q=1$ are always convex. This means that the FO problem (\ref{eq:fo}) with $\mathfrak{h} = 1$ is a convex optimization problem whenever $R_n$ is convex in $B$. Such convexity of $R_n$ occurs in many interesting problems, including linear regression and support vector machines.

%Our next result concerns constraints corresponding to $q = 2$. We omit the proof since it follows immediately from the definition of the constraint.

\begin{proposition}
%The constraints in the FO problem (\ref{eq:fo}) for $q = 2$ are a set of inequalities that each involve a difference of two quadratic functions.
The constraints in the FO problem (\ref{eq:fo}) for $q = 2$ are inequalities that each involve a difference of two convex quadratic functions.
\end{proposition}

This results says constraints with $q=2$ are always a difference of convex functions. This means that stationary points of the FO problem (\ref{eq:fo}) with $\mathfrak{h} = 2$ can be found using the effective constrained convex-concave procedure \cite{smola2005kernel,tuy1995dc,yuille2002concave} whenever $R_n$ is convex in $B$. Recall that $R_n$ is convex in many interesting problems like linear regression and support vector machines.

%Our last result concerns settings where $Z$ is binary, that is coded as either $Z\in\{0,1\}$ or $Z\in\{\pm\}$.

\begin{proposition}
\label{prop:zbinary}
If $Z$ is a binary random variable, which is coded as either $Z\in\{0,1\}$ or $Z\in\{\pm 1\}$, then the constraints in the FO problem (\ref{eq:fo}) for $m \geq 2$ are redundant with the corresponding constraint for $m =1$.
\end{proposition}

This result says that when $Z$ is binary, then the hierarchy simplifies and we only need to consider applying the level-$(1,\mathfrak{h})$ FO problems. We will use this simplification when conducting numerical experiments in Section \ref{sec:er}.
%\begin{proof}
%This follows directly from the definition of the constraint.
%\end{proof}

{The third approach applies when the relevant functions are mixed-integer non-convex quadratic-representable, which means the objective and constraints can be represented by non-convex quadratic functions with some variables constrained to be integer-valued. As described in Example \ref{ex:01loss}, this case holds for linear classification using the 0-1 classification loss.}

\begin{proposition}
{Suppose Assumptions \ref{ass:drule}--\ref{ass:convergence} hold. If, in the notation of Assumption \ref{ass:convergence}, we assume that the functions $f_n : \mathbb{R}^{d\times p}\rightarrow\mathbb{R}$ and $g_n : \mathbb{R}^{d\times p}\rightarrow\mathbb{R}^\eta$ are mixed-integer non-convex quadratic-representable, then the level-$(\mathfrak{g},\mathfrak{h})$ FO problem (\ref{eq:fo}) can be solved using a non-convex mixed-integer quadratically constrained program (non-convex MIQCP).}
\end{proposition}

{The development of numerical algorithms to solve non-convex MIQCP problems is an active research area \cite{burer2012non,kilincc2015two,chen2017spatial}, and a number of software packages \cite{sahinidis1996baron,adjiman2000global,burer2009copositive,lin2009global,vigerske2018scip,gurobi} are already available for solving such problems. The proof of the above result is omitted because it follows immediately from the facts that the moment constraints are polynomials and that any polynomial inequality constraint can be represented by quadratic constraints and a set of new variables. To understand the intuition behind this second fact, consider as an example the constraint $B_1^{\ 3} \leq 0$. We can represent this by two constraints $B_2^{\vphantom{2}} = B_1^{\ 2}$ and $B_1\cdot B_2 \leq 0$, where we have introduced a new variable $B_2$. These two constraints are non-convex quadratic constraints.}

\section{Statistical Consistency of FO Hierarchy}

\label{sec:scfoh}

We prove in this section that the sample-based constraints of the FO problem (\ref{eq:fo}) are in fact statistically well-behaved analogs of the independence constraint in (\ref{eqn:ofdr}). We consider the case of bounded random variables in this section:

\begin{assumption}
\label{ass:zdim}
The entries of the random variables $X,Z$ are almost surely bounded by $\alpha \geq 1$. Moreover, the maximal monomial degree of entries in $\omega(x,z)$ is $\rho \geq 1$, and the random variable $Z$ has dimensions $Z\in\mathbb{R}^r$.
\end{assumption}

\subsection{Concentration of Tensor Moment Estimates}

We begin by defining several multilinear operators. We define the empirical operators
\begin{equation}
\begin{aligned}
\widehat{\varphi}_{m,q}(B_1,\ldots,B_q) &= \textstyle\mathbb{E}_n\big(Z^{\otimes m}\bigotimes_{k=1}^{q}(B_k\Omega)\big)\\
\rlap{$\hspace{0.09em}\widehat{\nu}$}\hphantom{\widehat{\varphi}}_{m,q}(B_1,\ldots,B_q) &= \textstyle\mathbb{E}_n\big(Z^{\otimes m}\big)\otimes\mathbb{E}_n\big(\bigotimes_{k=1}^{q}(B_i\Omega)\big)\\
\end{aligned}
\end{equation}
and the expected operators
\begin{equation}
\begin{aligned}
\varphi_{m,q}(B_1,\ldots,B_q) &= \textstyle\mathbb{E}\big(Z^{\otimes m}\bigotimes_{k=1}^{q}(B_k\Omega)\big)\\
\rlap{$\hspace{0.09em}\nu$}\hphantom{\varphi}_{m,q}(B_1,\ldots,B_q) &= \textstyle\mathbb{E}\big(Z^{\otimes m}\big)\otimes\mathbb{E}\big(\bigotimes_{k=1}^{q}(B_i\Omega)\big)
\end{aligned}
\end{equation}
As a slight simplification of notation, when the argument of these multilinear operators is $(B)$ we take that to mean the argument is $(B,\ldots,B)$. We can thus identify these operators with terms in the FO problem (\ref{eq:fo}): The $\widehat{\varphi}_{m,q}(B)$ and $\widehat{\nu}_{m,q}(B)$ are precisely the terms appearing in the constraints.

\begin{proposition}
\label{prop:cphi}
If Assumptions \ref{ass:drule}, \ref{ass:zdim} hold, then we have
\begin{equation}
\label{eqn:cphi}
\textstyle\mathbb{P}\big(\|\widehat{\varphi}_{m,q}-\varphi_{m,q}\|_\circ > \mathcal{R}_{m,q}[n] + \gamma\big) \leq 2\exp\big(-\frac{n\gamma^2}{64p^q\alpha^{2m+2\rho q}}\big)
\end{equation}
for $\mathcal{R}_{m,q}[n] = 8\alpha^{m+\rho q}p^{q/2}\sqrt{\frac{dp\log(1+4q)+m\log r+q\log d}{n}}$.
\end{proposition}

\begin{proof}
We use a chaining argument. Suppose $\{t_i\}_{i=1}^N$ is a $\frac{1}{2q}$ covering of $\mathbb{S}^{dp-1}$, and note $N \leq (1+4q)^{dp}$ by the volume ratio bound \cite{wainwright2017high}. Define $T_i = M(t_i) \in\mathbb{R}^{d\times p}$. Let $P_q$ be the set of all permutations of $[q]$, and let
\begin{equation}
\label{eqn:tensym}
\textstyle\Phi(B_1,\ldots,B_q) = \frac{1}{q!}\sum_{\pi\in P_q}\big(\widehat{\varphi}_{m,q}(B_{\pi_1},\ldots,B_{\pi_q}) - \varphi_{m,q}(B_{\pi_1},\ldots,B_{\pi_q})\big).
\end{equation}
Observe that by construction: $\Phi(\cdot,\ldots,\cdot)$ is symmetric, and it satisfies the identity $\Phi(B) = \widehat{\varphi}_{m,q}(B) - \varphi_{m,q}(B)$. Now consider the telescoping sum
\begin{equation}
\label{eqn:tele}
\textstyle\Phi(B) = \Phi(T_i) + \sum_{k=1}^q \Phi(\stackrel{q-k}{\overbrace{B,\ldots,B}}, B - T_i, \stackrel{k-1}{\overbrace{T_i,\ldots,T_i}}).
\end{equation}
Recall $\|W(T_i)\|_2 = 1$ and $\|W(B - T_i)\|_2 \leq \frac{1}{2q}$ for $W(B) \in \mathbb{S}^{dp-1}$. Since $\|\cdot\|_*$ is a subordinate norm, we have $\|\Phi\|_\circ \leq \|\Phi(T_i)\| + \sum_{k=1}^q\frac{1}{2q}\|\Phi\|_*$. But note that $\Phi(\cdot,\ldots,\cdot)$ is symmetric, and so $\|\Phi\|_\circ = \|\Phi\|_*$ \cite{banach1938homogene,bochnak1971polynomials}. Thus we have $\|\Phi\|_\circ \leq 2\|\Phi(T_i)\|$. But by definition of the tensor norm $\|\cdot\|$ we have
\begin{equation}
\|\Phi(T_i)\| = \max_{u_k, v_k}\textstyle \big|\big\langle \Phi(T_i), \bigotimes_{k=1}^m u_k \bigotimes_{k=1}^q v_k\big\rangle\big|
%\textstyle \|\Phi(B)\| = \max_{u_k\in\mathbb{E}_p, v_k\in\mathbb{E}_d}\big\langle \Phi(B), \bigotimes_{k=1}^m u_k \bigotimes_{k=1}^q v_k\big\rangle
%\textstyle \|\Phi(B)\| = \max\big\{\big\langle \Phi(B), \bigotimes_{k=1}^m u_k \bigotimes_{k=1}^q v_k\big\rangle\ \big|\ u_k\in\mathbb{E}_p, v_k\in\mathbb{E}_d\big\}
%\begin{aligned}
%\textstyle\|\Phi(B)\| = \max\ &\textstyle \big\langle \Phi(B), \bigotimes_{k=1}^m u_k \bigotimes_{k=1}^q v_k\big\rangle\\\
%\mathrm{s.t.}\ &B \in \{\mathrm{mat}(t) : t \in \mathbb{S}^{dp-1}\}\\
%&u_k\in\mathbb{E}_p, v_k\in\mathbb{E}_d\}
%\end{aligned}
\end{equation}
for $u_k\in E_r, v_k\in E_d$; where $E_d = \{x \in \{0,1\}^d : \|x\|_1 = 1\}$. So it holds that
\begin{equation}
\label{eqn:pfref}
\|\Phi\|_\circ \leq 2\max_{i,u_k,v_k}\textstyle\big|\big\langle \Phi(T_i), \bigotimes_{k=1}^m u_k \bigotimes_{k=1}^q v_k\big\rangle\big|
\end{equation}
for $i\in [N], u_k\in E_r, v_k\in E_d$. Next consider any $s\in\mathbb{R}$, and observe that
\begin{equation}
\label{eqn:expbnd}
\begin{aligned}
\mathbb{E}\exp\big(s\|\Phi\|_\circ\big) & \leq \mathbb{E}\exp\big(2s\max_{i,u_k,v_k}\textstyle\big|\big\langle \Phi(T_i), \bigotimes_{k=1}^m u_k \bigotimes_{k=1}^q v_k\big\rangle\big|\big)\\
%& \leq 2Nr^md^q\max_{i,u_k,v_k}\textstyle\mathbb{E}\exp\big(2s\big\langle \Phi(T_i), \bigotimes_{k=1}^m u_k \bigotimes_{k=1}^q v_k\big\rangle\big)
& \textstyle\leq \sum_{\sigma\in\pm 1,i,u_k,v_k}\textstyle\mathbb{E}\exp\big(2s\sigma\big\langle \Phi(T_i), \bigotimes_{k=1}^m u_k \bigotimes_{k=1}^q v_k\big\rangle\big)
\end{aligned}
\end{equation}
We seek to bound the term on the right-hand side. Towards this end, note $\|B\Omega_i\| \leq \sqrt{p}\|W(B)\|_2\|\Omega_i\| \leq \sqrt{p}\alpha^\rho$ by the Cauchy-Schwarz inequality and Assumption \ref{ass:zdim}. This means that for $S_i = \sigma\big\langle Z^{\otimes m}(T_i\Omega)^{\otimes q}, \bigotimes_{k=1}^m u_k \bigotimes_{k=1}^q v_k\big\rangle$ we have  $\big|S_i\big| \leq \alpha^{m+\rho q}p^{q/2}$. Next observe that 
\begin{equation}
\label{eqn:bigone}
\begin{aligned}
\textstyle\mathbb{E}\exp\big(2s\sigma\big\langle \Phi(T_i), \bigotimes_{k=1}^m u_k \bigotimes_{k=1}^q v_k\big\rangle\big) &\leq \textstyle\big(\mathbb{E}\exp\big(\frac{4\epsilon s S_i}{n}\big)\big)^n\\
%&\qquad\leq\textstyle\mathbb{E}\exp\big(2\epsilon\lambda\big\langle \widehat{\varphi}_{m,q}(T_i), \bigotimes_{k=1}^m u_k \bigotimes_{k=1}^q v_k\big\rangle\big)\\
%&\qquad=\textstyle\big(\mathbb{E}\exp\big(2\epsilon\frac{\lambda}{n}\big\langle Z^{\otimes m}(T_i\Omega)^{\otimes q}, \bigotimes_{k=1}^m u_k \bigotimes_{k=1}^q v_k\big\rangle\big)\big)^n\\
&=\textstyle\big(\mathbb{E}\sum_{k=0}^\infty\frac{1}{k!}\big(\frac{4\epsilon s S_i}{n}\big)^k\big)^n\\
&=\textstyle \big(\mathbb{E}\sum_{k=0}^\infty\frac{1}{(2k)!}\big(\frac{4s S_i}{n}\big)^{2k}\big)^n\\
%&\textstyle\leq \big(\sum_{k=0}^\infty\frac{1}{k!}\big(p^q\big(\frac{4\lambda \alpha^{m+\rho q}}{n}\big)^2\big)^{k}\big)^n\\
&\textstyle\leq \big(\sum_{k=0}^\infty\frac{1}{k!}\big(\frac{16s^2p^q\alpha^{2m+2\rho q}}{n^2}\big)^{k}\big)^n\\
%&\textstyle = \exp\big(\frac{p^q}{n}\big(4\lambda \alpha^{m+\rho q}\big)^2\big)
&\textstyle = \exp\big(\frac{16s^2p^q\alpha^{2m+2\rho q}}{n}\big)
\end{aligned}
\end{equation}
where the first line follows by a stochastic symmetrization step (i.e., Jensen's inequality, multiplication with i.i.d. Rademacher random variables $\epsilon$ having distribution $\mathbb{P}(\epsilon = \pm 1) = \frac{1}{2}$, using the triangle inequality, and concluded by Jensen's inequality), the third line follows since $\epsilon$ is a symmetric random variable, and the fourth line follows by replacing $(2k!)$ with $k!$ and substituting the absolute bound on $|S_i|$. Combining the above with (\ref{eqn:expbnd}) gives
\begin{equation}
%\textstyle\mathbb{E}\exp\big(\lambda\|\Phi\|_\circ\big) \leq 2(1+4q)^{dp}r^md^q\exp\big(\frac{p^q}{n}\big(4\lambda \alpha^{m+\rho q}\big)^2\big).
\textstyle\mathbb{E}\exp\big(s\|\Phi\|_\circ\big) \leq 2(1+4q)^{dp}r^md^q\exp\big(\frac{16s^2p^q\alpha^{2m+2\rho q}}{n}\big).\end{equation}
Using the Chernoff bound gives 
\begin{equation}
\begin{aligned}
%\mathbb{P}\big(\|\Phi\|_\circ > t\big) &\leq 2(1+4q)^{dp}r^md^q\inf_{\lambda\in\mathbb{R}}\textstyle\exp\big(\frac{p^q}{n}\big(4\lambda \alpha^{m+\rho q}\big)^2-\lambda t\big)\\
%&=2(1+4q)^{dp}r^md^q\exp\big(-\frac{nt^2}{4(p^q(4\alpha^{m+\rho q})^2)}\big)
\mathbb{P}\big(\|\Phi\|_\circ > t\big) &\leq 2(1+4q)^{dp}r^md^q\inf_{s\in\mathbb{R}}\textstyle\exp\big(\frac{16s^2p^q\alpha^{2m+2\rho q}}{n}-s t\big)\\
&\textstyle=2(1+4q)^{dp}r^md^q\exp\big(-\frac{nt^2}{64p^q\alpha^{2m+2\rho q}}\big)
\end{aligned}
\end{equation}
The result now follows by choosing
\begin{equation}
\textstyle t = \sqrt{\frac{64p^q\alpha^{2m+2\rho q}}{n}\big(dp \log(1+4q) + m\log r + q \log d\big) + \gamma^2}
\end{equation}
and accordingly simplifying the resulting expression.
\end{proof}

%\begin{remark}
%Note that $\|\Phi(B)\| \leq \|\Phi\|_\circ\|V(B)\|_2^q$.
%\end{remark}

\begin{remark}
Though a similar proof was used in \cite{wainwright2017high} for random matrices and in \cite{tomioka2014spectral} for random tensors, we use a stronger argument that is adapted to our setup and results in a faster convergence rate where some terms are logarithmic that would otherwise be polynomial with a weaker argument. We use a stronger chaining argument than \cite{tomioka2014spectral,wainwright2017high} by using a telescoping sum (\ref{eqn:tele}) that reduces cross terms. We use a tensor symmetrization construction (\ref{eqn:tensym}) that allows us to exploit Banach's theorem \cite{banach1938homogene,bochnak1971polynomials}. We achieve better constants than \cite{wainwright2017high} by more carefully bounding our moment series expansion.
%Though a similar proof was used in \cite{wainwright2017high} for random matrices and in \cite{tomioka2014spectral} for random tensors, we use a stronger chaining argument that is adapted to our setup and results in a faster convergence rate for the concentration. Use of the chaining approach from \cite{tomioka2014spectral,wainwright2017high} would result in a $qd$ term instead of the $q\log d$ term that we achieve. We also achieve better constants than \cite{wainwright2017high} by more carefully bounding our moment series expansion.
\end{remark}

\begin{proposition}
\label{prop:cpsi}
If Assumptions \ref{ass:drule}, \ref{ass:zdim} hold, then we have
\begin{equation}
\label{eqn:cpsi}
\textstyle\mathbb{P}\big(\|\widehat{\nu}_{m,q}-\nu_{m,q}\|_\circ > 2\mathcal{R}_{m,q}[n] + 2\gamma\big) \leq 4\exp\big(-\frac{n\gamma^2}{64p^q\alpha^{2m+2\rho q}}\big).
\end{equation}
for $\mathcal{R}_{m,q}[n] = 8\alpha^{m+\rho q}p^{q/2}\sqrt{\frac{dp\log(1+4q)+m\log r+q\log d}{n}}$.
\end{proposition}

\begin{proof}
We cannot prove the result directly as in Proposition \ref{prop:cphi} because $\mathbb{E}\widehat{\nu}_{m,q}(B) \neq \nu_{m,q}(B)$, whereas the proof of Proposition \ref{prop:cphi} used the fact that $\mathbb{E}\widehat{\varphi}_{m,q}(B) = \varphi_{m,q}(B)$ in the symmetrization step of (\ref{eqn:bigone}). We instead have to use an indirect approach to prove this result. We begin by noting $\widehat{\varphi}_{m,0}(B) = \mathbb{E}_n(Z^{\otimes m})$, $\varphi_{m,0}(B) = \mathbb{E}(Z^{\otimes m})$, $\widehat{\varphi}_{0,q}(B) = \mathbb{E}_n((B\Omega)^{\otimes q})$, and $\varphi_{0,q}(B) = \mathbb{E}((B\Omega)^{\otimes q})$. For any $W(B) \in \mathbb{S}^{dp-1}$ we have that $\|B\Omega_i\| \leq \sqrt{p}\|W(B)\|_2\|\Omega_i\| \leq \sqrt{p}\alpha^\rho$ by the Cauchy-Schwarz inequality and Assumption \ref{ass:zdim}. This means that $\|\widehat{\varphi}_{m,0}\|_\circ \leq \alpha^m$ and $\|\varphi_{0,q}\|_\circ \leq \alpha^{\rho q}p^{q/2}$. Now consider
\begin{equation}
\begin{aligned}
\|\widehat{\nu}_{m,q} - \nu_{m,q}\|_\circ &= \|\widehat{\varphi}_{m,0}\otimes\widehat{\varphi}_{0,q} - \varphi_{m,0}\otimes\varphi_{0,q}\|_\circ\\
&\leq \|\widehat{\varphi}_{m,0}\|_\circ\cdot\|\widehat{\varphi}_{0,q} - \varphi_{0,q}\|_\circ + \|\varphi_{0,q}\|_\circ\cdot\|\widehat{\varphi}_{m,0}-\varphi_{m,0}\|_\circ\\
&\leq \alpha^m\|\widehat{\varphi}_{0,q} - \varphi_{0,q}\|_\circ + \alpha^{\rho q}p^{q/2}\|\widehat{\varphi}_{m,0}-\varphi_{m,0}\|_\circ
\end{aligned}
\end{equation}
Then the union bound implies
\begin{multline}
\textstyle\mathbb{P}\big(\|\widehat{\nu}_{m,q} - \nu_{m,q}\|_\circ \leq 2\mathcal{R}_{m,q}[n]+ 2\gamma\big) \geq \\
\textstyle 1 - \mathbb{P}\big(\alpha^m\|\widehat{\varphi}_{0,q} - \varphi_{0,q}\|_\circ > \mathcal{R}_{m,q}[n] +\gamma\big) +\\
\textstyle- \mathbb{P}\big(\alpha^{\rho q}p^{q/2}\|\widehat{\varphi}_{m,0} - \varphi_{m,0}\|_\circ > \mathcal{R}_{m,q}[n] +\gamma\big)
\end{multline}
for $\mathcal{R}_{m,q}[n] = 8\alpha^{m+\rho q}p^{q/2}\sqrt{\frac{dp\log(1+4q)+m\log r+q\log d}{n}}$, which upon using (\ref{eqn:cphi}) from Proposition \ref{prop:cphi} gives (\ref{eqn:cpsi}), which is the desired result.
\end{proof}

\subsection{Feasible Set Consistency}

We are now in a position to study the constraints of the FO problem (\ref{eq:fo}). Towards this goal, we first define
\begin{equation}
\mathcal{S} = \big\{B \in\mathcal{B} : B\Omega \perp \!\!\! \perp Z\big\}.
\end{equation}
This is the feasible set of (\ref{eqn:ofdr}), which chooses an optimal fair decision rule when the underlying distributions are exactly known, for a decision rule that satisfies Assumption \ref{ass:drule}. We next define the family of random sets
\begin{multline}
\widehat{\mathcal{S}}_{\mathfrak{g},\mathfrak{h}} = \big\{B \in \mathcal{B} : \textstyle\big\|\widehat{\varphi}_{m,q}(B)-\widehat{\nu}_{m,q}(B)\big\|\le\Delta_{m,q},\text{for } (m,q)\in[\mathfrak{g}]\times[\mathfrak{h}]\big\}.
%\widehat{\mathcal{S}}_{f,g} = \big\{B \in \mathcal{B} : \\
%\textstyle\big\|\mathbb{E}_n\big(Z^{\otimes m}(B\Omega)^{\otimes q}\big)-\mathbb{E}_n\big(Z^{\otimes m}\big)\otimes\mathbb{E}_n\big((B\Omega)^{\otimes q}\big)\big\|\le\Delta_{m,q}, \\
%\text{for } (m,q)\in[f]\times[g]\big\}
\end{multline}
This is simply the feasible set of the level-$(\mathfrak{g},\mathfrak{h})$ FO problem (\ref{eq:fo}).

\begin{proposition}
\label{prop:closed}
$\mathcal{S}$ and $\widehat{\mathcal{S}}_{\mathfrak{g},\mathfrak{h}}$ are closed, under {Assumptions \ref{ass:drule}, \ref{ass:zdim}.}
\end{proposition}

\begin{proof}
%We prove the result for $\mathcal{S}$, and we note that the proof for $\widehat{\mathcal{S}}_{\mathfrak{g},\mathfrak{h}}$ is a simple modification of this one. Consider any convergent sequence $B_k\in\mathbb{R}^{d\times p}$ with $B_k \in \mathcal{S}$ and $\lim_kB_k=B_0$. Theorem \ref{thm:kac} says that for all $k$ we have
%\begin{equation}
%\varphi_{m,q}(B_k) = \nu_{m,q}(B_k) \text{ for } m,q\geq 1.
%\end{equation}
%But the $\varphi$ and $\nu$ are continuous since they are multlinear operators on Euclidean space. This means $\lim_k \varphi_{m,q}(B_k) = \varphi_{m,q}(B_0)$ and $\lim_k \nu_{m,q}(B_k) = \nu_{m,q}(B_0)$ for $m,q\geq 1$. As a result we have
%\begin{equation}
%\varphi_{m,q}(B_0) = \nu_{m,q}(B_0) \text{ for } m,q\geq 1,
%\end{equation}
%which by Theorem 1 implies $B_0 \in \mathcal{S}$. This proves that $\mathcal{S}$ is closed.
We first prove the result for $\mathcal{S}$. Consider any convergent sequence $B_k\in\mathbb{R}^{d\times p}$ with $B_k \in \mathcal{S}$ and $\lim_kB_k=B_0$. {Because of our assumptions, the hypothesis of Theorem \ref{thm:kac} is satisfied. This theorem says for all $k$ we have}
\begin{equation}
\varphi_{m,q}(B_k) = \nu_{m,q}(B_k), \text{for } m,q\geq 1.
\end{equation}
But the $\varphi$ and $\nu$ are continuous since they are multilinear operators on Euclidean space. This means $\lim_k \varphi_{m,q}(B_k) = \varphi_{m,q}(B_0)$ and $\lim_k \nu_{m,q}(B_k) = \nu_{m,q}(B_0)$ for $m,q\geq 1$. As a result we have
\begin{equation}
\varphi_{m,q}(B_0) = \nu_{m,q}(B_0), \text{for } m,q\geq 1,
\end{equation}
which by Theorem 1 implies $B_0 \in \mathcal{S}$. This proves that $\mathcal{S}$ is closed.

The proof for $\widehat{\mathcal{S}}_{\mathfrak{g},\mathfrak{h}}$ is a simple modification of the above argument. Consider any convergent sequence $B_k\in\mathbb{R}^{d\times p}$ with $B_k \in \widehat{\mathcal{S}}_{\mathfrak{g},\mathfrak{h}}$ and $\lim_kB_k=B_0$. By definition of $\widehat{\mathcal{S}}_{\mathfrak{g},\mathfrak{h}}$ we have for all $k$ that
\begin{equation}
\textstyle\big\|\widehat{\varphi}_{m,q}(B_k)-\widehat{\nu}_{m,q}(B_k)\big\|\le\Delta_{m,q},\text{for } (m,q)\in[\mathfrak{g}]\times[\mathfrak{h}].
\end{equation}
But the $\widehat{\varphi}$ and $\widehat{\nu}$ are continuous since they are multilinear operators on Euclidean space, and so the normed function $\big\|\widehat{\varphi}_{m,q}(B)-\widehat{\nu}_{m,q}(B)\big\|$ is also continuous. As a result we have
\begin{multline}
\textstyle\big\|\widehat{\varphi}_{m,q}(B_0)-\widehat{\nu}_{m,q}(B_0)\big\| = \lim_k \big\|\widehat{\varphi}_{m,q}(B_k)-\widehat{\nu}_{m,q}(B_k)\big\|\leq\Delta_{m,q},\\ \text{for } m,q\geq 1.
\end{multline}
This means $B_0 \in \widehat{\mathcal{S}}_{\mathfrak{g},\mathfrak{h}}$ by definition. This proves that $\widehat{\mathcal{S}}_{\mathfrak{g},\mathfrak{h}}$ is closed.
\end{proof}

\begin{figure*}[t]
	\begin{center}
		\begin{subfigure}[t]{0.47\linewidth}
			\includegraphics[width=\linewidth]{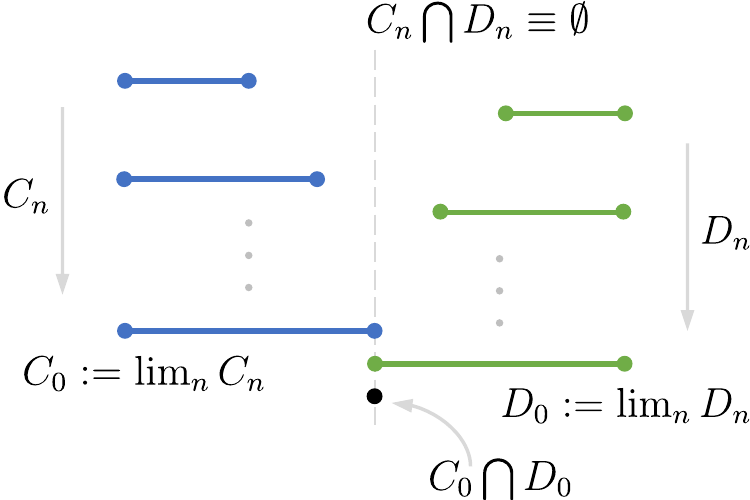}
			\caption{Unregularized Set Intersections}
		\end{subfigure}\qquad
		\begin{subfigure}[t]{0.47\linewidth}
			\includegraphics[width=\linewidth]{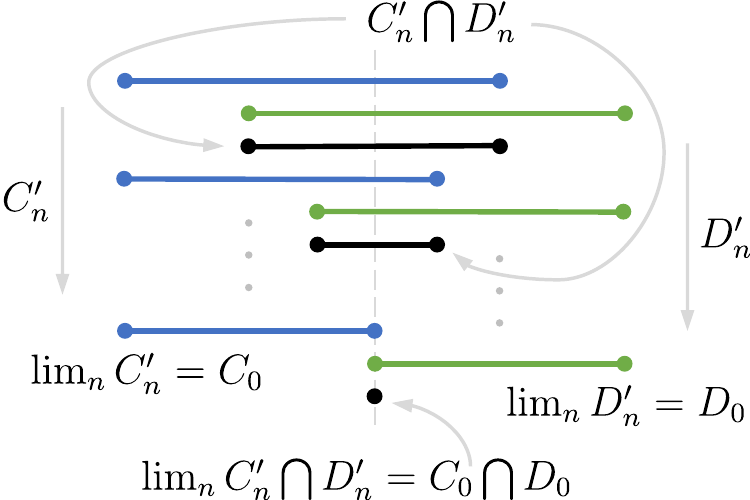}
			\caption{Regularized Set Intersections}
		\end{subfigure}
	\end{center}
	\caption{\label{fig:setex} The left shows how the intersection of a sequence of sets may not converge to the intersection of the limiting sets. The right shows how regularization of the sequence of sets can help to ensure that the intersection of the regularized sets converges to the intersection of the limiting sets.}
\end{figure*}

The sequence of random sets $\widehat{\mathcal{S}}_{\mathfrak{g},\mathfrak{h}}$ is technically difficult to study because each random set is defined by the intersection of many random constraint inequalities, with the number of these random constraints increasing towards infinity. There is a more subtle technical difficulty that needs to be addressed. The issue is that when intersecting a sequence of sets, the intersection of the sequence terms generally does not converge to the intersection of the limiting sets \cite{aswani2019statistics,matheron1975}. The next example demonstrates this phenomenon in a deterministic setting, and it provides some insight into how the situation can be addressed through a carefully designed regularization approach.

\begin{example}
\label{exa:detset}
%Let $C_n = \{x\in\mathbb{R} : x \leq -\frac{1}{n}\}$ and $D_n = \{x\in\mathbb{R} : \frac{1}{n} - x \leq 0\}$ be deterministic sequences of sets. Then $\lim_n C_n = \{x \in \mathbb{R} : x \leq 0\} =: C_0$ and $\lim_n D_n = \{x \in\mathbb{R} : -x \leq 0\} =: D_0$. However, note $C_n\bigcap D_n = \emptyset$. This means $\lim_n C_n\bigcap D_n = \emptyset \neq C_0\bigcap D_0 = \{0\}$.
Fig. \ref{fig:setex} provides a visualization of this example. Let us first define $C_n = [-1,-\frac{1}{n}]$ and $D_n = [\frac{1}{n},1]$, which each specify a deterministic sequence of compact sets. Then we have that $\lim_n C_n = [-1,0] =: C_0$ and that $\lim_n D_n = [0,1] =: D_0$. However, note that $C_n\bigcap D_n = \emptyset$. This means $\lim_n C_n\bigcap D_n = \emptyset \neq C_0\bigcap D_0 = \{0\}$. Now suppose we carefully regularize these sequences of sets. Specifically consider the regularized sequence of deterministic, compact sets $C_n' = [-1,-\frac{1}{n} + \Delta_n]$ and $D_n' = [\frac{1}{n}-\Delta_n, 1]$ for $\Delta_n = \frac{2}{n}$, where we think of the $\Delta_n$ as regularizing by inflating the sets. Clearly this choice of regularization goes to zero since $\lim_n\Delta_n = 0$. More importantly, we now have $C_n'\bigcap D_n' = [-\frac{1}{n}, \frac{1}{n}]$. This means we have $\lim_n C_n' = C_0$ and $\lim_n D_n' = D_0$ with $\lim_n C_n'\bigcap D_n' = \{0\} = C_0\bigcap D_0$.
\end{example}

The above example was deterministic, and it may not initially be clear whether such behavior is an issue for our random setting. The next example demonstrates a situation where this non-convergence occurs for $\widehat{\mathcal{S}}_{\mathfrak{g},\mathfrak{h}}$.

\begin{example}
\label{exa:noncon}
Consider a setting where $B\in\mathbb{R}$ and the distributions are $X \sim \mathrm{Ber}(x)$ and $Z\sim\mathrm{Ber}(z)$ with $X\perp \!\!\! \perp Z$. We assume that $x\in(0,1)$ and $z\in(0,1)$ to prevent degeneracies in this example. In this setup $\mathcal{S} = \mathcal{B}$. Now observe that $(Z_i)^m = Z_i$ and $(X_i)^q = X_i$ for $(m,q)\geq 1$ since $X_i,Z_i\in\{0,1\}$. This means the $(m,q) \geq 1$ constraints in $\widehat{\mathcal{S}}_{\mathfrak{g},\mathfrak{h}}$ for $\Delta_{m,q} = 0$ are
\begin{multline}
\textstyle\big|\big(\frac{1}{n}\sum_{i=1}^n(Z_i)^m(X_i)^q - \frac{1}{n}\sum_{i=1}^n(Z_i)^q\cdot\frac{1}{n}\sum_{i=1}^n(X_i)^q\big)B^q\big| = \\\textstyle\big|\big(\frac{1}{n}\sum_{i=1}^nZ_iX_i - \frac{1}{n}\sum_{i=1}^nZ_i\cdot\frac{1}{n}\sum_{i=1}^nX_i\big)B^q\big| = 0.
\end{multline}
This means $\widehat{\mathcal{S}}_{\mathfrak{g},\mathfrak{h}} = \mathcal{B}$ whenever $\mathcal{E}_n = \{\frac{1}{n}\sum_{i=1}^nZ_iX_i = \frac{1}{n}\sum_{i=1}^nZ_i\cdot\frac{1}{n}\sum_{i=1}^nX_i\}$ occurs, and that $\widehat{\mathcal{S}}_{\mathfrak{g},\mathfrak{h}} = \{0\}$ otherwise. And so trivially by the definition of $\widehat{\mathcal{S}}_{\mathfrak{g},\mathfrak{h}}$ we have $\aslimsup_n \widehat{\mathcal{S}}_{\mathfrak{g},\mathfrak{h}} \subseteq\mathcal{B}$. If we recall the classical setting of a $2\times 2$ contingency table, this event $\mathcal{E}_n$ is equivalent to having exact equality between a marginal and cross-term in the contingency table. As a result, we consider a test statistic inspired by the Pearson test for independence
\begin{equation}
T_n = n\cdot\big(\mathbb{E}_n(ZX) - \mathbb{E}_n(Z)\mathbb{E}_n(X)\big)^2. 
\end{equation}
Clearly by its definition, we have that $T_n = 0$ if and only if $\mathcal{E}_n$ holds. Also, a straightforward calculation gives
\begin{equation}
\textstyle\mathbb{E}(T_n) = (\frac{n-1}{n})(zx)(1-z-x-zx).
\end{equation}
Note that $\mathbb{E}(T_n) > 0$ since we assumed $x,z\in(0,1)$, and note that $\mathbb{E}(T_n)$ is monotonically increasing towards $\lim_n \mathbb{E}(T_n) = (zx)(1-z-x-zx) > 0$. Now using McDiarmid's inequality we get for any $t > 0$ that
\begin{equation}
\mathbb{P}(\mathcal{E}_n) \leq \mathbb{P}(T_n \leq \mathbb{E}(T_n) - t) \leq \exp(-nt^2/8).
\end{equation}
Choosing $t = (zx)(1-z-x-zx)/2$, the Borel-Cantelli lemma implies $\mathcal{E}_n$ cannot occur infinitely often. Hence we must have $\asliminf_n \widehat{\mathcal{S}}_{\mathfrak{g},\mathfrak{h}} = \{0\} \nsupseteq \mathcal{S}$.
\end{example}

%Example \ref{exa:detset} provides the key intuition for how potential non-convergence of $\widehat{\mathcal{S}}_{\mathfrak{g},\mathfrak{h}}$, as demonstrated in Example \ref{exa:noncon}, can be resolved. If we can regularize the sets $\widehat{\mathcal{S}}_{\mathfrak{g},\mathfrak{h}}$ by inflating them in such a way that the sets are inflated by a sufficient amount while ensuring that the amount of inflation decreases with $n$, then we may be able to ensure convergence of $\widehat{\mathcal{S}}_{\mathfrak{g},\mathfrak{h}}$ to $\mathcal{S}$. In fact, the notation of Example \ref{exa:detset} was chosen to be suggestive of how we will perform this regularization: We will purposefully keep the $\Delta_{m,q} > 0$ while allowing them to shrink towards zero.

Example \ref{exa:detset} provides the key intuition for how potential non-convergence of $\widehat{\mathcal{S}}_{\mathfrak{g},\mathfrak{h}}$, as demonstrated in Example \ref{exa:noncon}, can be resolved. If we can regularize the sets $\widehat{\mathcal{S}}_{\mathfrak{g},\mathfrak{h}}$ by sufficiently inflating them in such a way that the amount of inflation decreases with $n$, then we may be able to ensure the almost sure stochastic convergence of $\widehat{\mathcal{S}}_{\mathfrak{g},\mathfrak{h}}$ to $\mathcal{S}$. In fact, the notation of Example \ref{exa:detset} was chosen to be suggestive of how we will perform this regularization: We will purposefully keep the $\Delta_{m,q} > 0$ while allowing them to shrink towards zero.

More broadly, the FO problem (\ref{eq:fo}) has two types of tuning parameters, namely the $(\mathfrak{g},\mathfrak{h})$ that controls the number of moment constraints and the $\Delta_{m,q}$ that controls the strictness of the moment constraint. This gives us considerable flexibility when studying asymptotic properties. In the following results, we will have to make choices for both of these tuning parameters.

%A choice of faster rates requires knowledge of an appropriate value of $\alpha$ from Assumption \ref{ass:zdim}. An alternative approach is to choose slower rates that have the benefit of working for any value of $\alpha$. Here, we will take the latter approach.

\begin{theorem}
\label{thm:setcon}
%Suppose that $\Delta_{m,q} = O(n^{-1/4})$ and $\mathfrak{g} = \mathfrak{h} = O(\log\log n)$. If Assumptions \ref{ass:drule}, \ref{ass:zdim}, and \ref{ass:2norm} hold, then $\aslim_n \widehat{\mathcal{S}}_{\mathfrak{g},\mathfrak{h}} = \mathcal{S}$.
{Suppose $\Delta_{m,q} = 3(1+\log n)\cdot\mathcal{R}_{m,q}[n]$ and $\mathfrak{g} = \mathfrak{h} = O(\log n)$, such that $\Delta_{\mathfrak{g},\mathfrak{h}} = o(1)$.} If Assumptions \ref{ass:drule}, \ref{ass:2norm}, \ref{ass:zdim} hold, then $\aslim_n \widehat{\mathcal{S}}_{\mathfrak{g},\mathfrak{h}} = \mathcal{S}$.
\end{theorem}

\begin{proof}
For the first part of the proof we will show $\asliminf_n \widehat{\mathcal{S}}_{\mathfrak{g},\mathfrak{h}} \supseteq \mathcal{S}$. Indeed, suppose this is not true. Then there exists $B_0 \in \mathcal{S}$ and an open neighborhood $\mathcal{N}\subseteq\mathcal{B}$ of $B_0$ such that $\mathcal{N}\bigcap\widehat{\mathcal{S}}_{\mathfrak{g},\mathfrak{h}} = \emptyset$ infinitely often (Theorem 4.5 of \cite{rockafellar2009variational}). We can rewrite one of these events as
\begin{equation}
\label{eqn:event}
\textstyle\big\{\mathcal{N}\bigcap\widehat{\mathcal{S}}_{\mathfrak{g},\mathfrak{h}} = \emptyset\big\} = \bigcup_{m\in[\mathfrak{g}]}\bigcup_{q\in[\mathfrak{h}]}\big\{\displaystyle\inf_{B\in\mathcal{N}}\|\widehat{\Xi}_{m,q}(B)\| > \Delta_{m,q}\big\},
\end{equation}
where for convenience we define the multilinear operators ${\Xi}_{m,q} = {\varphi}_{m,q}-{\nu}_{m,q}$, $\widehat{\Xi}_{m,q} = \widehat{\varphi}_{m,q}-\widehat{\nu}_{m,q}$, $\Phi_{m,q} = \widehat{\varphi}_{m,q}-\varphi_{m,q}$, and $\Psi_{m,q} = \widehat{\nu}_{m,q} - \nu_{m,q}$. Because Theorem \ref{thm:kac} can be rewritten under the assumptions of this theorem as
\begin{equation}
\label{eqn:kacres}
\sup_{B\in\mathcal{S}} \|\varphi_{m,q}(B)-\nu_{m,q}(B)\| = 0 \text{ for } m,q\geq 1,
\end{equation}
application of the triangle inequality yields
\begin{equation}
\begin{aligned}
%\|\widehat{\Xi}_{m,q}(B_0)\| &\leq \sup_{B\in\mathcal{N}}\|\Phi_{m,q}(B)\| + \sup_{B\in\mathcal{N}}\|\Psi_{m,q}(B)\| \\
\|\widehat{\Xi}_{m,q}(B_0)\| &\leq \|\Xi_{m,q}(B_0)\| + \|\Phi_{m,q}(B_0)\| + \|\Psi_{m,q}(B_0)\| \\
&\leq\lambda^{q/2}\|\Phi_{m,q}\|_\circ + \lambda^{q/2}\|\Psi_{m,q}\|_\circ
\end{aligned}
\end{equation}
%Noting $3(1+\log n)\lambda^{q/2}\mathcal{R}_{m,q}[n] = o(n^{-1/4})$ for $(m,q)\in[\mathfrak{g}]\times[\mathfrak{h}]$ under the hypothesis of this theorem, we define $\mathcal{G}_{m,q}[n] = (1+\log n)\lambda^{q/2}\mathcal{R}_{m,q}[n]$. For all $n$ sufficiently large, the union bound gives us that
Let $\mathcal{G}_{m,q}[n] = (1+\log n)\lambda^{q/2}\mathcal{R}_{m,q}[n]$. Note that for all $n$ sufficiently large, the union bound gives us that
\begin{equation}
\begin{aligned}
\textstyle\mathbb{P}\big(\mathcal{N}\bigcap\widehat{\mathcal{S}}_{\mathfrak{g},\mathfrak{h}} = \emptyset\big) &\textstyle\leq \sum_{m\in [\mathfrak{g}]}\sum_{q\in[\mathfrak{h}]}\mathbb{P}\big(\lambda^{q/2}\|\Phi_{m,q}\|_\circ>\mathcal{G}_{m,q}[n]\big) + \\
&\textstyle\qquad \sum_{m\in [\mathfrak{g}]}\sum_{q\in[\mathfrak{h}]}\mathbb{P}\big(\lambda^{q/2}\|\Psi_{m,q}\|_\circ>2\mathcal{G}_{m,q}[n]\big)\\
&\leq O((\log n/n)^2)
\end{aligned}
\end{equation}
where the last line used Propositions \ref{prop:cphi} and \ref{prop:cpsi}, along with the relation that $\exp(-\frac{n\gamma^2}{64p^q\alpha^{2m+2\rho q}}) = O(1/n^2)$ for $\gamma = \log n\cdot\mathcal{R}_{m,q}[n]$. Thus the Borel-Cantelli lemma says $\mathcal{N}\bigcap\widehat{\mathcal{S}}_{\mathfrak{g},\mathfrak{h}} = \emptyset$ only finitely many times, which is a contradiction. This proves $\asliminf_n \widehat{\mathcal{S}}_{\mathfrak{g},\mathfrak{h}} \supseteq \mathcal{S}$.

For the second part of the proof we will show $\aslimsup_n \widehat{\mathcal{S}}_{\mathfrak{g},\mathfrak{h}} \subseteq \mathcal{S}$. Indeed, suppose this is not true. Then there exists $B_0 \in \limsup_n \widehat{\mathcal{S}}_{\mathfrak{g},\mathfrak{h}}$ and a closed neighborhood $\mathcal{N}\subseteq\mathcal{B}$ of $B_0$ such that $\mathcal{N}\bigcap\mathcal{S} = \emptyset$ and $\mathcal{N}\bigcap\widehat{\mathcal{S}}_{\mathfrak{g},\mathfrak{h}} \neq \emptyset$ infinitely often (Theorem 4.5 of \cite{rockafellar2009variational}). But Theorem \ref{thm:kac} implies there exists some $m,q\geq 1$ such that we have
\begin{equation}
\zeta := \inf_{B\in\mathcal{N}}\|\varphi_{m,q}(B)-\nu_{m,q}(B)\| > 0.
\end{equation}
We will keep $m,q$ fixed at these values for the remainder of the proof. Now note that for one of the events $\mathcal{N}\bigcap\widehat{\mathcal{S}}_{\mathfrak{g},\mathfrak{h}} \neq \emptyset$ we have
\begin{equation}
\textstyle\big\{\mathcal{N}\bigcap\mathcal{S}_{\mathfrak{g},\mathfrak{h}} \neq \emptyset\big\} \subseteq \displaystyle\big\{\inf_{B\in\mathcal{N}}\|\widehat{\Xi}_{m,q}(B)\| \leq \Delta_{m,q}\big\}.
\end{equation}
Application of the triangle inequality yields
\begin{multline}
\zeta = \inf_{B\in\mathcal{N}}\|{\Xi}_{m,q}(B)\| \leq \\
\inf_{B\in\mathcal{N}}\|\widehat{\Xi}_{m,q}(B)\| + \sup_{B\in\mathcal{N}}\|\Phi_{m,q}(B)\| + \sup_{B\in\mathcal{N}}\|\Psi_{m,q}(B)\| \leq\\
\inf_{B\in\mathcal{N}}\|\widehat{\Xi}_{m,q}(B)\| + \lambda^{q/2}\|\Phi_{m,q}\|_\circ + \lambda^{q/2}\|\Psi_{m,q}\|_\circ.
\end{multline}
Let $\mathcal{G}_{m,q}[n] = (1+\log n)\lambda^{q/2}\mathcal{R}_{m,q}[n]$. Note that for all $n$ sufficiently large, we have $\zeta-\Delta_{m,q} \geq \zeta/2 \geq 3\mathcal{G}_{m,q}[n]$. Hence the union bound gives
\begin{equation}
\begin{aligned}
\textstyle\mathbb{P}\big(\mathcal{N}\bigcap\widehat{\mathcal{S}}_{\mathfrak{g},\mathfrak{h}} \neq \emptyset\big) &\textstyle\leq \mathbb{P}\big(\lambda^{q/2}\|\Phi_{m,q}\|_\circ>\mathcal{G}_{m,q}[n]\big) + \\
&\textstyle\qquad \mathbb{P}\big(\lambda^{q/2}\|\Psi_{m,q}\|_\circ>2\mathcal{G}_{m,q}[n]\big)\\
&\leq O(1/n^2)
\end{aligned}
\end{equation}
where the last line used Propositions \ref{prop:cphi} and \ref{prop:cpsi}, along with the relation that $\exp(-\frac{n\gamma^2}{64p^q\alpha^{2m+2\rho q}}) = O(1/n^2)$ for $\gamma = \log n\cdot\mathcal{R}_{m,q}[n]$. Thus the Borel-Cantelli lemma says $\mathcal{N}\bigcap\widehat{\mathcal{S}}_{\mathfrak{g},\mathfrak{h}} \neq \emptyset$ only finitely many times, which is a contradiction. This proves $\aslimsup_n \widehat{\mathcal{S}}_{\mathfrak{g},\mathfrak{h}} \subseteq \mathcal{S}$.
\end{proof}

\subsection{Solution Set Consistency}

Next consider the solution set
\begin{equation}
\widehat{\mathcal{O}}_{\mathfrak{g},\mathfrak{h}} = \arg\min_{B}\big\{R_n(B\cdot\omega(x,z))\ \big|\ B \in \widehat{\mathcal{S}}_{\mathfrak{g},\mathfrak{h}}\big\}
\end{equation}
for the level-$(\mathfrak{g},\mathfrak{h})$ FO problem (\ref{eq:fo}). Similarly, consider the solution set
\begin{equation}
\mathcal{O} = \arg\min_{B}\big\{R(B\cdot\omega(x,z))\ \big|\ B \in \mathcal{S}\big\}
\end{equation}
for the optimization problem (\ref{eqn:ofdr}), which chooses an optimal fair decision rule when the underlying distributions are exactly known.

Our next result shows that solving the FO problem (\ref{eq:fo}) provides a statistically consistent approximation to solving the optimization problem (\ref{eqn:ofdr}), and we state the result using the solutions sets $\widehat{\mathcal{O}}_{\mathfrak{g},\mathfrak{h}}$ and $\mathcal{O}$ defined above.

\begin{theorem}
\label{thm:opcon}
{Suppose $\Delta_{m,q} = 3(1+\log n)\cdot\mathcal{R}_{m,q}[n]$ and $\mathfrak{g} = \mathfrak{h} = O(\log n)$, so that $\Delta_{\mathfrak{g},\mathfrak{h}} = o(1)$.} If Assumptions \ref{ass:drule}--\ref{ass:convergence}, \ref{ass:zdim} hold, then $\aslimsup_n \widehat{\mathcal{O}}_{\mathfrak{g},\mathfrak{h}} \subseteq \mathcal{O}$.% and $\aslimsup_n\widehat{\mathcal{S}}_{\mathfrak{g},\mathfrak{h}}\neq\emptyset$.
\end{theorem}

\begin{proof}
First consider the indicator function $\Gamma(B, \widehat{\mathcal{S}}_{\mathfrak{g},\mathfrak{h}})$. Combining our Theorem \ref{thm:setcon} with Proposition 7.4 of \cite{rockafellar2009variational} gives $\aselim \Gamma(\cdot,\widehat{\mathcal{S}}_{\mathfrak{g},\mathfrak{h}}) = \Gamma(\cdot,\mathcal{S})$ relative to $\mathbb{R}^{d\times p}$. Next we claim $\aslim \Gamma(\cdot,\widehat{\mathcal{S}}_{\mathfrak{g},\mathfrak{h}}) = \Gamma(\cdot,\mathcal{S})$ relative to $\mathbb{R}^{d\times p}$. Since Proposition \ref{prop:closed} says the $\widehat{\mathcal{S}}_{\mathfrak{g},\mathfrak{h}}$ are closed, the remark after Theorem 7.10 of \cite{rockafellar2009variational} implies it is sufficient to show that for every $B_0 \in\mathcal{S}$ we have $B_0 \notin \widehat{\mathcal{S}}_{\mathfrak{g},\mathfrak{h}}$ only a finite number of times. A similar argument to the first part of the proof for Theorem \ref{thm:setcon} can be used to show this, and so we omit the details.

Next we note that the level-$(\mathfrak{g},\mathfrak{h})$ FO problem (\ref{eq:fo}) can be written as $\min_B h_n(B) + \Gamma(B, \widehat{\mathcal{S}}_{\mathfrak{g},\mathfrak{h}})$, and the optimization problem (\ref{eqn:ofdr}) can be written as $\min_B h(B) + \Gamma(B, \mathcal{S})$. Now using Theorem 7.46 of \cite{rockafellar2009variational} gives us that
\begin{equation}
\aselim \big(h_n(\cdot) + \Gamma(\cdot, \widehat{\mathcal{S}}_{\mathfrak{g},\mathfrak{h}})\big) = h(\cdot) + \Gamma(\cdot, \mathcal{S}).
\end{equation}
The result now follows by direct application of Proposition 7.30 of \cite{rockafellar2009variational}.
\end{proof}

\begin{remark}
If the optimization problem (\ref{eqn:ofdr}) is infeasible, then we will have $\mathcal{O} = \emptyset$ and $\aslimsup_n \widehat{\mathcal{O}}_{\mathfrak{g},\mathfrak{h}} = \emptyset$, with $\widehat{\mathcal{O}}_{\mathfrak{g},\mathfrak{h}} \neq \emptyset$ only finitely many times. 
\end{remark}

\begin{remark}
We can guarantee under the case of additional assumptions that $\aslimsup_n \widehat{\mathcal{O}}_{\mathfrak{g},\mathfrak{h}} \neq \emptyset$, with $\widehat{\mathcal{O}}_{\mathfrak{g},\mathfrak{h}} = \emptyset$ only finitely many times. In particular, it can be shown this occurs when the underlying problem satisfies some regularity conditions (see Theorem 7.33 of \cite{rockafellar2009variational}) and $\mathcal{O} \neq \emptyset$. If $\mathcal{O}$ consists of a single point, then it can also be shown that $\aslim_n\widehat{\mathcal{O}}_{\mathfrak{g},\mathfrak{h}}=\mathcal{O}$.
\end{remark}

The conclusion ``$\aslimsup_n \widehat{\mathcal{O}}_{\mathfrak{g},\mathfrak{h}} \subseteq \mathcal{O}$'' of the above theorem says all cluster points (i.e., convergent subsequences) as $n$ increases of optimal solutions to the sample-based FO problem (\ref{eq:fo}) belong to the set of optimal solutions to the problem (\ref{eqn:ofdr}) that we initially set out to solve using a sample-based approach. A stronger result is generally not true \cite{rockafellar2009variational}; however, as mentioned above it can be shown that if $\mathcal{O}$ is singleton then we have $\aslim_n\widehat{\mathcal{O}}_{\mathfrak{g},\mathfrak{h}}=\mathcal{O}$.

\subsection{Finite Sample Bounds}

{The solution set consistency results of the previous subsection are asymptotic, and here we provide finite sample bounds that more precisely characterize this consistency. For our FO problem (\ref{eq:fo}), there are really two kinds of consistency that we need to discuss. One kind of consistency is the usual notion of how good the sample-based optimal fair decision rule $\widehat{\delta}_n(x,z) = \widehat{B}_n\cdot\omega(x,z)$ for any $\widehat{B}_n \in \widehat{\mathcal{O}}_{\mathfrak{g},\mathfrak{h}}$ is in terms of minimizing the risk $R(\cdot)$. The second kind of consistency is to quantify how close $\widehat{\delta}_n(X,Z) = \widehat{B}_n\Omega$ is in terms of being independent to $Z$.}

{To study the first kind of consistency, we have to strengthen Assumption \ref{ass:convergence}. Recall this assumption says the approximate risk function composed with the parametric decision rule epi-converges almost surely. We will replace this assumption with a finite sample analog that specifies uniform convergence:}

\begin{assumption}\label{ass:convergence_prime}
{
Let $h_n(B)$ and $h(B)$ be the functions that are defined in Assumption \ref{ass:convergence}. We assume that $\sup_{B\in\mathcal{B}} |h_n(B) - h(B)| \leq r_n$ holds with probability at least $1 - c_n$, where we have that $\lim_n r_n = 0$ and $\lim_n c_n = 0$.}
\end{assumption}

{With the modified assumption and the distance definition (\ref{eqn:mdef}), we can prove finite sample bounds for the FO problem (\ref{eq:fo}). Recall that $\widehat{\delta}_n(x,z) = \widehat{B}_n\cdot\omega(x,z)$ for any $\widehat{B}_n \in \widehat{\mathcal{O}}_{\mathfrak{g},\mathfrak{h}}$ is a sample-based optimal fair decision rule, and $\delta^*(x,z) = B^*\cdot\omega(x,z)$ for any $B^*\in\mathcal{O}$ is an optimal fair decision rule. }

\begin{theorem}
\label{thm:fsbnd}
{Suppose $\Delta_{m,q} = 3(1+\log n)\cdot\mathcal{R}_{m,q}[n]$ and $\mathfrak{g} = \mathfrak{h} = \kappa_1\log n$ (rounded down when non-integer), where $\kappa_1 = (20p\log \alpha + 5\log p + 1)^{-1}$. If Assumptions \ref{ass:drule}, \ref{ass:2norm}, \ref{ass:zdim}, \ref{ass:convergence_prime} hold, then we have: $R(\widehat{\delta}_n) \leq R(\delta^*) + 2r_n$, with probability at least $1 - 6(\kappa_1\log n/n)^2 - 2c_n$; and that 
\begin{equation}
%\mathbb{H}(\widehat{\delta}_n(X,Z), Z) \leq \max\Big\{e^{1/\kappa_2}n^{\kappa_1/\kappa_2}\Delta_{\mathfrak{g},\mathfrak{h}} + \textstyle\frac{1}{2\pi(\kappa_1\log n + 1)}, \frac{\kappa_2(r+d)}{\kappa_1\log n + 1}\Big\}
\mathbb{H}(\widehat{\delta}_n(X,Z), Z) \leq e^{1/\kappa_2}n^{\kappa_1/\kappa_2}\Delta_{\mathfrak{g},\mathfrak{h}} + \textstyle\frac{\kappa_2(r+d)}{\kappa_1\log n + 1}
\end{equation}
with probability at least $1 - 6(\kappa_1\log n/n)^2$, where $\kappa_2 = e\alpha^\rho\lambda p$.}
\end{theorem} 

\begin{proof}
{We begin by bounding the probability that $\widehat{\mathcal{S}}_{\mathfrak{g},\mathfrak{h}} \supseteq \mathcal{S}$. Observe that we can rewrite the complement of this event as
\begin{equation}
\label{eqn:event_fs}
\textstyle\big\{\widehat{\mathcal{S}}_{\mathfrak{g},\mathfrak{h}} \nsupseteq \mathcal{S}\big\} = \bigcup_{m\in[\mathfrak{g}]}\bigcup_{q\in[\mathfrak{h}]}\big\{\displaystyle\sup_{B\in\mathcal{S}}\|\widehat{\Xi}_{m,q}(B)\| > \Delta_{m,q}\big\},
\end{equation}
where for convenience we define the multilinear operators ${\Xi}_{m,q} = {\varphi}_{m,q}-{\nu}_{m,q}$, $\widehat{\Xi}_{m,q} = \widehat{\varphi}_{m,q}-\widehat{\nu}_{m,q}$, $\Phi_{m,q} = \widehat{\varphi}_{m,q}-\varphi_{m,q}$, and $\Psi_{m,q} = \widehat{\nu}_{m,q} - \nu_{m,q}$. Because Theorem \ref{thm:kac} can be rewritten under the assumptions of this theorem as
\begin{equation}
\sup_{B\in\mathcal{S}} \|\varphi_{m,q}(B)-\nu_{m,q}(B)\| = 0 \text{ for } m,q\geq 1,
\end{equation}
then for any $B\in\mathcal{S}$ the application of the triangle inequality yields
\begin{equation}
\begin{aligned}
%\|\widehat{\Xi}_{m,q}(B_0)\| &\leq \sup_{B\in\mathcal{N}}\|\Phi_{m,q}(B)\| + \sup_{B\in\mathcal{N}}\|\Psi_{m,q}(B)\| \\
\|\widehat{\Xi}_{m,q}(B)\| &\leq \|\Xi_{m,q}(B)\| + \|\Phi_{m,q}(B)\| + \|\Psi_{m,q}(B)\| \\
&\leq\lambda^{q/2}\|\Phi_{m,q}\|_\circ + \lambda^{q/2}\|\Psi_{m,q}\|_\circ
\end{aligned}
\end{equation}
Let $\mathcal{G}_{m,q}[n] = (1+\log n)\lambda^{q/2}\mathcal{R}_{m,q}[n]$, and note that the union bound gives
\begin{equation}
\begin{aligned}
\label{eqn:psnfs}
\textstyle\mathbb{P}\big(\widehat{\mathcal{S}}_{\mathfrak{g},\mathfrak{h}} \nsupseteq \mathcal{S}\big) &\textstyle\leq \sum_{m\in [\mathfrak{g}]}\sum_{q\in[\mathfrak{h}]}\mathbb{P}\big(\lambda^{q/2}\|\Phi_{m,q}\|_\circ>\mathcal{G}_{m,q}[n]\big) + \\
&\textstyle\qquad \sum_{m\in [\mathfrak{g}]}\sum_{q\in[\mathfrak{h}]}\mathbb{P}\big(\lambda^{q/2}\|\Psi_{m,q}\|_\circ>2\mathcal{G}_{m,q}[n]\big)\\
&\leq 6(\kappa_1\log n/n)^2
\end{aligned}
\end{equation}
where the last line used Propositions \ref{prop:cphi} and \ref{prop:cpsi}. This implies $\mathbb{P}(\widehat{\mathcal{S}}_{\mathfrak{g},\mathfrak{h}} \supseteq \mathcal{S}) \geq 1-6(\kappa_1\log n/n)^2$, which means that $\mathbb{P}(h_n(\hat{B}_n) \leq h_n(B^*)$ for all $B^*\in\mathcal{O}) \geq \mathbb{P}(\mathcal{O}\subseteq \widehat{\mathcal{S}}_{\mathfrak{g},\mathfrak{h}}) \geq \mathbb{P}(\widehat{\mathcal{S}}_{\mathfrak{g},\mathfrak{h}} \supseteq \mathcal{S}) \geq 1-6(\kappa_1\log n/n)^2$. Combining this with Assumption \ref{ass:convergence_prime} implies $R(\widehat{\delta}_n) \leq R(\delta^*) + 2r_n$, with probability at least $1 - 6(\kappa_1\log n/n)^2 - 2c_n$. This proves the first part of the result.}

{We prove the second part of the result in two steps. As the first step, we consider the event
\begin{equation}
%\mathcal{E} = \textstyle \bigcup_{m\in[\mathfrak{g}]}\bigcup_{q\in[\mathfrak{h}]}\big\{\displaystyle\sup_{B\in\widehat{O}_{\mathfrak{g},\mathfrak{h}}}\|\widehat{\Xi}_{m,q}(B)\| > 2\Delta_{m,q}\big\},
\mathcal{E} = \textstyle \bigcup_{m\in[\mathfrak{g}]}\bigcup_{q\in[\mathfrak{h}]}\big\{\sup_{\widehat{B}_n\in\widehat{O}_{\mathfrak{g},\mathfrak{h}}}\|\Xi_{m,q}(\widehat{B}_n)\| > 2\Delta_{m,q}\big\},
\end{equation}
and note that for $\widehat{B}_n \in \widehat{O}_{\mathfrak{g},\mathfrak{h}}$ application of the triangle inequality yields
\begin{equation}
\begin{aligned}
\|\Xi_{m,q}(\widehat{B}_n)\| &\leq \|\widehat{\Xi}_{m,q}(\widehat{B}_n)\| + \|\Phi_{m,q}(\widehat{B}_n)\| + \|\Psi_{m,q}(\widehat{B}_n)\| \\
&\leq\Delta_{m,q} + \lambda^{q/2}\|\Phi_{m,q}\|_\circ + \lambda^{q/2}\|\Psi_{m,q}\|_\circ
\end{aligned}
\end{equation}
since $\|\widehat{\Xi}_{m,q}(\widehat{B}_n)\| \leq \Delta_{m,q}$ by definition of $\widehat{O}_{\mathfrak{g},\mathfrak{h}}$. Thus the union bound gives
\begin{equation}
\begin{aligned}
\textstyle\mathbb{P}\big(\mathcal{E}\big) &\textstyle\leq \sum_{m\in [\mathfrak{g}]}\sum_{q\in[\mathfrak{h}]}\mathbb{P}\big(\lambda^{q/2}\|\Phi_{m,q}\|_\circ>\mathcal{G}_{m,q}[n]\big) + \\
&\textstyle\qquad \sum_{m\in [\mathfrak{g}]}\sum_{q\in[\mathfrak{h}]}\mathbb{P}\big(\lambda^{q/2}\|\Psi_{m,q}\|_\circ>2\mathcal{G}_{m,q}[n]\big)\\
&\leq 6(\kappa_1\log n/n)^2
\end{aligned}
\end{equation}
where the last line used Propositions \ref{prop:cphi} and \ref{prop:cpsi}. This implies $\mathbb{P}(\|\Xi_{m,q}(\widehat{B}_n)\| \leq 2\Delta_{m,q}\ \text{for } (m,q)\in[\mathfrak{g}]\times[\mathfrak{h}]) \geq 1-6(\kappa_1\log n/n)^2$.}

{We conclude with the second step for our proof of the second part of the result. Because our random variables are bounded, we can use series expansions to express the characteristic functions in the definition (\ref{eqn:mdef}) of $\mathbb{H}(\widehat{B}_n\Omega; Z)$. In particular, we have that
\begin{equation}
J(s,t,\zeta) - P(s,t,\zeta) = \textstyle \sum_{m=1}^\infty\sum_{q=1}^\infty\frac{(\mathfrak{i}\zeta)^{m+q}}{m!\cdot q!}\cdot\langle\Xi_{m,q}(\widehat{B}_n), s^{\otimes m}t^{\otimes q}\rangle.
\end{equation}
We need to bound the modulus of the above. H\"{o}lder's inequality gives us that $|\langle\Xi_{m,q}(\widehat{B}_n), s^{\otimes m}t^{\otimes q}\rangle| \leq (r^m + d^q)^{1/2}\big\|\Xi_{m,q}(\widehat{B}_n)\big\| \leq (r+d)^{(m+q)}\|\Xi_{m,q}(\widehat{B}_n)\|$. In the proof of Propositions \ref{prop:cphi} and \ref{prop:cpsi} we showed $\|\psi_{m,q}(\widehat{B}_n)\|_\circ \leq \alpha^{m+\rho q}p^{q/2}$ and $\|\nu_{m,q}(\widehat{B}_n)\|_\circ \leq \alpha^{m+\rho q}p^{q/2}$. Thus $\|\Xi_{m,q}(\widehat{B}_n)\| \leq 2\alpha^{m+\rho q}(\lambda p)^{q/2}$, which we will use for $m = \mathfrak{g}+1$ and $q = \mathfrak{h}+1$. We next use these bounds with a standard argument (see for instance Section 26 of \cite{billingsley1995probability}) that first uses Jensen's inequality and then uses the elementary inequality for the complex exponential that $|\exp(i\zeta) - \sum_{m=0}^\mathfrak{g}(i\zeta)^m/m!| \leq |\zeta|^{\mathfrak{g}+1}/(\mathfrak{g}+1)!$. This two step argument implies that for $|\zeta| \leq T$ we have
\begin{multline}
\big|J(s,t,\zeta) - P(s,t,\zeta) - \textstyle \sum_{m=1}^\mathfrak{g}\sum_{q=1}^\mathfrak{h}\frac{(\mathfrak{i}\zeta)^{m+q}}{m!\cdot q!}\cdot\langle\Xi_{m,q}(\widehat{B}_n), s^{\otimes m}t^{\otimes q}\rangle\big| \\\textstyle \leq \frac{2}{(\mathfrak{g}+1)!\cdot(\mathfrak{h}+1)!}\cdot\alpha^{\mathfrak{g}+1+\rho (\mathfrak{h}+1)}\cdot(\lambda p)^{(\mathfrak{h}+1)/2}\cdot((r+d)T)^{\mathfrak{g}+\mathfrak{h}+2}.
\end{multline}
Using the reverse triangle inequality implies the modulus is bounded by
\begin{multline}
\label{eqn:mdbndbdn}
\big|J(s,t,\zeta) - P(s,t,\zeta)\big| \leq \textstyle \sum_{m=1}^\mathfrak{g}\sum_{q=1}^\mathfrak{h}\frac{((r+d)\zeta)^{m+q}}{m!\cdot q!}\cdot\big\|\Xi_{m,q}(\widehat{B}_n)\big\| + \\
\textstyle\frac{2}{(\mathfrak{g}+1)!\cdot(\mathfrak{h}+1)!}\cdot(\alpha^\rho\lambda p(r+d)T)^{\mathfrak{g}+\mathfrak{h}+2}
\end{multline}
for all $|\zeta|\leq T$. Combining this with the first step of the proof for the second part of the result implies that with probability at least $1-6(\kappa_1\log n/n)^2$ we have for $|\zeta|\leq T$ that
\begin{equation}
\big|J(s,t,\zeta) - P(s,t,\zeta)\big| \leq \textstyle 2\exp((r+d)T)\cdot\Delta_{\mathfrak{g},\mathfrak{h}} + \textstyle\frac{2(\alpha^\rho\lambda p(r+d)T)^{\mathfrak{g}+\mathfrak{h}+2}}{(\mathfrak{g}+1)!\cdot(\mathfrak{h}+1)!}
\end{equation}
where the first term follows from the exponential series. If we choose that $T = (\kappa_1\log n + 1)/(\kappa_2(r+d))$, then using the standard error bound $(\mathfrak{g}+1)! \geq (2\pi(\mathfrak{g}+1))^{1/2}((\mathfrak{g}+1)/e)^{\mathfrak{g}+1}$ for Stirling's approximation leads to
\begin{equation}
%\big|J(s,t,\zeta) - P(s,t,\zeta)\big| \leq \textstyle 8n^{-1/4}\exp\big((r+d)\zeta\big) + \frac{2(\alpha^\rho\lambda p(r+d)\zeta)^{\mathfrak{g}+\mathfrak{h}+2}}{(\mathfrak{g}+1)!\cdot(\mathfrak{h}+1)!}\cdot\exp\big(\alpha^\rho\zeta\sqrt{\lambda p(r+d)}\big)
\big|J(s,t,\zeta) - P(s,t,\zeta)\big| \leq \textstyle 2e^{1/\kappa_2}n^{\kappa_1/\kappa_2}\cdot\Delta_{\mathfrak{g},\mathfrak{h}} + \textstyle\frac{1}{\pi(\kappa_1\log n + 1)},
\end{equation}
which holds with probability at least $1-6(\kappa_1\log n/n)^2$. The second result follows by applying this bound and choice of $T$ to the definition (\ref{eqn:mdef}).}
\end{proof}

\begin{remark}
{The result of the above theorem can be interpreted as implying that $|R(\widehat{\delta}_n)-R(\delta^*)| = O(r_n)$ and that $\mathbb{H}(\widehat{\delta_n}(X,Z); Z) = O(1/\log n)$, with high probability. This is because we have that $n^{\kappa_1/\kappa_2}\Delta_{\mathfrak{g},\mathfrak{h}} = o(1/\log n)$ under the conditions specified in the above theorem.}
\end{remark}

\subsection{Approximate Independence}

\label{sec:ai}

Let $U\in\mathbb{R}^p$ and $V\in\mathbb{R}^d$ be random vectors, and consider the quantity
{
\begin{equation}
\begin{aligned}
\mathbb{M}(U;V) = \inf\ &\epsilon\\
\text{s.t. }&\|\mathbb{E}\big(U^{\otimes m}V^{\otimes q}\big) - \mathbb{E}\big(U^{\otimes m}\big)\otimes\mathbb{E}\big(V^{\otimes q}\big)\big\| \leq \epsilon^{m+q}\cdot m!\cdot q!,\\
&\qquad\text{for } m,q\geq 1.
\end{aligned}
\end{equation}}
We call the quantity $\mathbb{M}(U;V)$ the \emph{mutual majorization} of $U$ and $V$, and the choice of this name is meant to draw a direct analogy to mutual information. The mutual majorization is nonnegative $\mathbb{M}(U;V) \geq 0$ and symmetric $\mathbb{M}(U;V) = \mathbb{M}(V;U)$ by definition. One utility of this definition for the mutual majorization is that it bounds approximate independence.

\begin{proposition}
\label{prop:mm}
{Let $M_{(U,V)}(s,t) = \mathbb{E}\exp(\langle s,U\rangle + \langle t,V\rangle)$ be the moment generating function for the multivariate random variable $(U,V)$ where $U\in\mathbb{R}^p$ and $V\in\mathbb{R}^d$. Suppose that $M_{(U,V)}(s,t)$ is finite in a neighborhood of the origin. If $\mathbb{M}(U; V) \leq \epsilon$, then $\mathbb{H}(U;V) \leq 2(\epsilon\cdot(r+d))^{2/3}$ when $\epsilon\cdot(r+d)\leq 1$.}
\end{proposition}

%If $U$ and $V$ are independent, then we have that $\mathbb{M}(U;V) = 0$ by Theorem \ref{thm:kac}. To show the converse, we prove its contrapositive: If $U$ and $V$ are dependent, then $\mathbb{M}(U;V) > 0$ since Theorem \ref{thm:kac} implies that we have $\|\mathbb{E}\big(U^{\otimes m}V^{\otimes q}\big) - \mathbb{E}\big(U^{\otimes m}\big)\otimes\mathbb{E}\big(V^{\otimes q}\big)\big\| > 0$ for some $m,q\geq 1$.
\begin{proof}
{We need to bound the modulus of $J(s,t,\zeta) - P(s,t,\zeta)$. Because $M_{(U,V)}(s,t)$ exists in a neighborhood of the origin, this means the characteristic functions can be represented as infinite series. Thus we have
\begin{multline}
\label{eqn:geobnd}
\big|J(s,t,\zeta) - P(s,t,\zeta)\big| = \\
\textstyle \big|\sum_{m=1}^\infty\sum_{q=1}^\infty\frac{(\mathfrak{i}\zeta)^{m+q}}{m!\cdot q!}\cdot\langle\mathbb{E}\big(U^{\otimes m}V^{\otimes q}\big) - \mathbb{E}\big(U^{\otimes m}\big)\otimes\mathbb{E}\big(V^{\otimes q}\big), s^{\otimes m}t^{\otimes q}\rangle\big| \leq \\
\textstyle\sum_{m=1}^\infty\sum_{q=1}^\infty(\epsilon(r+d)\zeta)^{m+q} = (\tau/(1-\tau))^2.
\end{multline}
when $\tau = \epsilon(r+d)\zeta \in [0,1)$. If we choose $T^{-1} = \epsilon(r+d) + (\epsilon(r+d))^{2/3}$, then the result follows by applying this bound to the definition (\ref{eqn:mdef}).}
\end{proof}

The implication of this result is we can use mutual majorization {as a surrogate for} approximate independence. We thus define an optimization problem that chooses an optimal $\epsilon$-approximately-fair decision rule by solving
\begin{equation}
\label{eqn:Lfdr}
\textstyle\delta^*(x,z) \in \arg\min_{\delta(\cdot,\cdot)}\big\{R(\delta)\ \big|\ \mathbb{M}(\delta(X,Z); Z) \leq \epsilon\big\}.
\end{equation}
The level-$(\mathfrak{g},\mathfrak{h})$ FO problem (\ref{eq:fo}) with appropriate choice of $\Delta_{m,q}$ is a  statistically well-behaved, sample-based approximation of the above problem.
In order to be able to discuss this, we first define the set
\begin{equation}
\mathcal{S}(\epsilon) = \big\{B \in\mathcal{B} : \mathbb{M}(B\Omega; Z) \leq \epsilon\big\}
\end{equation}
and the solution set
\begin{equation}
\mathcal{O}(\epsilon) = \arg\min_{B}\big\{R(B\cdot\omega(x,z))\ \big|\ B \in \mathcal{S}(\epsilon)\big\}.
\end{equation}
These are respectively the feasible set and solution set of the optimization problem (\ref{eqn:Lfdr}), which chooses an optimal $\epsilon$-approximately-fair decision rule when the underlying distributions are exactly known.

%\begin{proposition}
%\label{prop:closed2}
%If Assumption \ref{ass:drule} holds, then the set $\mathcal{S}(\epsilon)$ is closed.
%\end{proposition}
%
%\begin{proof}
%Consider any convergent sequence $B_k\in\mathbb{R}^{d\times p}$ with $B_k \in \mathcal{S}(\epsilon)$ and $\lim_kB_k=B_0$. We have $\sup_{m,q\geq 1}\|\varphi_{m,q}(B_k) - \nu_{m,q}(B_k)\|^{1/(m+q)}\leq \epsilon$ by definition of $\mathcal{S}(\epsilon)$. We can rewrite this condition as
%\begin{equation}
%\big\|\varphi_{m,q}(B_k) - \nu_{m,q}(B_k)\big\|^{1/(m+q)}\leq \epsilon, \text{for } m,q\geq 1.
%\end{equation}
%But the $\varphi$ and $\nu$ are continuous since they are multlinear operators on Euclidean space. This means $\lim_k \varphi_{m,q}(B_k) = \varphi_{m,q}(B_0)$ and $\lim_k \nu_{m,q}(B_k) = \nu_{m,q}(B_0)$ for $m,q\geq 1$. As a result we have
%\begin{multline}
%\textstyle\big\|\varphi_{m,q}(B_0)-\nu_{m,q}(B_0)\big\|^{1/(m+q)} = \\
%\textstyle\lim_k \big\|{\varphi}_{m,q}(B_k)-{\nu}_{m,q}(B_k)\big\|^{1/(m+q)}\leq\epsilon, \text{for } m,q\geq 1.
%\end{multline}
%This means $B_0\in\mathcal{S}(\epsilon)$ by definition, which proves that $\mathcal{S}(\epsilon)$ is closed.
%\end{proof}

\begin{theorem}
\label{thm:apphi}
{Let $\Delta_{m,q} = \epsilon^{m+q}\cdot m!\cdot q! + 3(1+\log n)\cdot\mathcal{R}_{m,q}[n]$ and suppose $\mathfrak{g} = \mathfrak{h} = O(\log n)$, such that $\log n\cdot\mathcal{R}_{\mathfrak{g},\mathfrak{h}}[n] = o(1)$.} If Assumption \ref{ass:drule} holds, then $\mathcal{S}(\epsilon)$ is closed. If Assumptions \ref{ass:2norm}, \ref{ass:zdim} also hold, then $\aslim_n \widehat{\mathcal{S}}_{\mathfrak{g},\mathfrak{h}} = \mathcal{S}(\epsilon)$. If Assumption \ref{ass:convergence} also holds, then $\aslimsup_n \widehat{\mathcal{O}}_{\mathfrak{g},\mathfrak{h}} \subseteq \mathcal{O}(\epsilon)$.
\end{theorem}

\begin{remark}
The proof is omitted because it is a straightforward modification of the proofs for Proposition \ref{prop:closed} and Theorems \ref{thm:setcon} and \ref{thm:opcon}.
\end{remark}

\begin{remark}
Recall we already proved $\widehat{\mathcal{S}}_{\mathfrak{g},\mathfrak{h}}$ is closed in Proposition \ref{prop:closed}.
\end{remark}

{We can also prove a finite sample version of the above result, which shows that consistency holds for sample-based analogs of (\ref{eqn:Lfdr}).}

\begin{theorem}
{Suppose $\Delta_{m,q} = \epsilon^{m+q}\cdot m!\cdot q! + 3(1+\log n)\cdot\mathcal{R}_{m,q}[n]$ and $\mathfrak{g} = \mathfrak{h} = \kappa_1\log n$ (rounded down when non-integer), where $\kappa_1 = (20p\log \alpha + 5\log p + 1)^{-1}$. If Assumptions \ref{ass:drule}, \ref{ass:2norm}, \ref{ass:zdim}, \ref{ass:convergence_prime} hold, then: $R(\widehat{\delta}_n) \leq R(\delta^*) + 2r_n$, with probability at least $1 - 6(\kappa_1\log n/n)^2 - 2c_n$; and when $\epsilon\cdot(r+d)\leq 1$ then we also have that 
\begin{multline}
\label{eqn:mmhbnd}
\mathbb{H}(\widehat{\delta}_n(X,Z), Z) \leq 2(\epsilon\cdot(r+d))^{2/3} + \\ \kappa_3\cdot(1+\log n)\cdot\mathcal{R}_{\mathfrak{g},\mathfrak{h}}[n] + \textstyle\frac{1}{(\mathfrak{g}+1)!\cdot(\mathfrak{h}+1)!}\cdot\kappa_4^{\ \mathfrak{g}+\mathfrak{h}+2}
\end{multline}
with probability at least $1 - 6(\kappa_1\log n/n)^2$, where the constants used above are $\kappa_3 = 3\exp(1/\epsilon)$ and $\kappa_4 = \alpha^\rho\lambda p/\epsilon$.}
\end{theorem} 

\begin{proof}
{The proof is identical to that of Theorem \ref{thm:fsbnd}, up to (\ref{eqn:mdbndbdn}). (This means the first part of the current result is proved the same way as in Theorem \ref{thm:fsbnd}.) To complete the proof we first bound (\ref{eqn:mdbndbdn}) using the $\Delta_{m,q}$ in the hypothesis of this theorem. Comparing to (\ref{eqn:geobnd}), we get with probability at least $1-6(\kappa_1\log n/n)^2$ we have that
\begin{multline}
\big|J(s,t,\zeta) - P(s,t,\zeta)\big| \leq \textstyle 2(\tau/(1-\tau))^2 + \\6\exp((r+d)T)\cdot(1+\log n)\cdot\mathcal{R}_{\mathfrak{g},\mathfrak{h}}[n] + \textstyle\frac{2(\alpha^\rho\lambda p(r+d)T)^{\mathfrak{g}+\mathfrak{h}+2}}{(\mathfrak{g}+1)!\cdot(\mathfrak{h}+1)!}
\end{multline}
when $\tau=\epsilon(r+d)T \in [0,1)$ and for all $|\zeta|\leq T$. If we choose $T^{-1} = \epsilon(r+d) + (\epsilon(r+d))^{2/3}$, then the second result follows by applying this bound and choice of $T$ to the definition (\ref{eqn:mdef}).}
\end{proof}

\begin{remark}
{The result of the above theorem implies that we have $\limsup_n \mathbb{H}(\widehat{\delta}_n(X,Z), Z) \leq 2(\epsilon\cdot(r+d))^{2/3}$ because the second and third terms in (\ref{eqn:mmhbnd}) converge to zero under the conditions of the above theorem.}
\end{remark}

\section{Hierarchy Consistency for Unbounded Random Variables}
\label{sec:ubrv}
{In the previous section, we proved consistency of the FO problem (\ref{eq:fo}) when the involved random variables are bounded. However, the underlying generalizations of Kac's Theorem, which relate moment conditions to independence, also apply to unbounded random variables whose moment generating function is finite about the origin (Theorem \ref{thm:kac}) and to unbounded random variables with some number of finite moments but not necessarily with a moment generating function that exists near the origin (Theorem \ref{thm:kac2}).} 

{In this section we show that the sample-based constraints of the FO problem (\ref{eq:fo}) are statistically well-behaved analogs of the independence constraint in (\ref{eqn:ofdr}) when the involved random variables are unbounded. We will consider two cases. The first is when the involved random variables are sub-Gaussian, and the second is for random variables with finite moments.}

\subsection{Sub-Gaussian Case}

{Our first task is to relax Assumption \ref{ass:zdim}, which assumed the involved random variables are bounded. There is a subtlety in relaxing this assumption for sub-Gaussian random variables.}

\begin{example}
{Let $X \sim \mathcal{N}(0,1)$ be a standard normal and define $U = X^k$ for some $k \in \mathbb{Z}_+$. Then $U$ is sub-Gaussian for $k=1$, but $U$ is not sub-Gaussian for $k \geq 2$. Furthermore, the moment generating function for $U$ is finite in a neighborhood about the origin only for $k \in \{1,2,4\}$, or restated the moment generating function is not well-defined for $k = 3$ or $k \geq 5$ \cite{berg1988cube}.}
\end{example}

{The consequence of this example is that if we want to consider a sub-Gaussian case, then we need to specify that the joint distribution of $(Z,\Omega)$ is sub-Gaussian rather than assuming that $(X,Z)$ is sub-Gaussian. Thus, in lieu of Assumption \ref{ass:zdim} we make the following assumption:}

\begin{assumption}
\label{ass:zdim_subgau}
{The (joint) random variable $(Z,\Omega)$ is sub-Gaussian (\ref{eqn:subgaualt}) with $M \geq 1$ and $\sigma^2 \geq 0$, and the random variable $Z$ has dimensions $Z\in\mathbb{R}^r$.}
\end{assumption}

{With this assumption, we can now study consistency of the FO problem (\ref{eq:fo}) when the involved random variables are sub-Gaussian. We first prove a result on the convergence of the tensor moment estimates.}

\begin{proposition}
\label{prop:cphi_subgau}
{If Assumptions \ref{ass:drule}, \ref{ass:zdim_subgau} hold, then we have
\begin{equation}
\begin{aligned}
&\textstyle\mathbb{P}\big(\|\widehat{\varphi}_{m,q}-\varphi_{m,q}\|_\circ > \hphantom{2}\mathcal{C}_{m,q}[n]\cdot\gamma\big) &\leq \hphantom{4}\big(\gamma^6\cdot n^2\big)^{-1}\\
%&\textstyle\mathbb{P}\big(\|\rlap{$\hspace{0.09em}\widehat{\nu}$}\hphantom{\widehat{\varphi}}_{m,q}-\rlap{$\hspace{0.09em}\nu$}\hphantom{\varphi}_{m,q}\|_\circ > 4\mathcal{C}_{m,q}[n]\cdot\gamma\big) \leq 4\big(\gamma^6\cdot n^2\big)^{-1}
&\textstyle\mathbb{P}\big(\|\rlap{$\hspace{0.09em}\widehat{\nu}$}\hphantom{\widehat{\varphi}}_{m,q}-\rlap{$\hspace{0.09em}\nu$}\hphantom{\varphi}_{m,q}\|_\circ > 2\mathcal{C}_{m,q}[n]\cdot\gamma + \mathcal{C}_{m,q}[n]^2\cdot\gamma^2\big) &\leq 4\big(\gamma^6\cdot n^2\big)^{-1}\end{aligned}
\end{equation}
for $\mathcal{C}_{m,q}[n] = [\frac{e^2M^22^{7}5^3}{\pi n}\cdot (1+4q)^{dp}(rm^3)^m(dq^3)^q(24\sigma^2/e)^{3m+3q}]^{1/6}$.}
\end{proposition}

\begin{proof}
{The proof for the first part of this result follows the same steps as the proof of Proposition \ref{prop:cphi} up to and including (\ref{eqn:pfref}). Next observe that
\begin{equation}
\label{eqn:varbnd_subgau}
\begin{aligned}
\mathbb{E}\big(\|\Phi\|_\circ^{\ 6}\big) &\leq \mathbb{E}\big(2^6\max_{i,u_k,v_k}\textstyle|\langle \Phi(T_i), \bigotimes_{k=1}^m u_k \bigotimes_{k=1}^q v_k\rangle|^6\big)\\
&\leq \textstyle 2^6\cdot\sum_{i,u_k,v_k}\mathbb{E}\big(\langle \Phi(T_i), \bigotimes_{k=1}^m u_k \bigotimes_{k=1}^q v_k\rangle^6\big)
\end{aligned}
\end{equation}
We seek to bound the term on the right-hand side. For convenience, define $S_i = \langle Z^{\otimes m}(T_i\Omega)^{\otimes q}, \bigotimes_{k=1}^m u_k \bigotimes_{k=1}^q v_k\rangle$ and $V_i = S_i - \mathbb{E}(S_i)$. Next observe that the Marcinkiewicz-Zygmund inequality \cite{rio2009moment} implies that
\begin{equation}
\textstyle\mathbb{E}\big(\langle \Phi(T_i), \bigotimes_{k=1}^m u_k \bigotimes_{k=1}^q v_k\rangle^6\big) \leq 5^3\cdot\mathbb{E}(V_i^{\ 6})/n^3.
\end{equation}
We next have to bound the expectation on the right. Consider
\begin{equation}
\begin{aligned}
\textstyle \mathbb{E}\big(V_i^{\ 6}\big) &\leq 2\mathbb{E}\big(\epsilon^6S_i^{\ 6}\big)\\
&\leq \textstyle 2\cdot\big[\mathbb{E}\big(\langle u_k, Z\rangle^{12m}\big)\cdot\mathbb{E}\big(\langle v_k,\hspace{0.4em}T_i\Omega\rangle^{12q}\big)\big]^{1/2}\\
&\leq \textstyle 2\cdot\big[\mathbb{E}\big(\langle u_k, Z\rangle^{12m}\big)\cdot\mathbb{E}\big(\langle \rlap{$
T_i$}\hphantom{T}^\mathsf{T}v_k, \Omega\rangle^{12q}\big)\big]^{1/2}\\
&\leq\textstyle 2M\sigma^{6m+6q}\cdot\big[\frac{(12m)!\cdot(12q)!}{(6m)!\cdot(6q)!}\big]^{1/2}\\
&\leq\textstyle 2eM\cdot (24\sigma^2/e)^{3m+3q}\cdot m^{3m}\cdot q^{3q}/\sqrt{\pi}\\
\end{aligned}
\end{equation}
where the first line follows by a stochastic symmetrization step (i.e., Jensen's inequality, multiplication with i.i.d. Rademacher random variables $\epsilon$ having distribution $\mathbb{P}(\epsilon = \pm 1) = \frac{1}{2}$, using the triangle inequality, and concluded by Jensen's inequality), the second line follows by the Cauchy-Schwarz inequality, the third line uses a matrix transpose $\rlap{$T_i$}\hphantom{T}^\mathsf{T}$, the fourth line follows by (\ref{eqn:subgaumom}) because $\|\rlap{$T_i$}\hphantom{T}^\mathsf{T}v_k\|_2 \leq \|v_k\|_2$ since $T_i = M(t_i)$ for $t_i\in \mathbb{S}^{dp-1}$, and the fifth line uses Stirling's approximation. Combining the above with (\ref{eqn:varbnd_subgau}) gives
\begin{equation}
%\textstyle\mathbb{E}\big(\|\Phi\|_\circ^{\ 4}\big) \leq 2^{7}5^3\cdot eM(1+4q)^{dp}(rm^3)^m(dq^3)^q(24\sigma^2/e)^{3m+3q}/\sqrt{\pi n^6}.
\textstyle\mathbb{E}\big(\|\Phi\|_\circ^{\ 6}\big) \leq \frac{eM2^{7}5^3}{\sqrt{\pi}n^3}\cdot (1+4q)^{dp}(rm^3)^m(dq^3)^q(24\sigma^2/e)^{3m+3q}.
\end{equation}
Let $\kappa = (eM/\sqrt{\pi})^{1/6}$ and note that Markov's inequality implies
\begin{equation}
\label{eqn:msg}
\mathbb{P}\big(\|\widehat{\varphi}_{m,q}-\varphi_{m,q}\|_\circ > \mathcal{C}_{m,q}[n]\cdot\gamma/\kappa\big) \leq \big(\gamma^6\cdot n^2\big)^{-1}.
\end{equation}
The first result now follows by nothing that $\kappa > 1$.}

{The proof for the second part of this result proceeds slightly differently than the proof of Proposition \ref{prop:cpsi}. Recall that we have $\widehat{\varphi}_{m,0}(B) = \mathbb{E}_n(Z^{\otimes m})$, $\varphi_{m,0}(B) = \mathbb{E}(Z^{\otimes m})$, $\widehat{\varphi}_{0,q}(B) = \mathbb{E}_n((B\Omega)^{\otimes q})$, and $\varphi_{0,q}(B) = \mathbb{E}((B\Omega)^{\otimes q})$. Let $\kappa = eM/\sqrt{\pi}$, and observe that Jensen's inequality implies
\begin{equation}
\|\varphi_{m,0}\|_\circ^{\ 6} \leq \mathbb{E}\big(\langle u_k, Z\rangle^{6m}\big) \leq eM\cdot(12\sigma^2/e)^{3m}m^{3m}/\sqrt{\pi} \leq\mathcal{C}_{m,0}[n]^6/\kappa.
\end{equation}
A similar calculation shows that for some $T = M(t)$ with $t\in\mathbb{S}^{dp-1}$ we have
\begin{equation}
%\|\varphi_{0,q}\|_\circ^{\ 6} \leq \mathbb{E}\big(\langle v_k, T\Omega\rangle^{6m}\big) \leq \mathbb{E}\big(\langle T^\mathsf{T}v_k, \Omega\rangle^{6m}\big) \leq eM\cdot(12\sigma^2/e)^{3q}q^{3q}/\sqrt{\pi}.
\|\varphi_{0,q}\|_\circ^{\ 6} \leq \mathbb{E}\big(\langle T^\mathsf{T}v_k, \Omega\rangle^{6m}\big) \leq eM\cdot(12\sigma^2/e)^{3q}q^{3q}/\sqrt{\pi} \leq \mathcal{C}_{0,q}[n]^6/\kappa.
\end{equation}
Next note that two applications of the triangle inequality imply
\begin{multline}
\label{eqn:tisug}
\|\widehat{\nu}_{m,q} - \nu_{m,q}\|_\circ \leq \|\varphi_{m,0}\|_\circ\cdot\|\widehat{\varphi}_{0,q} - \varphi_{0,q}\|_\circ + \\
\|\varphi_{0,q}\|_\circ\cdot\|\widehat{\varphi}_{m,0}-\varphi_{m,0}\|_\circ + \|\widehat{\varphi}_{m,0}-\varphi_{m,0}\|_\circ\cdot\|\widehat{\varphi}_{0,q} - \varphi_{0,q}\|_\circ.
\end{multline}
Hence the union bound implies
\begin{equation}
\textstyle\mathbb{P}\big(\|\widehat{\nu}_{m,q} - \nu_{m,q}\|_\circ > 2\mathcal{C}_{m,q}[n]\cdot\gamma + \mathcal{C}_{m,q}[n]^2\cdot\gamma^2\big) \leq \mathsf{I} + \mathsf{II} + \mathsf{III} + \mathsf{IV}
\end{equation}
for terms we define next. To bound these terms, we use (\ref{eqn:msg}). Observe that $\mathsf{I} = \mathbb{P}\big(\|\widehat{\varphi}_{0,q} - \varphi_{0,q}\|_\circ > \mathcal{C}_{m,q}[n]\cdot\gamma\big)\leq(\gamma^6\cdot n^2\big)^{-1}$, that $\mathsf{II} = \mathbb{P}\big(\|\widehat{\varphi}_{m,0} - \varphi_{m,0}\|_\circ > \mathcal{C}_{m,q}[n]\cdot\gamma\big)\leq(\gamma^6\cdot n^2\big)^{-1}$, that
\begin{equation}
\begin{aligned}
\mathsf{III} &= \textstyle \mathbb{P}\big(\mathcal{C}_{m,0}[n]\cdot\|\widehat{\varphi}_{0,q} - \varphi_{0,q}\|_\circ > \kappa\cdot\mathcal{C}_{m,q}[n]\cdot\gamma\big) \\
&\leq \mathbb{P}\big(\|\widehat{\varphi}_{0,q} - \varphi_{0,q}\|_\circ > \mathcal{C}_{0,q}[n]\cdot\gamma/\kappa\big)\\
&\leq(\gamma^6\cdot n^2\big)^{-1}
\end{aligned}
\end{equation}
and that
\begin{equation}
\begin{aligned}
\mathsf{IV} &= \textstyle \mathbb{P}\big(\mathcal{C}_{0,q}[n]\cdot\|\widehat{\varphi}_{m,0} - \varphi_{m,0}\|_\circ > \kappa\cdot\mathcal{C}_{m,q}[n]\cdot\gamma\big)\\
&\leq \mathbb{P}\big(\|\widehat{\varphi}_{m,0} - \varphi_{m,0}\|_\circ > \mathcal{C}_{m,0}[n]\cdot\gamma/\kappa\big)\\
&\leq(\gamma^6\cdot n^2\big)^{-1}
\end{aligned}
\end{equation}
Combining the above with (\ref{eqn:tisug}) gives the second result.}
\end{proof}

{With the above result on concentration of the moment tensors in the sub-Gaussian case, we can now state our results about consistency of the FO problem (\ref{eq:fo}). We start with a result on asymptotic consistency.}

\begin{theorem}
\label{thm:apphi_subgau}
{Suppose $\Delta_{m,q} = 3\cdot\mathcal{C}_{m,q}[n] + \mathcal{C}_{m,q}[n]^2$, and suppose we have $\mathfrak{g} = \mathfrak{h} = O(\sqrt{\log n})$, such that $\Delta_{\mathfrak{g},\mathfrak{h}} = o(1)$. If Assumption \ref{ass:drule} holds, then $\mathcal{S}(\epsilon)$ is closed. If Assumptions \ref{ass:2norm}, \ref{ass:zdim_subgau} also hold, then $\aslim_n \widehat{\mathcal{S}}_{\mathfrak{g},\mathfrak{h}} = \mathcal{S}(\epsilon)$. If Assumption \ref{ass:convergence} also holds, then $\aslimsup_n \widehat{\mathcal{O}}_{\mathfrak{g},\mathfrak{h}} \subseteq \mathcal{O}(\epsilon)$.}
\end{theorem}

\begin{remark}
{The proof is omitted because it is a straightforward modification of the proofs for Proposition \ref{prop:closed} and Theorems \ref{thm:setcon} and \ref{thm:opcon}.}
\end{remark}

{Our next result provides a finite sample characterization of the consistency of solutions to the FO problem (\ref{eq:fo}) in this sub-Gaussian case.}

\begin{theorem}
\label{thm:fsbnd_subgau}
{Suppose $\Delta_{m,q} = 3\cdot\mathcal{C}_{m,q}[n] + \mathcal{C}_{m,q}[n]^2$, and suppose that $\mathfrak{g} = \mathfrak{h} = \sqrt{\kappa_5\log n}$ (rounded down when non-integer), where we have $\kappa_5 = (\max\{5, 20dp + 5\log(rd) + 30\log(24\sigma^2)\})^{-1}$. If Assumptions \ref{ass:drule}, \ref{ass:2norm}, \ref{ass:convergence_prime}, \ref{ass:zdim_subgau} hold, then we have: $R(\widehat{\delta}_n) \leq R(\delta^*) + 2r_n$, with probability at least $1 - 6\kappa_5\log n/n^2 - 2c_n$; and for $n \geq 3 > e$ we have
\begin{equation}
\mathbb{H}(\widehat{\delta}_n(X,Z), Z) \leq e^{1/\kappa_6}n^{\kappa_5/\kappa_6}\Delta_{\mathfrak{g},\mathfrak{h}} + \kappa_6(r+d)\cdot[\sqrt{\kappa_5\log n} + 1]^{-1/2}
\end{equation}
with probability at least $1 - 6\kappa_5\log n/n^2$, where the constant in the above is $\kappa_6 = \max\{4,2\sqrt{e}\sigma M\}$.}
\end{theorem} 

\begin{remark}
{ The proof is omitted because it is a straightforward modification of the proof for Theorem \ref{thm:fsbnd} after noting that Cauchy-Schwarz and Jensen's inequalities imply $\|\Xi_{m,q}(\widehat{B}_n)\| \leq 2eM\cdot(\sqrt{4\sigma^2/e})^{m+q}m^{m/2}q^{q/2}/\sqrt{\pi}$.}
\end{remark}

\begin{remark}
{The result of the above theorem can be interpreted as implying that $|R(\widehat{\delta}_n)-R(\delta^*)| = O(r_n)$ and $\mathbb{H}(\widehat{\delta_n}(X,Z); Z) = O((\log n)^{-1/4})$, with high probability. This is because we have $n^{\kappa_5/\kappa_6}\Delta_{\mathfrak{g},\mathfrak{h}} = o((\log n)^{-1/4})$ under the conditions specified in the above theorem.}
\end{remark}

\subsection{Finite Moments Case}

{Our last set of results concern relaxing Assumption \ref{ass:zdim} to the case of unbounded random variables with finite moments. Instead of Assumptions \ref{ass:zdim} or \ref{ass:zdim_subgau}, we make the following assumption:}

\begin{assumption}
\label{ass:zdim_fm}
{Consider the (joint) random variable $(Z,\Omega)$, and define $M_{m,q} = \sup_{(s,t)\in\mathbb{S}^{p+d-1}}\mathbb{E}(\langle s, Z\rangle^{m}\langle t,\Omega\rangle^{q})$. Assume that any moments used in the results are finite, and the random variable $Z$ has dimensions $Z\in\mathbb{R}^r$.}
\end{assumption}

{With this assumption, we can now study approximate consistency of the FO problem (\ref{eq:fo}) when the involved random variables have finite moments. We first prove a result on the convergence of the tensor moment estimates.}

\begin{proposition}
\label{prop:cphi_fm}
{If Assumptions \ref{ass:drule}, \ref{ass:zdim_fm} hold, then we have
\begin{equation}
\begin{aligned}
&\textstyle\mathbb{P}\big(\|\widehat{\varphi}_{m,q}-\varphi_{m,q}\|_\circ > \hphantom{2}\mathcal{Y}_{m,q}[n]\cdot\gamma\big) &\leq \hphantom{4}\big(\gamma^2\cdot n\big)^{-1}\\
%&\textstyle\mathbb{P}\big(\|\rlap{$\hspace{0.09em}\widehat{\nu}$}\hphantom{\widehat{\varphi}}_{m,q}-\rlap{$\hspace{0.09em}\nu$}\hphantom{\varphi}_{m,q}\|_\circ > 4\mathcal{C}_{m,q}[n]\cdot\gamma\big) \leq 4\big(\gamma^6\cdot n^2\big)^{-1}
&\textstyle\mathbb{P}\big(\|\rlap{$\hspace{0.09em}\widehat{\nu}$}\hphantom{\widehat{\varphi}}_{m,q}-\rlap{$\hspace{0.09em}\nu$}\hphantom{\varphi}_{m,q}\|_\circ > 2\mathcal{Y}_{m,q}[n]\cdot\gamma + \mathcal{Y}_{m,q}[n]^2\cdot\gamma^2\big) &\leq 4\big(\gamma^2\cdot n\big)^{-1}\end{aligned}
\end{equation}
for $\mathcal{Y}_{m,q}[n] = (8/n)^{1/2}\cdot(M_{4m,0}\cdot M_{0,4q})^{1/4}$.}
\end{proposition}

\begin{remark}
{ The proof is omitted because it is a straightforward modification of the proof for Proposition \ref{prop:cphi_subgau}.}
\end{remark}

{We conclude with a result about the finite sample behavior of solutions to the FO problem (\ref{eq:fo}) when the involved random variables have finite moments. The difference in the hypothesis of this result, relative to the results for the cases of bounded or sub-Gaussian random variables, is that here we will characterize solutions when $\mathfrak{g}$ and $\mathfrak{h}$ are held as fixed constants. In the previous results, we assumed $\mathfrak{g}$ and $\mathfrak{h}$ were increasing with $n$.}

\begin{theorem}
\label{thm:fsbnd_gencase}
{Suppose $\Delta_{m,q} = 3\cdot\mathcal{Y}_{m,q}[n] + \mathcal{Y}_{m,q}[n]^2$, and that $\mathfrak{g}$ and $\mathfrak{h}$ are constants. If Assumptions \ref{ass:drule}, \ref{ass:2norm}, \ref{ass:convergence_prime}, \ref{ass:zdim_fm} hold, then we have: $R(\widehat{\delta}_n) \leq R(\delta^*) + 2r_n$, with probability at least $1 - 6\cdot\mathfrak{g}\cdot\mathfrak{h}/n - 2c_n$; and we have that
\begin{equation}
\mathbb{H}(\widehat{\delta}_n(X,Z), Z) \leq \exp((r+d)T)\cdot\Delta_{\mathfrak{g},\mathfrak{h}} + \textstyle\frac{1}{T}
\end{equation}
with probability at least $1 - 6\cdot\mathfrak{g}\cdot\mathfrak{h}/n$, where $T$ is the constant such that $T^{\mathfrak{g}+\mathfrak{h}+3}=(\mathfrak{g}+1)!\cdot(\mathfrak{h}+1)!/(\lambda^{(\mathfrak{h}+1)/2}\cdot(M_{\mathfrak{g}+1,\mathfrak{h}+1} + M_{\mathfrak{g}+1,0}\cdot M_{0,\mathfrak{h}+1}))$.}
\end{theorem} 

\begin{remark}
{ The proof is omitted because it is a straightforward modification of the proof for Theorems \ref{thm:kac2} and \ref{thm:fsbnd}.}
\end{remark}

\begin{remark}
{The above theorem implies $\limsup_n \mathbb{H}(\widehat{\delta}_n(X,Z), Z) \leq 1/T$ because $\Delta_{\mathfrak{g},\mathfrak{h}} = o(1)$ under the conditions of the above theorem.}
\end{remark}

\section{Numerical Experiments}

\label{sec:er}

\begin{table}[t]
	\caption{\label{tab:datasets} List of Datasets Used in Numerical Experiments}
	\begin{center}
		\begin{scriptsize}
			\begin{tabular}{l|rrlll}
				\toprule
				Dataset & $p$ & $n$ & $Z$ Type & Task & Source \\
				\midrule
				Arrhythmia	& 10 	& 453	& Binary		& Classification	& \cite{guvenir1997supervised} 			\\
				Biodeg			& 40 	& 1055	& Categorical	& Classification	& \cite{mansouri2013quantitative} \\
				Communities		& 96	& 1994	& Continuous	& Regression		& \cite{department1992census,department1992law,statistics2004department,redmond2002data}								\\
				%Ecoli			& 6 	& 333	& Continuous	& Classification	& \cite{horton1996probabilistic}	\\
				EEG				& 12	& 4000	& Binary		& Regression		& \cite{fernandez2018feature}		\\
				Energy			& 8 	& 768	& Continuous	& Regression	& \cite{tsanas2012accurate} 		\\
				German Credit	& 49 	& 1000	& Continuous	& Classification	& \cite{Lichman:2013} 			\\
				%Image			& 18 	& 660	& Continuous	& Classification	& \cite{Lichman:2013} 			\\
				Letter			& 15 	& 20000	& Continuous	& Classification	& \cite{frey1991letter} 			\\
				%Magic			& 9		& 19020	& Continuous	& Classification	& \cite{bock2004methods} 			\\
				Music			& 68	& 1034	& Continuous	& Regression		& \cite{zhou2014predicting}		\\
				Parkinson's		& 18	& 5875	& Binary		& Both	& \cite{Lichman:2013} 			\\
				Pima Diabetes			& 7		& 768	& Continuous	& Classification	& \cite{smith1988using} 			\\
				Recidivism		& 6		& 5278	& Binary 		& Classification	& \cite{angwin2016machine} 		\\
				SkillCraft		& 17	& 3338	& Continuous	& Classification				& \cite{thompson2013video} 		\\
				Statlog			& 35	& 3486	& Binary		& Classification	& \cite{Lichman:2013} 			\\
				Steel			& 25	& 1941	& Categorical	& Classification	& \cite{Lichman:2013} 			\\
				Taiwan Credit	& 22	& 29623	& Binary		& Classification	& \cite{yeh2009comparisons} 		\\
				Wine Quality	& 11	& 6497	& Binary		& Both				& \cite{cortez2009modeling} 		\\
				\bottomrule
			\end{tabular}
		\end{scriptsize}
	\end{center}
\end{table}

In this section, we implement various levels of the FO problem (\ref{eq:fo}) for: classification, regression, and decision-making. In all cases, fairness is measured using disparate impact. Unless otherwise noted, all experiments were carried out using the Mosek 9 optimization package \cite{mosek2002mosek}. We first discuss the issue of hyperparameter selection for the FO problem, and then we describe the benchmark fairness methods that we compare our approach to. Next, we present classification and regression implementations of FO on a series of datasets from the UC Irvine Machine Learning Repository \cite{Lichman:2013}, the full list of which is in Table \ref{tab:datasets}. Finally, we present a case study on the use of FO to perform fair morphine dosing.

%All results are averaged over 50 iterations, where 70\% of the data is used for training and the 30\% is used for measuring the aforementioned metrics (reshuffled every iteration). Unless mentioned explicitly, all hyperparameters are chosen using 10-fold cross-validation. Classifier accuracy is measured by area-under-the-curve (AUC), which represents the area under the receiver operator characteristic (ROC) curve for a given binary classifier. Regression accuracy is measured by mean-squared-error (MSE), unless indicated otherwise. 

\subsection{Hyperparameter Selection}

{Applying the FO problem (\ref{eq:fo}) to particular datasets requires choosing several hyperparameters, namely: $\lambda$, $(\mathfrak{g},\mathfrak{h})$, and $\Delta_{m,q}$. The parameter $\lambda$ bounds the Euclidean norm of the model coefficients $B$, and it can be shown using standard duality arguments that varying $\lambda$ is equivalent to controlling the amount of $\ell_2$ regularization of the model coefficients. The parameters $(\mathfrak{g},\mathfrak{h})$ control the level of the FO problem, and the theory developed in previous sections says that consistency is achieved when $\mathfrak{g}$ and $\mathfrak{h}$ grow at a logarithmic or square-root-logarithmic rate. This implies that in practice small values of $(\mathfrak{g},\mathfrak{h})$ should be used. Last, the formulation in Section \ref{sec:ai} suggests an approach that makes the parameter choices $\Delta_{m,q} = \epsilon^{m+q}\cdot m!\cdot q!$. This is beneficial because it replaces multiple parameters $\Delta_{m,q}$ for $(m,q) \in [\mathfrak{g}]\times[\mathfrak{h}]$ with a single parameter $\epsilon$.}

{Consequently, applying the FO problem (\ref{eq:fo}) to a particular dataset requires choosing four hyperparameters, which is feasible using cross-validation. However, there is a subtlety because we have two criteria to evaluate the quality of a particular model, and these two criteria are generally (but not always) opposing each other. The first criteria is model accuracy, and the second criteria is model fairness. Because these criteria are generally opposing, cross-validation can only generate a Pareto frontier, which is a curve that for a particular quantitative level of fairness specifies the most accurate model possible at that level of fairness. Choosing a particular model from among that Pareto frontier requires a subjective choice for how much reduction in model accuracy is tolerable for any given increase in model fairness. To make this discussion more concrete, we consider examples of cross-validation for fair linear regression and fair linear classification.}

\begin{example}{\label{ex:fsvm} Consider a classification setup with $(X_i, Y_i) \in \mathbb{R}^p\times\{-1,+1\}$ and $Z_i \in \mathbb{R}$, and suppose we choose a linear decision rule $\delta(x) = Bx$ with $B \in \mathbb{R}^{1\times p}$. Then fair SVM using the level-(2,2) FO problem (\ref{eq:fo}) is given by
\begin{equation}
\label{eqn:fsvm_22}
\begin{aligned}
\min_{B \in \mathbb{R}^{1\times p}}\ & \textstyle\frac{1}{n}\sum_{i=1}^ns_i\\
\text{s.t. }& s_i \geq 0, &\text{for } i \in [n]\\
&s_i \geq 1-Y_i\cdot BX_i, &\text{for } i \in [n]\\
&-\hphantom{2}\epsilon^2 \leq BM_{(1,1)}\hphantom{B^\mathsf{T}} \leq \hphantom{2}\epsilon^2\\
&-2\epsilon^3 \leq BM_{(2,1)}\hphantom{B^\mathsf{T}} \leq 2\epsilon^3\\
&-2\epsilon^3 \leq BM_{(1,2)}B^\mathsf{T} \leq 2\epsilon^3\\
&-4\epsilon^4 \leq BM_{(2,2)}B^\mathsf{T} \leq 4\epsilon^4\\
&\|B\|_2 \leq \sqrt{\lambda}
%\textstyle\big\|\mathbb{E}_n\big(Z^{\otimes m}(B\Omega)^{\otimes q}\big)-\mathbb{E}_n\big(Z^{\otimes m}\big)\otimes\mathbb{E}_n\big((B\Omega)^{\otimes q}\big)\big\|\le\Delta_{m,q},\\
%&\qquad\text{for } (m,q)\in[\mathfrak{g}]\times[\mathfrak{h}].
\end{aligned}
\end{equation}
where we have the matrices
\begin{equation}
\label{eqn:exmat}
\begin{aligned}
M_{(1,1)} &= \textstyle\frac{1}{n}\sum_{i=1}^n Z_i\cdot X_i - \frac{1}{n}\sum_{i=1}^n Z_i\cdot \frac{1}{n}\sum_{i=1}^n X_i\\
M_{(2,1)} & = \textstyle\frac{1}{n}\sum_{i=1}^n Z_i^{\ 2}\cdot X_i - \frac{1}{n}\sum_{i=1}^n Z_i^{\ 2}\cdot \frac{1}{n}\sum_{i=1}^n X_i\\
M_{(1,2)} & = \textstyle\frac{1}{n}\sum_{i=1}^n Z_i\cdot X_i^{\vphantom{\mathsf{T}}}X_i^\mathsf{T} - \frac{1}{n}\sum_{i=1}^n Z_i\cdot \frac{1}{n}\sum_{i=1}^n X_i^{\vphantom{\mathsf{T}}}X_i^\mathsf{T}\\
M_{(2,1)} & = \textstyle\frac{1}{n}\sum_{i=1}^n Z_i^{\ 2}\cdot X_i^{\vphantom{\mathsf{T}}}X_i^\mathsf{T} - \frac{1}{n}\sum_{i=1}^n Z_i^{\ 2}\cdot \frac{1}{n}\sum_{i=1}^n X_i^{\vphantom{\mathsf{T}}}X_i^\mathsf{T}
\end{aligned}
\end{equation}
Observe that the constraint in (\ref{eqn:fsvm_22}) involving the matrix $M_{(m,q)}$ for any value of $(m,q) \in [2]\times[2]$ is precisely the specific form of the $(m,q)$ constraint in (\ref{eq:fo}) for this particular setup. Fair SVM using the level-$(\mathfrak{g},\mathfrak{h})$ FO problem for $1 \leq \mathfrak{g},\mathfrak{h} \leq 2$ is given by (\ref{eqn:fsvm_22}) with the appropriate constraints involving $M_{(m,q)}$ removed. An example of using five-fold cross-validation to construct a Pareto frontier for fair SVM is shown in Fig \ref{fig:cv_fsvm}. For each possible value of the hyperparameters, cross-validation generates a quantitative value for model accuracy and for model fairness. These pairs of values describe points that are plotted in Fig \ref{fig:fsvm_cv}. Fig \ref{fig:fsvm_pf} shows the Pareto frontier. Locations on the Pareto frontier with a point marker can be directly achieved by a model with a given set of hyperparameters, while locations on the Pareto frontier in between two point markers can be achieved by a randomized prediction that randomly chooses from one of two deterministic predictions that arise from the two models corresponding to the two point markers.}
\end{example}

\begin{figure*}[!t]
	\begin{center}
		\begin{subfigure}[t]{0.33\linewidth}
			\includegraphics[width=\linewidth]{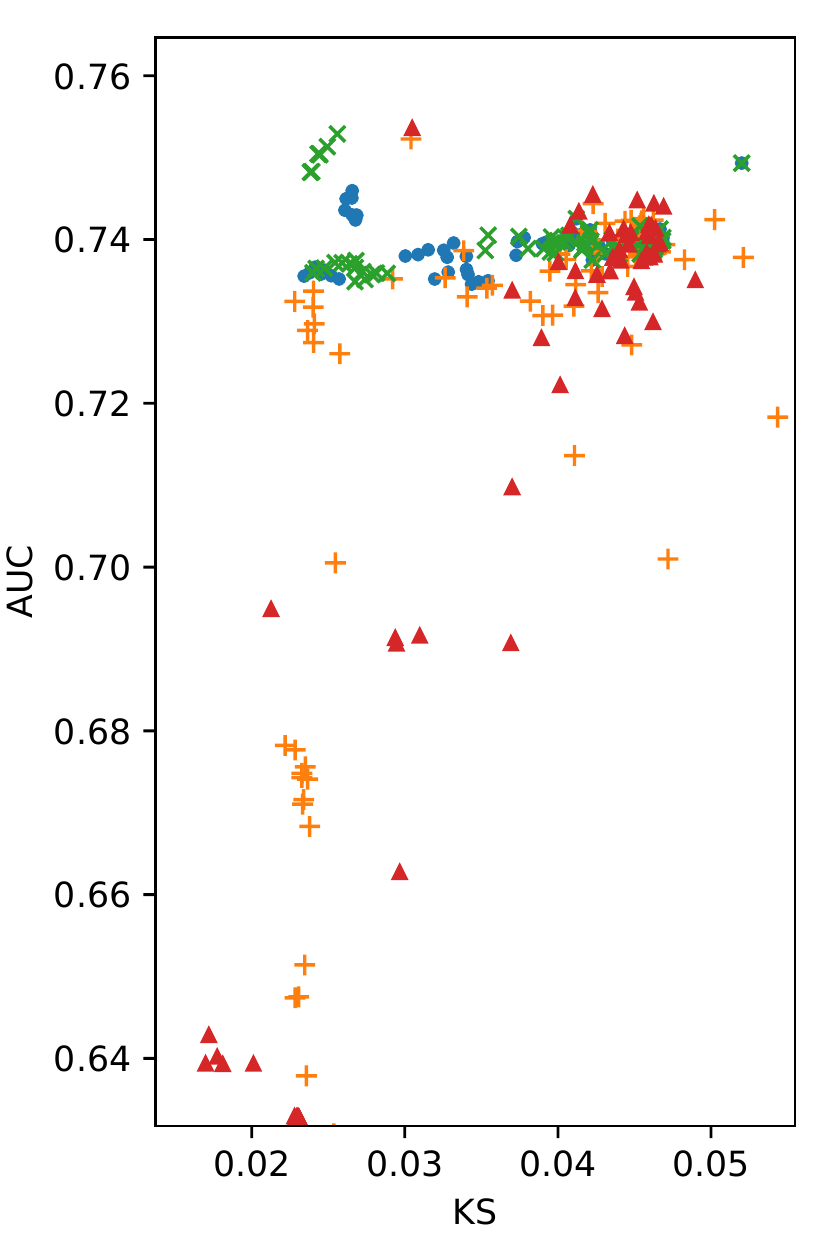}
			\caption{\label{fig:fsvm_cv} Cross-Validation}
		\end{subfigure}\hfill
		\begin{subfigure}[t]{0.33\linewidth}
			\includegraphics[width=\linewidth]{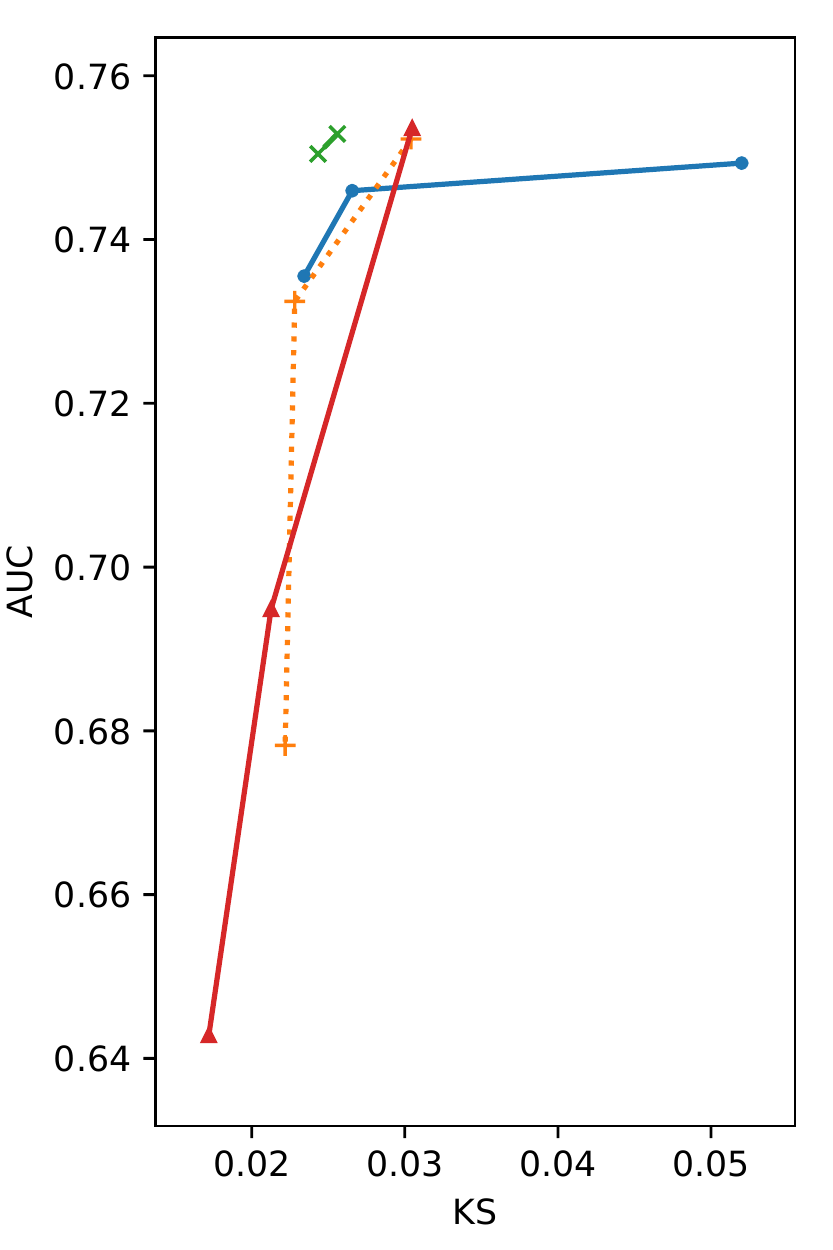}
			\caption{\label{fig:fsvm_lpf} Level Pareto Frontiers }
		\end{subfigure}\hfill
		\begin{subfigure}[t]{0.33\linewidth}
			\includegraphics[width=\linewidth]{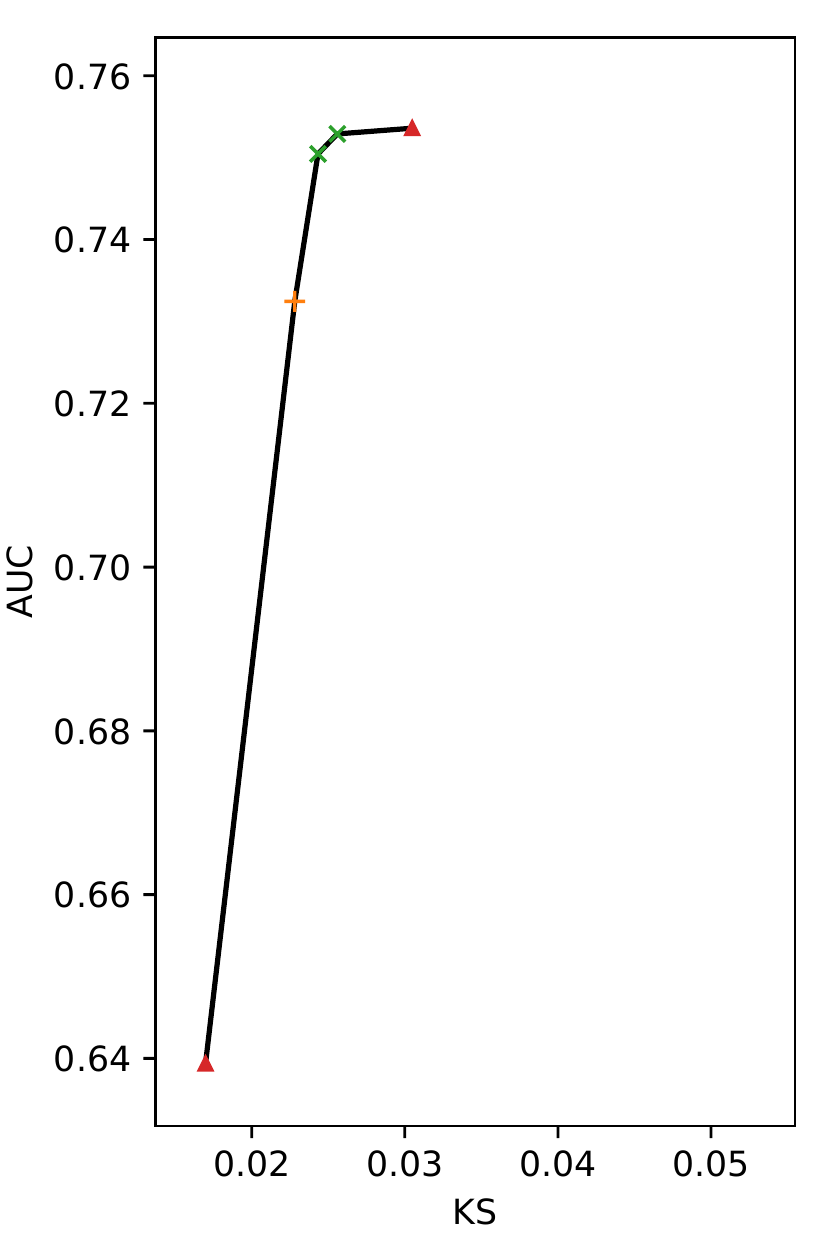}
			\caption{\label{fig:fsvm_pf} Pareto Frontier}
		\end{subfigure}
		\caption{\label{fig:cv_fsvm} Pareto frontier for fair SVM on the Letter dataset. Cross-validation is used to identify points of possible tradeoff between model accuracy (measured by area under the curve) and fairness (measured by Kolmogorov-Smirnov distance between the joint and product distributions of the model prediction and the protected information) using the FO problem (left), Pareto frontiers can be constructed for each individual level of the FO problem (middle), and a single Pareto frontier can be constructed for all the levels of the FO problem (right). For the points, circles are level-(1,1), pluses are level-(1,2), exes are level-(2,1), and triangles are level-(2,2).}
	\end{center}
\end{figure*}

\begin{figure*}[!t]
	\begin{center}
		\begin{subfigure}[t]{0.33\linewidth}
			\includegraphics[width=\linewidth]{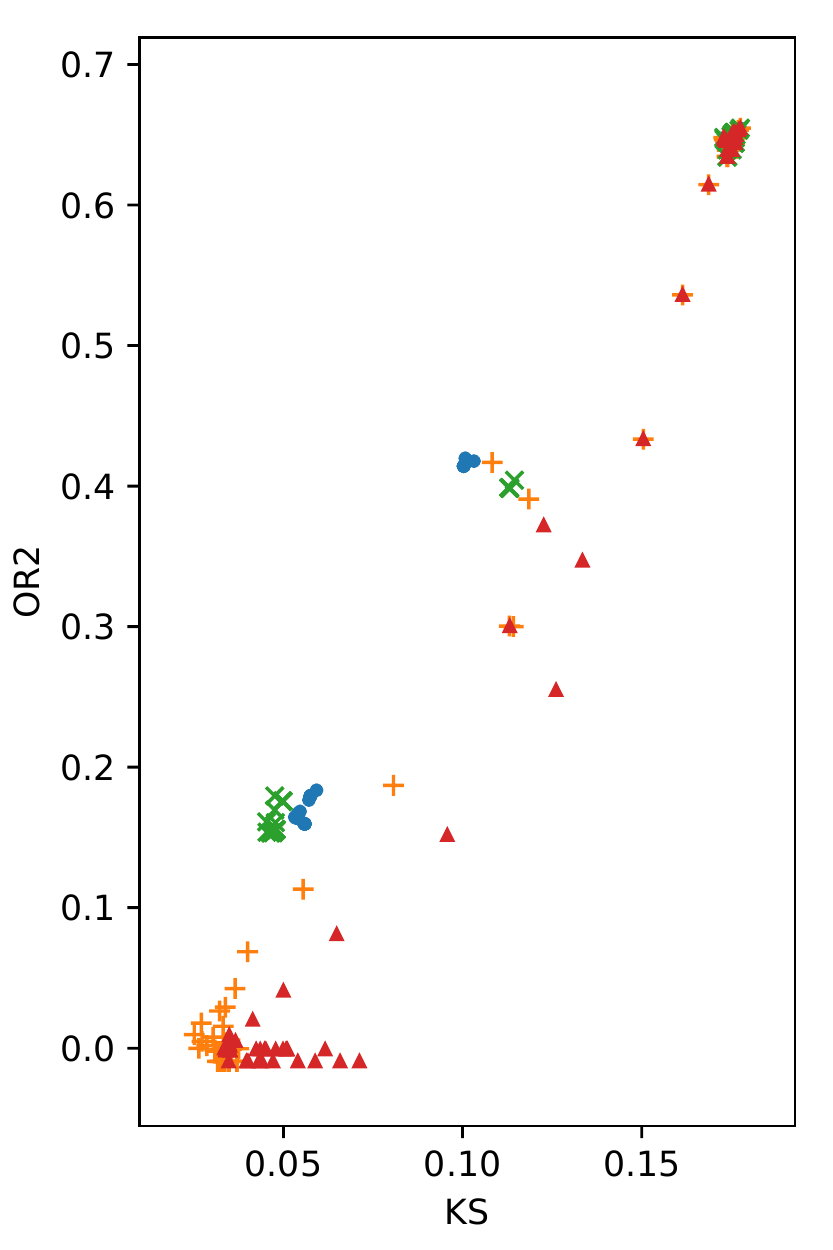}
			\caption{\label{fig:freg_cv} Cross-Validation}
		\end{subfigure}\hfill
		\begin{subfigure}[t]{0.33\linewidth}
			\includegraphics[width=\linewidth]{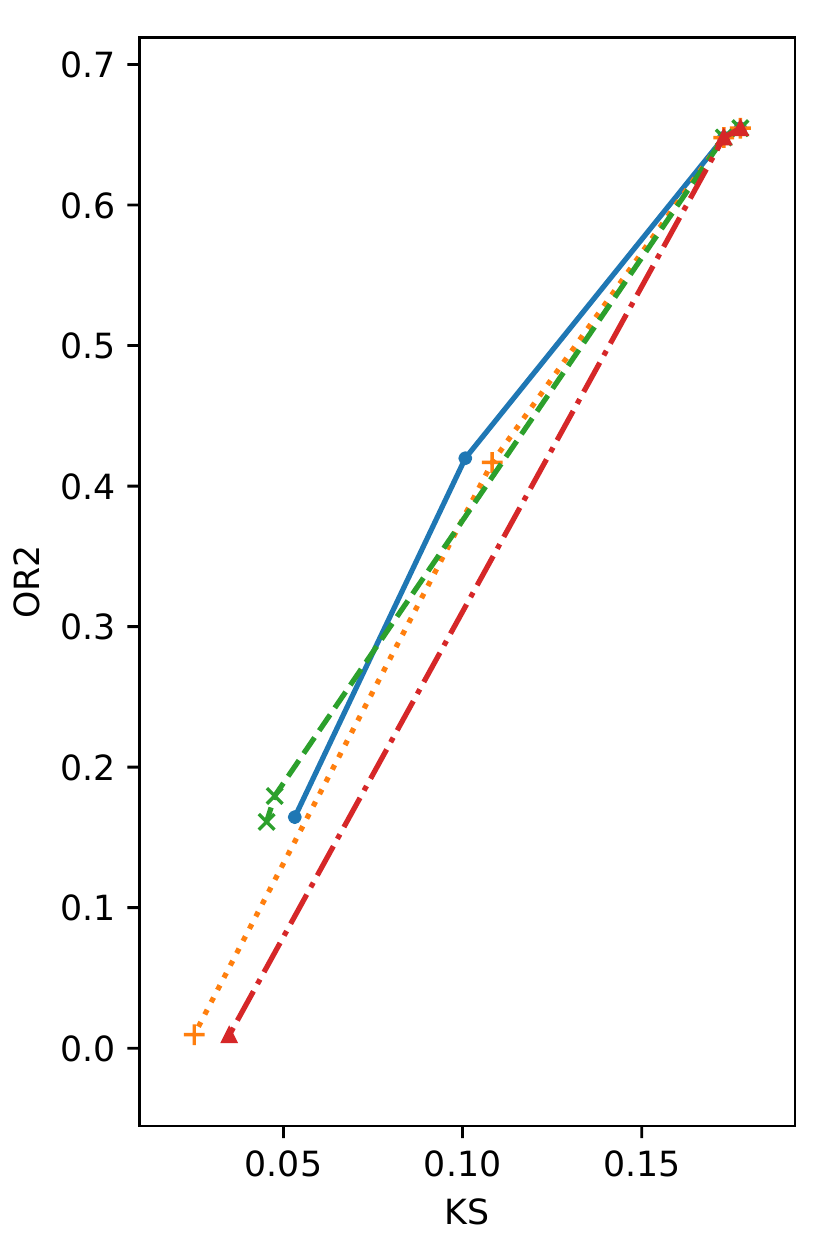}
			\caption{\label{fig:freg_lpf} Level Pareto Frontiers }
		\end{subfigure}\hfill
		\begin{subfigure}[t]{0.33\linewidth}
			\includegraphics[width=\linewidth]{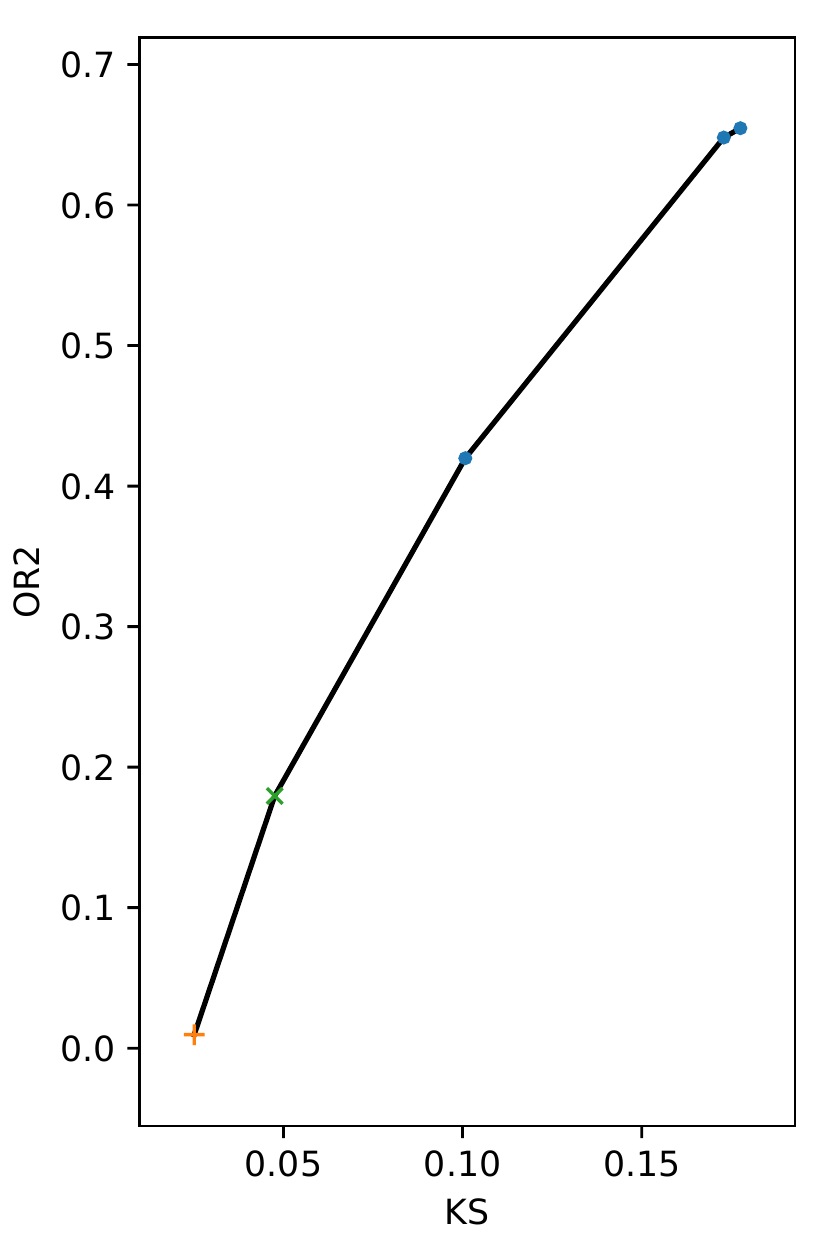}
			\caption{\label{fig:freg_pf} Pareto Frontier}
		\end{subfigure}
		\caption{\label{fig:cv_freg} Pareto frontier for fair regression on the Communities dataset. Cross-validation is used to identify points of possible tradeoff between model accuracy (measured by out-of-sample $R^2$) and fairness (measured by Kolmogorov-Smirnov distance between the joint and product distributions of the model prediction and the protected information) using the FO problem (left), Pareto frontiers can be constructed for each individual level of the FO problem (middle), and a single Pareto frontier can be constructed for all the levels of the FO problem (right). For the points, circles are level-(1,1), pluses are level-(1,2), exes are level-(2,1), and triangles are level-(2,2).}
	\end{center}
\end{figure*}

\begin{example}{\label{ex:freg}  Consider a regression setup with $(X_i, Y_i) \in \mathbb{R}^p\times\mathbb{R}$ and $Z_i \in \mathbb{R}$, and suppose we choose a linear decision rule $\delta(x) = Bx$ with $B \in \mathbb{R}^{1\times p}$. Then fair regression using the level-(2,2) FO problem (\ref{eq:fo}) is
\begin{equation}
\label{eqn:freg_22}
\begin{aligned}
\min_{B \in \mathbb{R}^{1\times p}}\ & \textstyle\frac{1}{n}\sum_{i=1}^n(Y_i - BX_i)^2\\
\text{s.t. }&-\hphantom{2}\epsilon^2 \leq BM_{(1,1)}\hphantom{B^\mathsf{T}} \leq \hphantom{2}\epsilon^2\\
&-2\epsilon^3 \leq BM_{(2,1)}\hphantom{B^\mathsf{T}} \leq 2\epsilon^3\\
&-2\epsilon^3 \leq BM_{(1,2)}B^\mathsf{T} \leq 2\epsilon^3\\
&-4\epsilon^4 \leq BM_{(2,2)}B^\mathsf{T} \leq 4\epsilon^4\\
&\|B\|_2 \leq \sqrt{\lambda}
%\textstyle\big\|\mathbb{E}_n\big(Z^{\otimes m}(B\Omega)^{\otimes q}\big)-\mathbb{E}_n\big(Z^{\otimes m}\big)\otimes\mathbb{E}_n\big((B\Omega)^{\otimes q}\big)\big\|\le\Delta_{m,q},\\
%&\qquad\text{for } (m,q)\in[\mathfrak{g}]\times[\mathfrak{h}].
\end{aligned}
\end{equation}
where the matrices are as in (\ref{eqn:exmat}). Observe that the constraint in (\ref{eqn:freg_22}) involving the matrix $M_{(m,q)}$ for any value of $(m,q) \in [2]\times[2]$ is precisely the specific form of the $(m,q)$ constraint in (\ref{eq:fo}) for this particular setup. Fair regression using the level-$(\mathfrak{g},\mathfrak{h})$ FO problem for $1 \leq \mathfrak{g},\mathfrak{h} \leq 2$ is given by (\ref{eqn:freg_22}) with the appropriate constraints involving $M_{(m,q)}$ removed. An example of using five-fold cross-validation to construct a Pareto frontier for fair regression is shown in Fig \ref{fig:cv_freg}. For each possible value of the hyperparameters, cross-validation generates a quantitative value for model accuracy and for model fairness. These pairs of values describe points that are plotted in Fig \ref{fig:freg_cv}. Fig \ref{fig:freg_pf} shows the Pareto frontier. Locations on the Pareto frontier with a point marker can be directly achieved by a model with a given set of hyperparameters, while locations on the Pareto frontier in between two point markers can be achieved by a randomized prediction that randomly chooses from one of two deterministic predictions that arise from the two models corresponding to the two point markers.}
\end{example}

\subsection{Comparison Methods}

In the following subsections, we compare FO to three other methods. The methods of \cite{berk2017convex} and \cite{kamishima2012fairness} are designed for fair classification and fair regression, respectively, and are similar to our method in that they enforce fairness at training time. We also compare FO to the method of \cite{calmon2017optimized}, although this takes a pre-processing approach. { In all comparison methods, we include an $\ell_2$ regularization on the model coefficients $B$. This is done to ensure an equitable comparison to the FO problem (\ref{eq:fo}), which includes a constraint on the Euclidean norm of the model coefficients.}

\paragraph{Berk et al. \cite{berk2017convex}}

The method of \cite{berk2017convex} is one of the few comparable methods for fair regression. They also take an in-training approach, defining two regularization terms that enforce fairness. Let $P_z=\{i\in[n]:Z_i=z\}$, and note $\#P_z$ refers to the cardinality of these sets. Given a binary protected attribute $Z$, they define a regularizer for group fairness
\begin{equation}
\textstyle\big((\# P_{-1}\cdot\# P_{+1})^{-1}\sum_{i \in P_{-1}}\sum_{j\in P_{+1}}d(Y_i,Y_j)\cdot(X_i^\mathsf{T}\beta-X_j^\mathsf{T}\beta)\big)^2,
\end{equation} 
for some distance measure $d(\cdot,\cdot)$. Note that this is similar to the term constrained in FO for $(m,q)=(1,1)$. They also define the following regularizer for individual fairness:
\begin{equation}
\textstyle(\# P_{-1}\cdot\# P_{1})^{-1}\sum_{i \in P_{-1}}\sum_{j\in P_{+1}}d(Y_i,Y_j)\cdot(X_i^\mathsf{T}\beta-X_j^\mathsf{T}\beta)^2.
\end{equation}
\noindent This term is similar to a term in FO for $(m,q)= (1,2)$, although not equivalent. It has the benefit of being convex, although the double-summation term can be computationally prohibitive for large datasets. In our implementation, we estimate this term from a sub-sample (10\%) of the data when this issue arises. We note that this method can only accommodate binary-valued protected attributes, and so we cannot provide comparisons to several of the datasets for fair regression. For this method, the group fairness and individual fairness terms are implemented as a penalty in the objective. 

\paragraph{Calmon et al. \cite{calmon2017optimized}}

This work is comparable to that of \cite{zemel2013learning}. Both of these works formulate nonparametric optimization problems whose solution yields a conditional distribution $f_{\widehat{X},\widehat{Y}|X,Y,Z}$ that then probabilistically transforms the data. We only compare our method to the approach introduced in \cite{calmon2017optimized}, since their formulation directly builds on that of \cite{zemel2013learning}.

Given a predefined notion of deviation amongst distributions, this method minimizes the overall deviation of $f_{\widehat{X},\widehat{Y}}$ from $f_{X,Y\vphantom{\widehat{Y}}}$. In the original work, the authors chose to minimize $\frac{1}{2}\sum_{x,y}|f_{\widehat{X},\widehat{Y}}(x,y)-f_{X,Y\vphantom{\widehat{Y}}}(x,y)|$. They also include constraints on pointwise distortion $\mathbb{E}_{\widehat{X},\widehat{Y}|X,Y}[\theta((X,Y),(\widehat{X},\widehat{Y})]$ for some user-defined function $\theta:\left\lbrace\mathbb{R}^p\times\{\pm1\}\right\rbrace^2\rightarrow\mathbb{R}_{\geq 0}$. There are also bounds on the dependency of the new main label $\widehat{Y}$ on the original protected label $J(f_{\widehat{Y}|Z}[y|z],f_{\vphantom{\widehat{Y}}Y}(y))$, where $J(a,b) = |\frac{a}{b}-1|$ is defined to be the probability ratio measure. Thus, the final formulation is 
\begin{equation}
\label{eq:calmon}
\begin{aligned}
	\min\ &\textstyle\frac{1}{2}\sum_{x,y}|f_{\widehat{X},\widehat{Y}}(x,y)-f_{X,Y\vphantom{\widehat{Y}}}(x,y)|\\
	\text{s.t. }&\mathbb{E}_{\widehat{X},\widehat{Y}|X,Y}[\theta((X,Y),(\widehat{X},\widehat{Y})|x,y]\le c, & \text{for all } x,y\\
	&|f_{\vphantom{\widehat{Y}}Y}(y)^{-1}f_{\widehat{Y}|Z}[y|z]-1|\le d, & \text{for all } y,z\\
	&f_{\widehat{X},\widehat{Y}|X,Y,Z}\textrm{ are all distributions.}
\end{aligned}
\end{equation}
Following the procedure used by the authors, we approximate $f_{X,Y,Z}$ with the empirical distribution of the original data, separated into a pre-selected number of bins. The resulting optimization problem will have $8(\#\textrm{bins})^{2p}$ parameters, which can quickly become computationally intractable when the dataset is high-dimensional. To account for this, we follow the original work and choose the 3 features most correlated with the main label $Y$. Each dimension is split into 8 bins. We choose $\theta((x',y'),(x,y))$ to be $0$ if $y=y'$ and $x=x'$, $0.5$ if $y=y'$ and $x,x$ vary by at most one in any dimension, and $1$ otherwise: This is similar to the $\theta$ chosen in the original paper.

\paragraph{Kamishima et al. \cite{kamishima2012fairness}}

Another comparable method is that of \cite{kamishima2012fairness}, which also aims to enforce fairness at training time. As opposed to our approach of bounding interaction moments, they instead regularize with a mutual information term. Also, this method differs from our framework notably in that it imposes different treatments for different protected classes, violating the principle of individual fairness; as a result, it is also unable to handle continuous protected attributes. The authors implement their regularizer in the context of logistic regression. Let $\sigma$ be a sigmoid function and $g_{\beta}[y|x,z]=y\sigma(x^{\textsf{T}}\beta)+(1-y)(1-\sigma(x^{\textsf{T}}\beta))$, and note that the notation $\beta_z$ indicates that this approach has a different set of coefficients for each possible value of $Z$. the authors approximate the mutual information as
\begin{equation}
\textstyle n^{-1}\sum_{i=1}^{n}\sum_{y\in\{\pm1\}}g_{\beta_{Z_i}}[y|X_i,Z_i]\log\frac{\widehat{P}[y|Z_i]}{\widehat{P}(y)},
\end{equation}
with $\widehat{P}[y|z] = (\# P_{z})^{-1}\sum_{i\in P_{z}}g_{\beta_{z}}[y|X_i,z]$ and $\widehat{P}(y) = \frac{1}{n}\sum_{i=1}^{n}g_{\beta_{z_i}}[y|X_i,Z_i]$. This is then weighted and added to the objective as a regularizer. We include this method as a comparison to our fair SVM, while noting the core differences mentioned above. All experiments for this method were done using the sequential least squares programming approach of \cite{kraft1988software}.

\begin{figure*}[!t]
	\begin{center}
		\begin{subfigure}[t]{0.33\linewidth}
			\includegraphics[width=\linewidth]{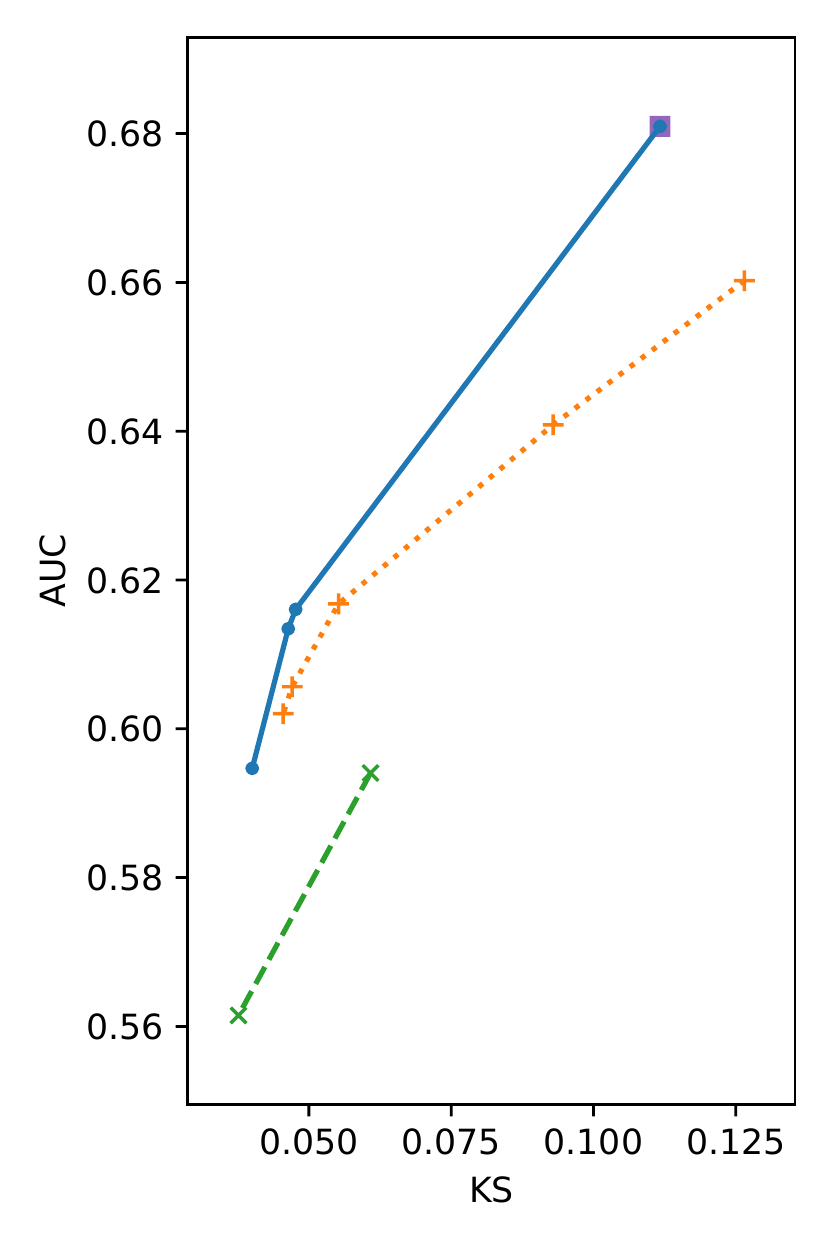}
			\caption{Arrhythmia}
		\end{subfigure}\hfill
		\begin{subfigure}[t]{0.33\linewidth}
			\includegraphics[width=\linewidth]{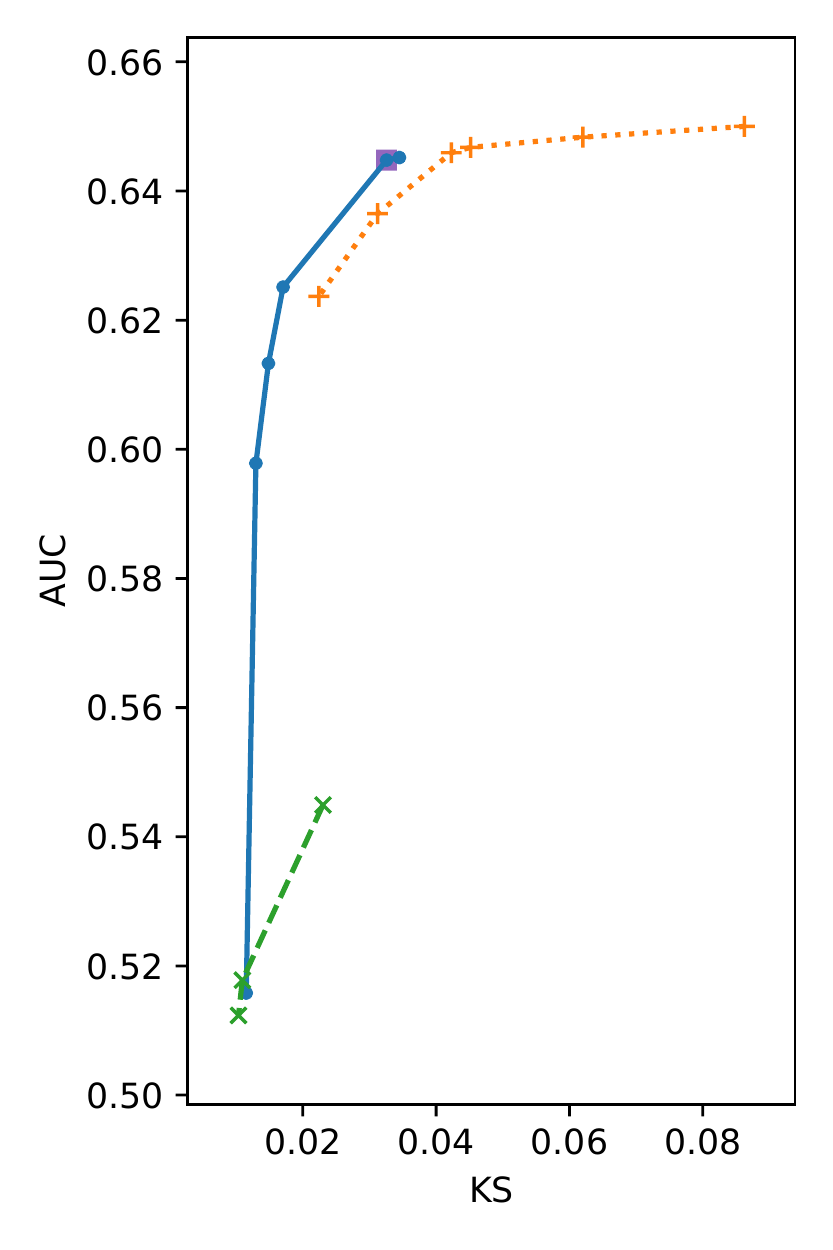}
			\caption{Parkinson's}
		\end{subfigure}\hfill
		\begin{subfigure}[t]{0.33\linewidth}
			\includegraphics[width=\linewidth]{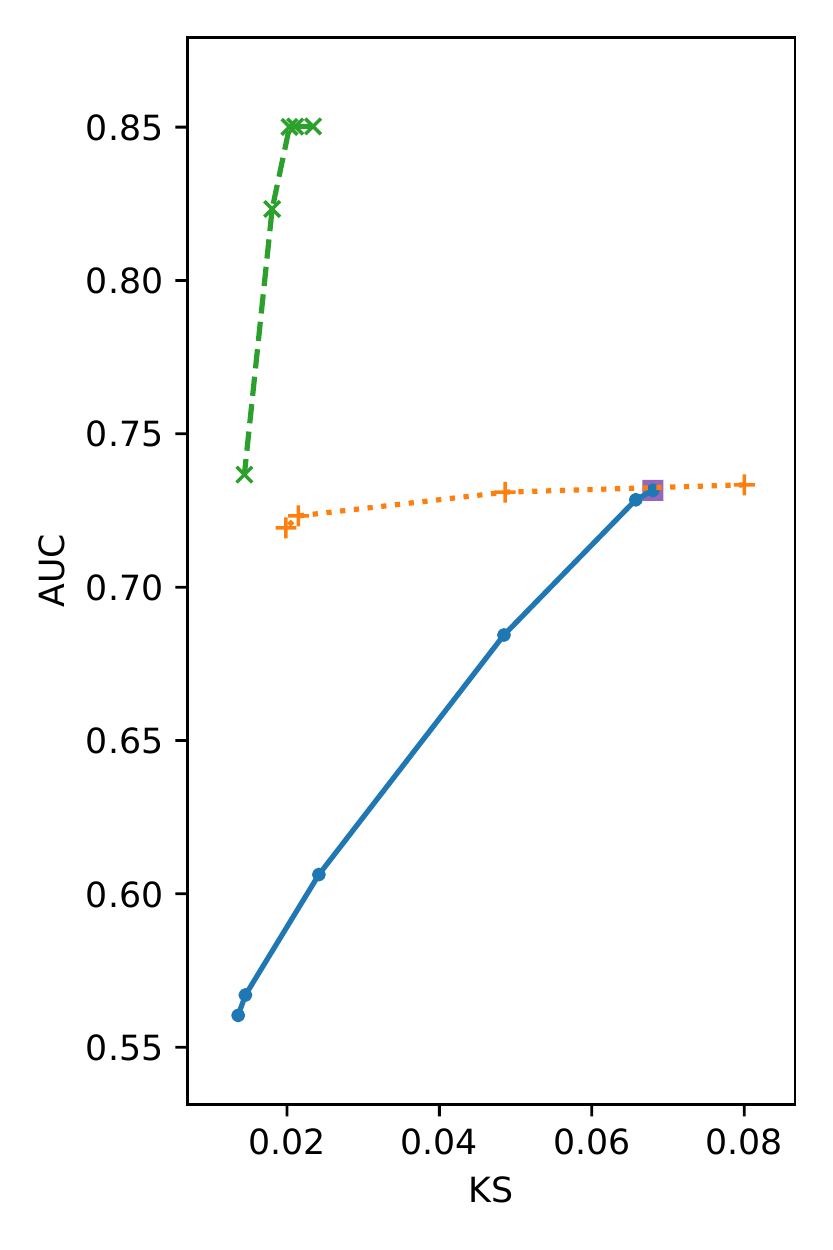}
			\caption{Recidivism}
		\end{subfigure}\hfill\\
		\begin{subfigure}[t]{0.33\linewidth}
			\includegraphics[width=\linewidth]{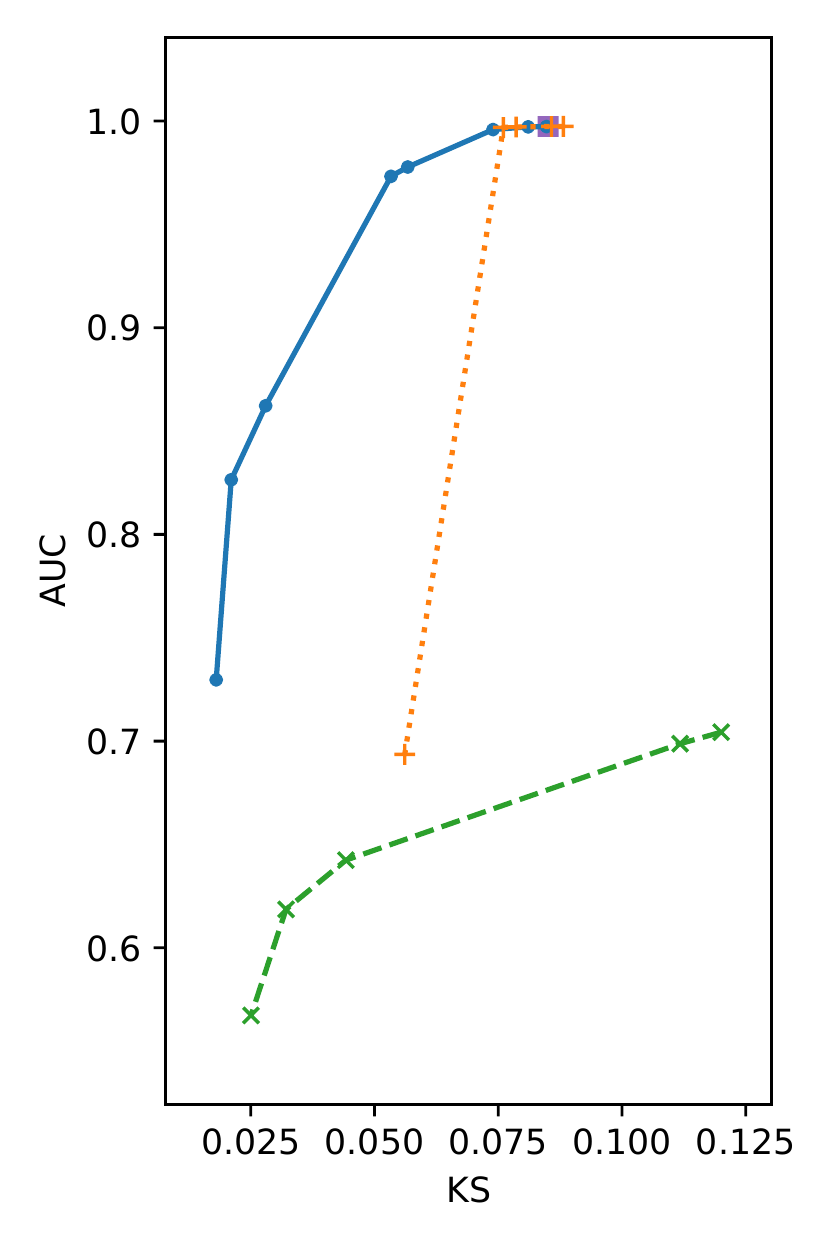}
			\caption{Statlog}
		\end{subfigure}\hfill
		\begin{subfigure}[t]{0.33\linewidth}
			\includegraphics[width=\linewidth]{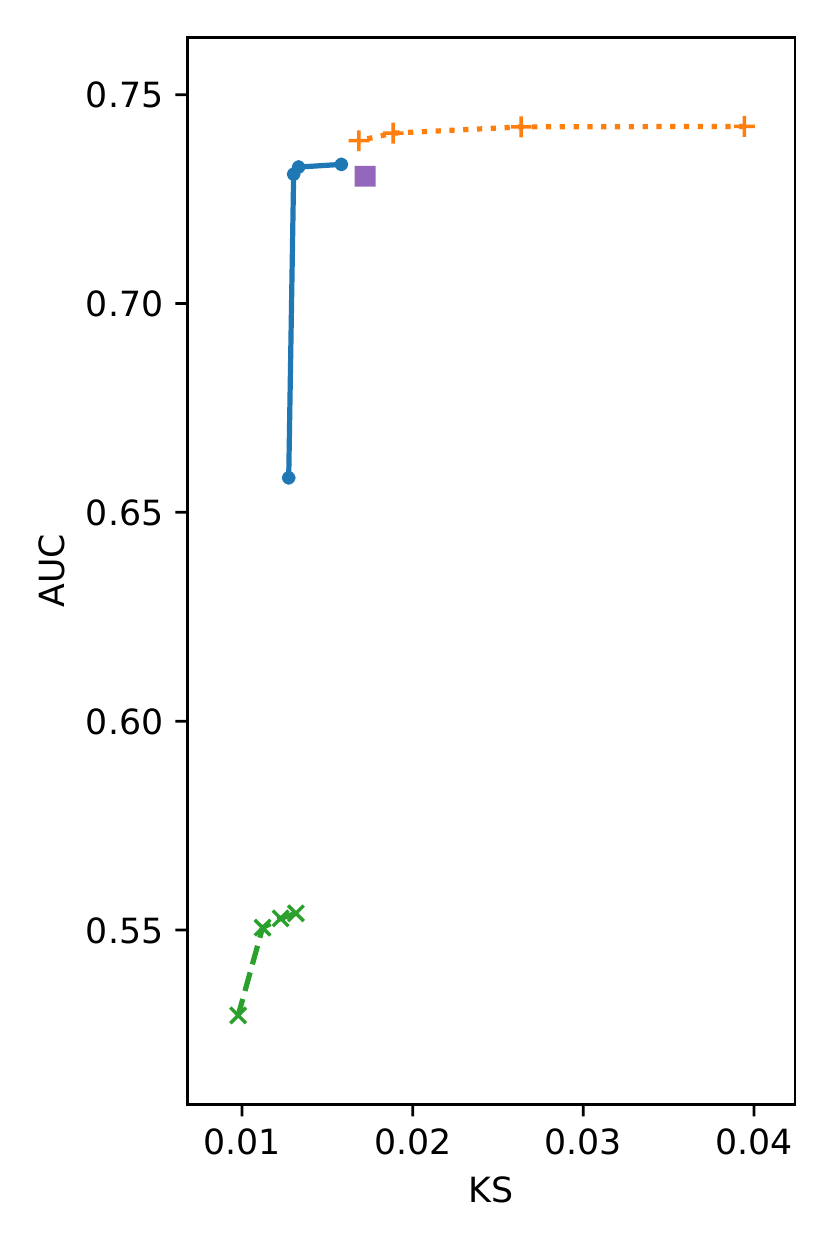}
			\caption{Taiwan Credit}
		\end{subfigure}\hfill
		\begin{subfigure}[t]{0.33\linewidth}
			\includegraphics[width=\linewidth]{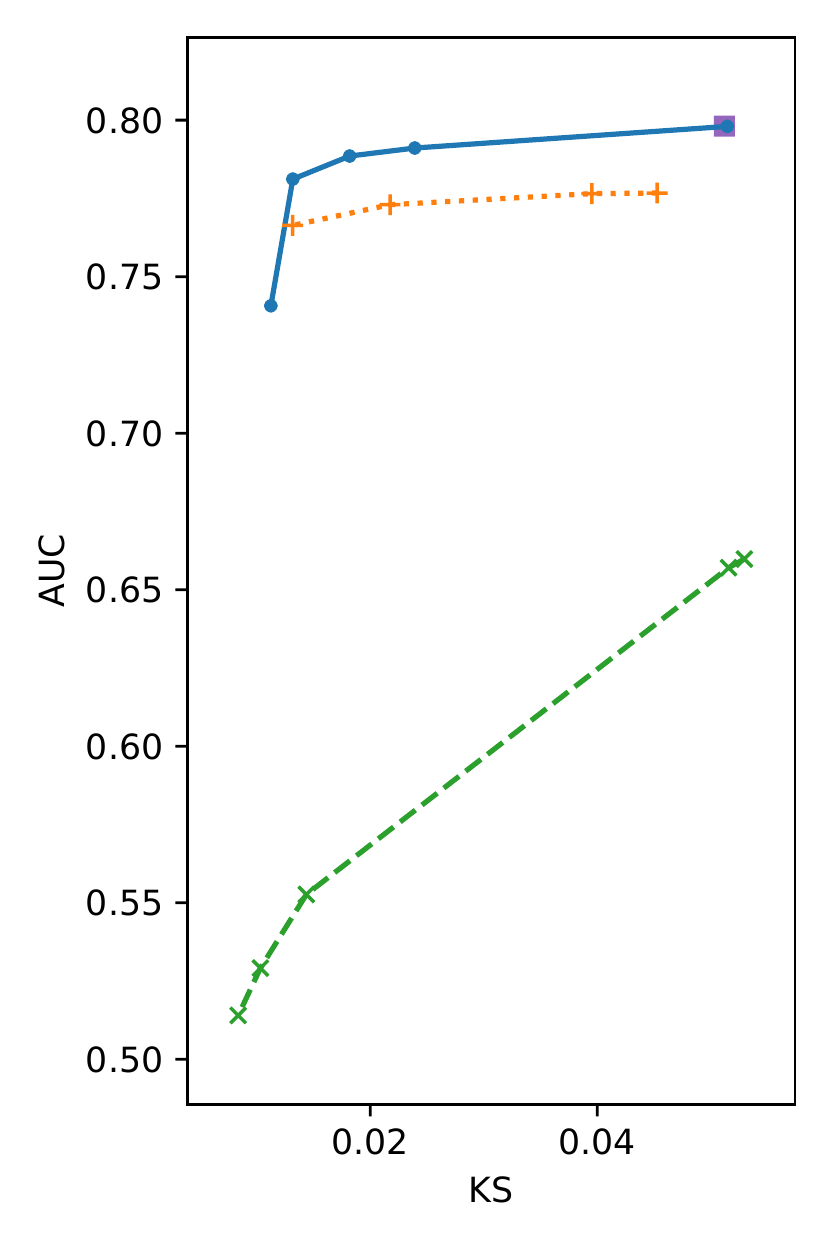}
			\caption{Wine Quality}
		\end{subfigure}\hfill		
		\caption{\label{fig:fsvm} Pareto frontiers for fair SVM on datasets with binary protected attribute. The approaches compared are the FO formulation (solid line), Kamishima et al. \cite{kamishima2012fairness} (dotted line), and Calmon et al. \cite{calmon2017optimized} (dashed line). The square mark denotes linear SVM without any fairness modifications.}
	\end{center}
\end{figure*}

\begin{figure*}[!t]
	\begin{center}
		\begin{subfigure}[t]{0.33\linewidth}
			\includegraphics[width=\linewidth]{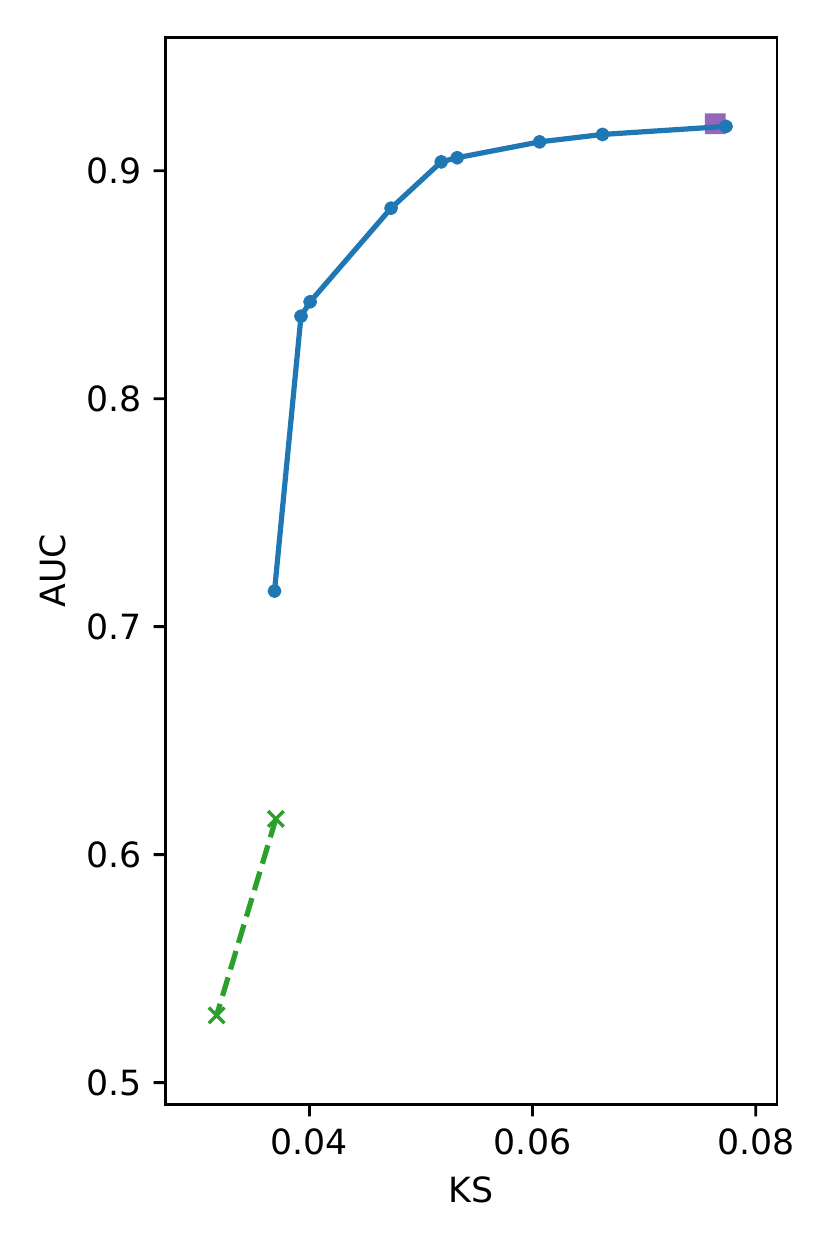}
			\caption{Biodeg}
		\end{subfigure}\hfill
		\begin{subfigure}[t]{0.33\linewidth}
			\includegraphics[width=\linewidth]{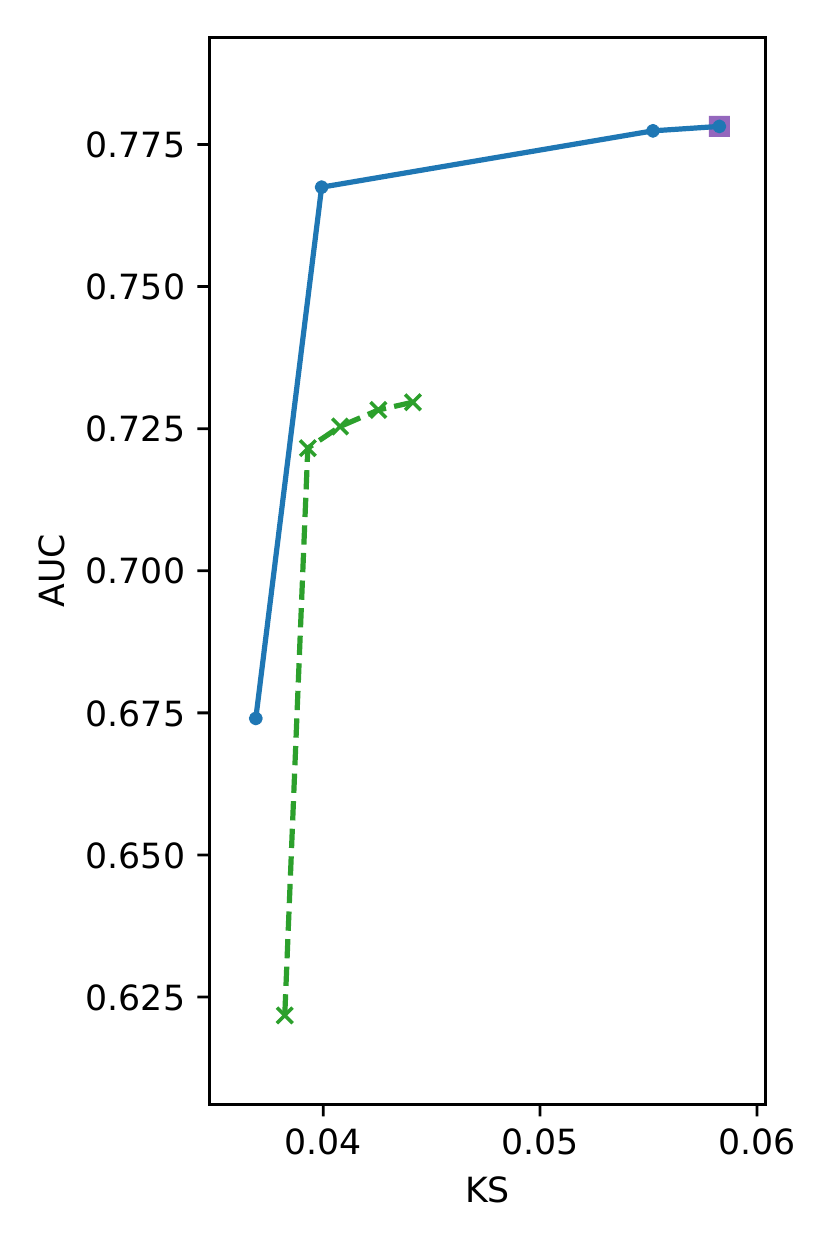}
			\caption{German Credit}
		\end{subfigure}\hfill
		\begin{subfigure}[t]{0.33\linewidth}
			\includegraphics[width=\linewidth]{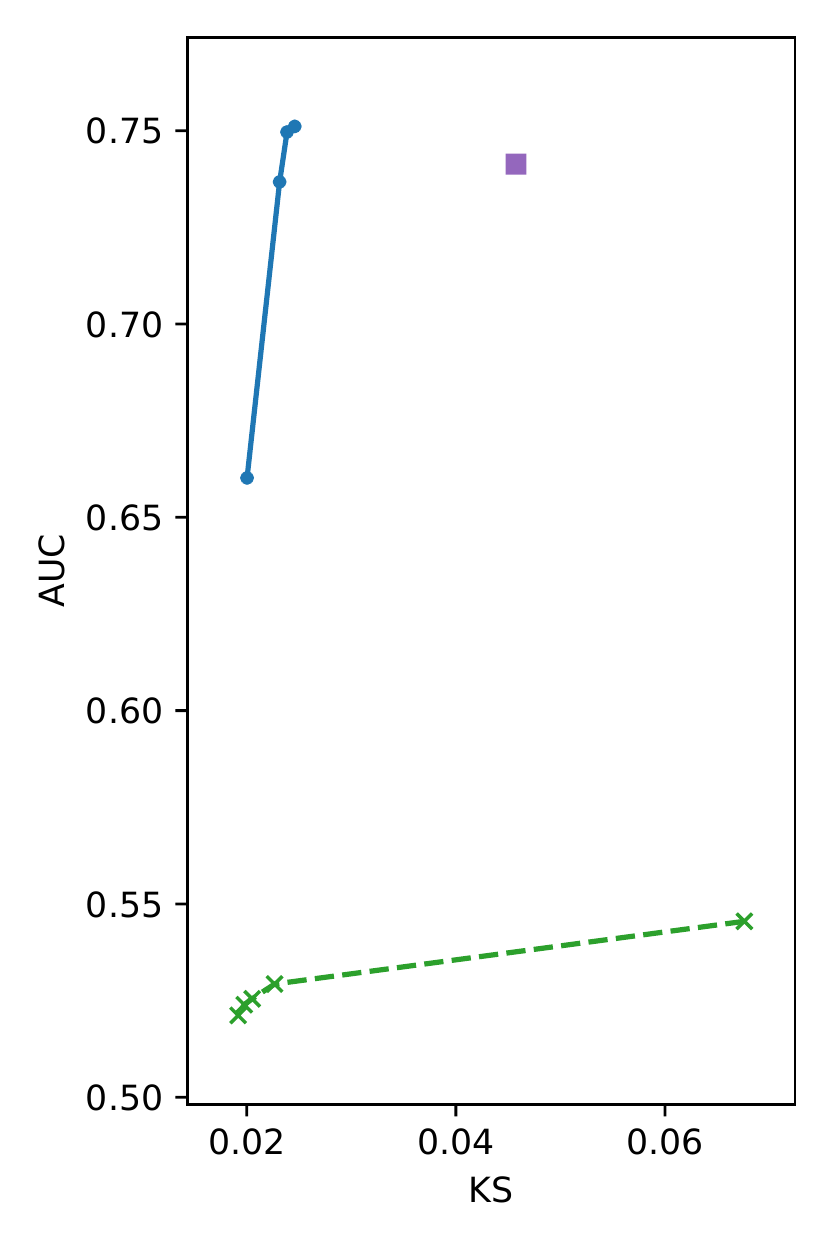}
			\caption{Letter}
		\end{subfigure}\hfill\\
		\begin{subfigure}[t]{0.33\linewidth}
			\includegraphics[width=\linewidth]{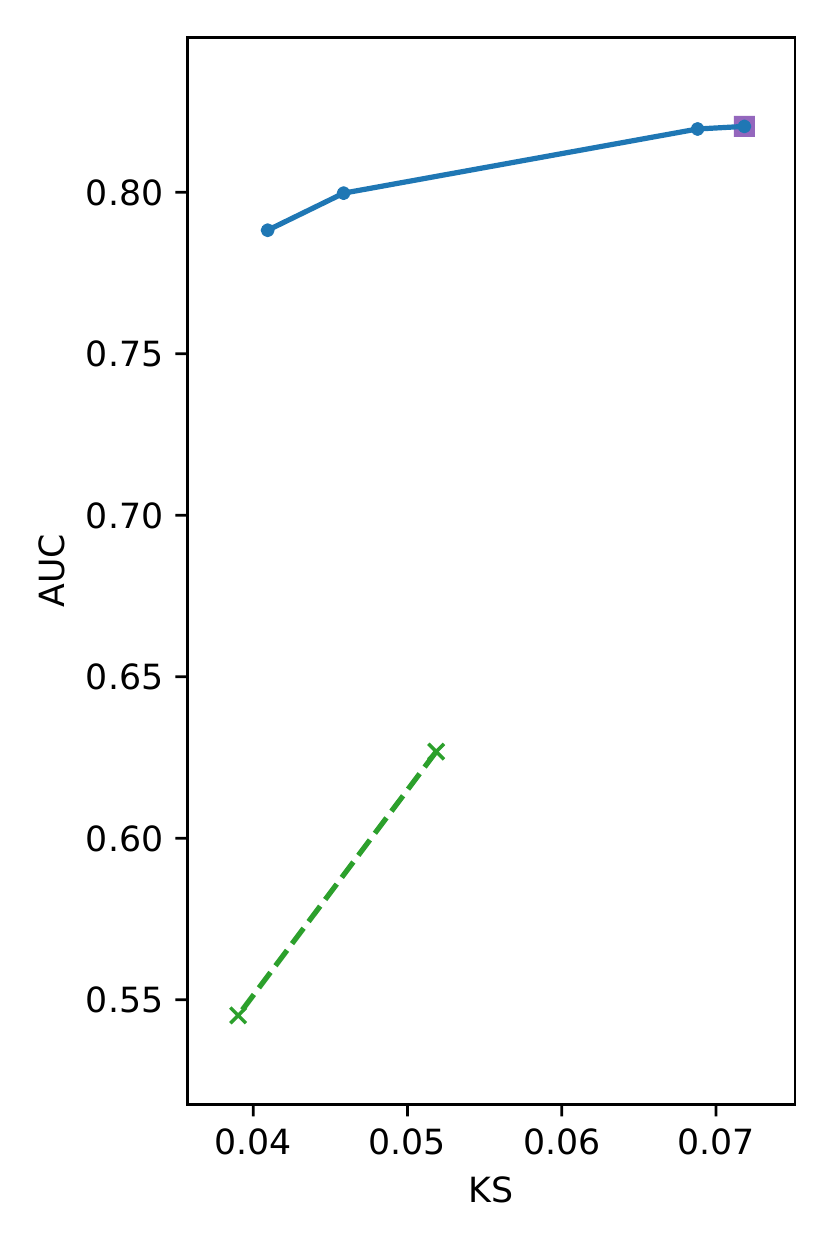}
			\caption{Pima Diabetes}
		\end{subfigure}\hfill
		\begin{subfigure}[t]{0.33\linewidth}
			\includegraphics[width=\linewidth]{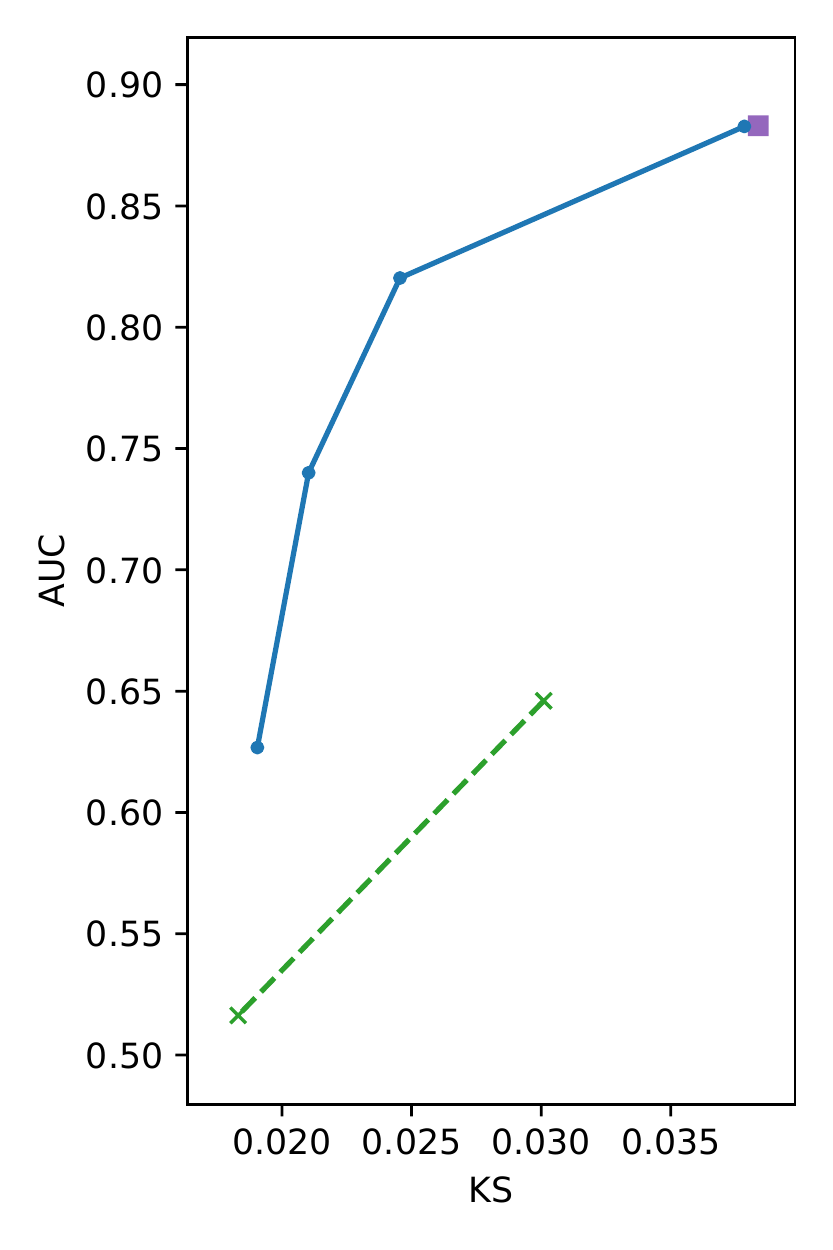}
			\caption{Skillcraft}
		\end{subfigure}\hfill
		\begin{subfigure}[t]{0.33\linewidth}
			\includegraphics[width=\linewidth]{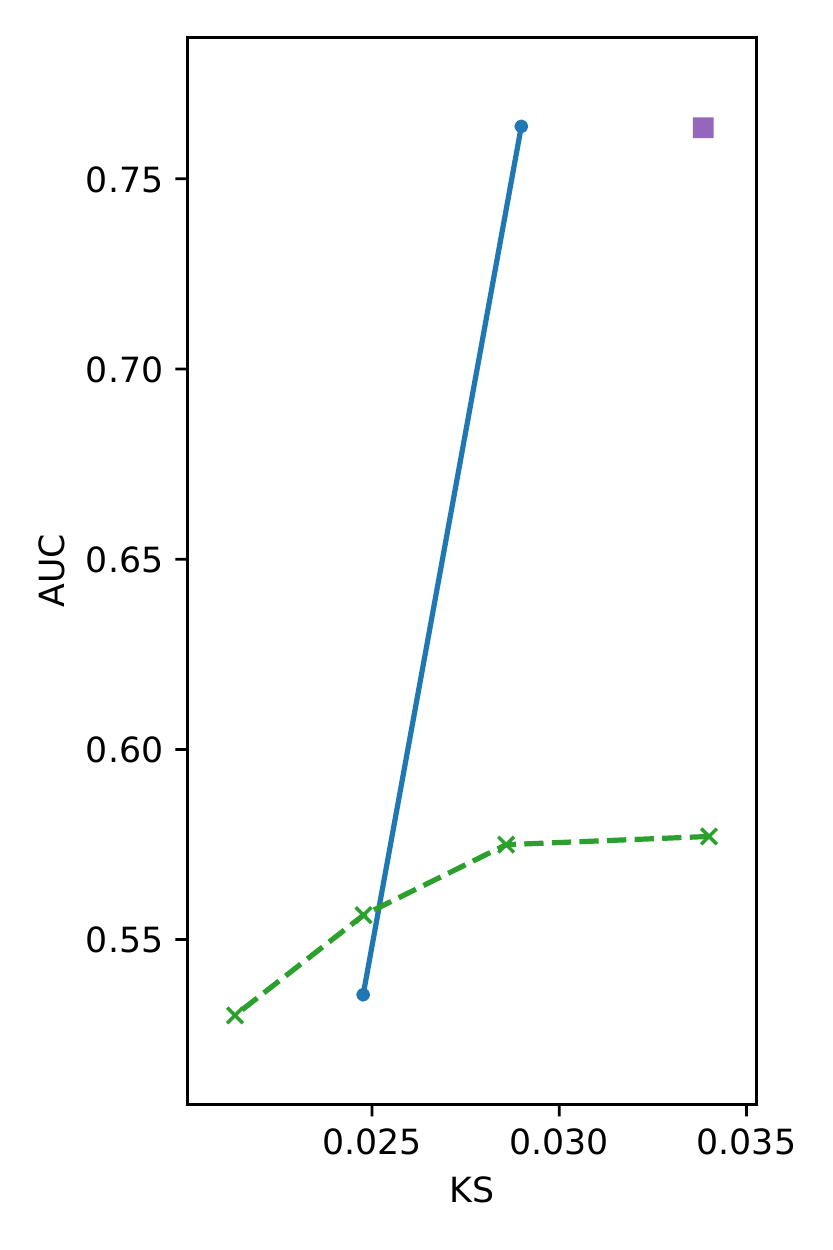}
			\caption{Steel}
		\end{subfigure}\hfill		
		\caption{\label{fig:fsvm2} Pareto frontiers for fair SVM on datasets with continuous or categorical protected attributes. The approaches compared are the FO formulation (solid line) and Calmon et al. \cite{calmon2017optimized} (dashed line). The square mark denotes linear SVM without any fairness modifications.}
	\end{center}
\end{figure*}

\subsection{Fair SVM}

We consider classification problems using a series of datasets. {For the FO approach, we consider the formulation in Example \ref{ex:fsvm}. We perform five-fold cross validation repeated five times. The Pareto frontiers of different approaches are shown in Fig \ref{fig:fsvm} and Fig \ref{fig:fsvm2}. Accuracy is measured by the area under the curve (AUC) since classifier models are often used as scores that are then subject to different thresholds. Fairness is measured by the Kolmogorov-Smirnov distance between the joint and product distributions of the model prediction and the protected information. The variance of the results over the five repetitions is low, and so this is not plotted to make the results easier to visualize.} Since the mutual-information-based method of \cite{kamishima2012fairness} cannot accommodate continuous protected variables, results are not reported for this method for the associated datasets. We note our method often improves fairness with less cost (in terms of accuracy) than the method of \cite{calmon2017optimized}. This is to be expected, as such pre-processing approaches do not take into account the downstream task that the transformed data is to be used for. Our method is also able to match or improve the fairness results of the mutual information approach. Recall that this method maintains explicitly different treatments for different protected classes, while ours adheres to the principle of individual fairness. Given this, it is unsurprising that the method of \cite{kamishima2012fairness} can sometimes achieve fairness at a lower cost to accuracy, although our method even outperforms on this metric for a number of datasets. Further, this feature of disparate treatments can yield fairness values notably worse than even a standard SVM. { Interestingly, for the Taiwan Credit, Letter, and Steel datasets our method can do strictly better in terms of both accuracy and fairness than linear SVM without fairness modifications.}

\begin{figure*}[!t]
	\begin{center}
		\begin{subfigure}[t]{0.33\linewidth}
			\includegraphics[width=\linewidth]{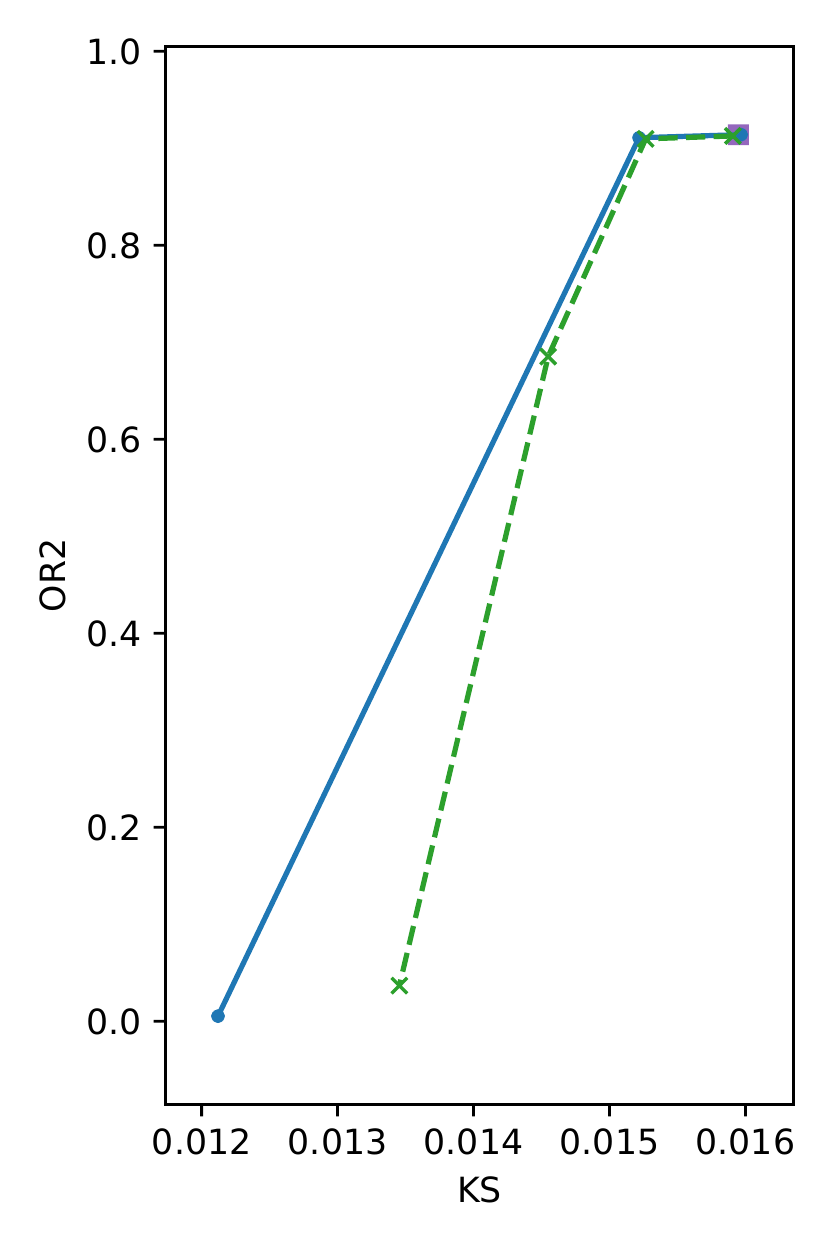}
			\caption{EEG}
		\end{subfigure}\hfill
		\begin{subfigure}[t]{0.33\linewidth}
			\includegraphics[width=\linewidth]{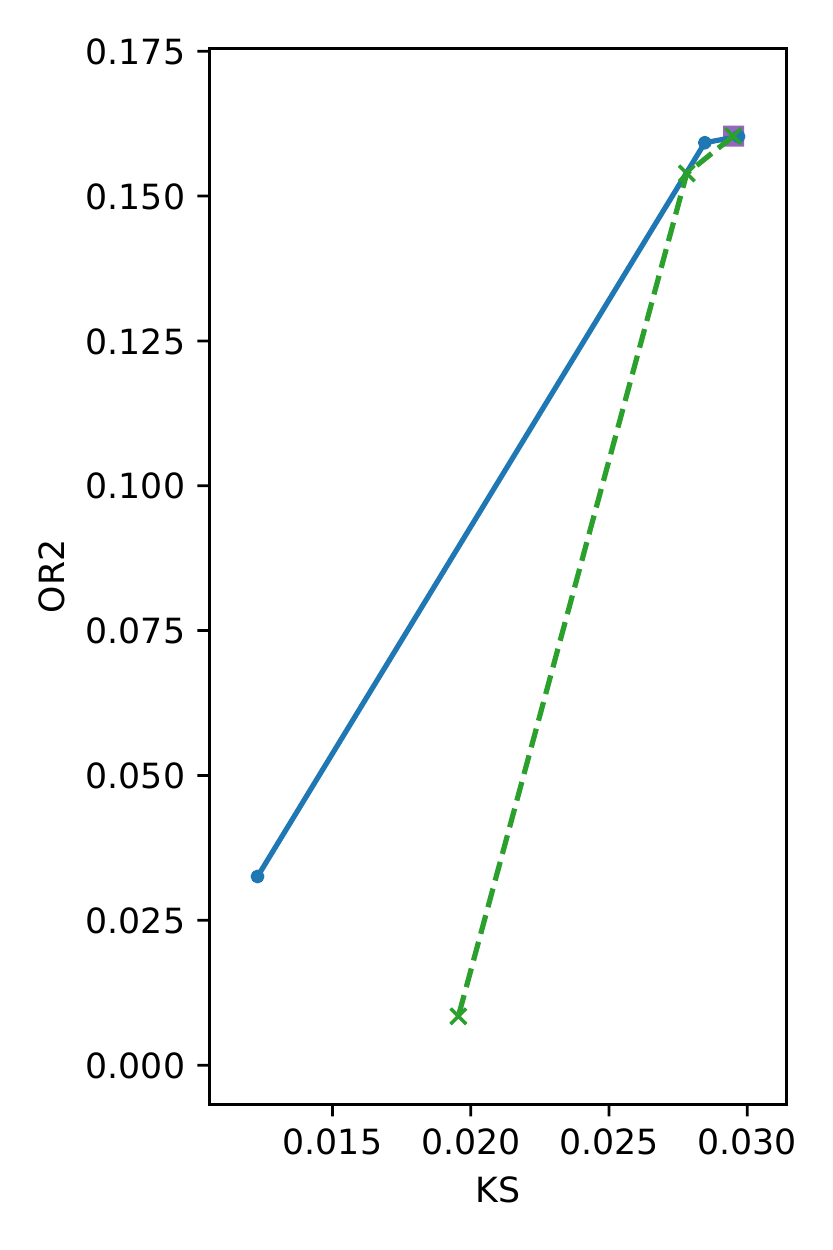}
			\caption{Parkinson's}
		\end{subfigure}\hfill
		\begin{subfigure}[t]{0.33\linewidth}
			\includegraphics[width=\linewidth]{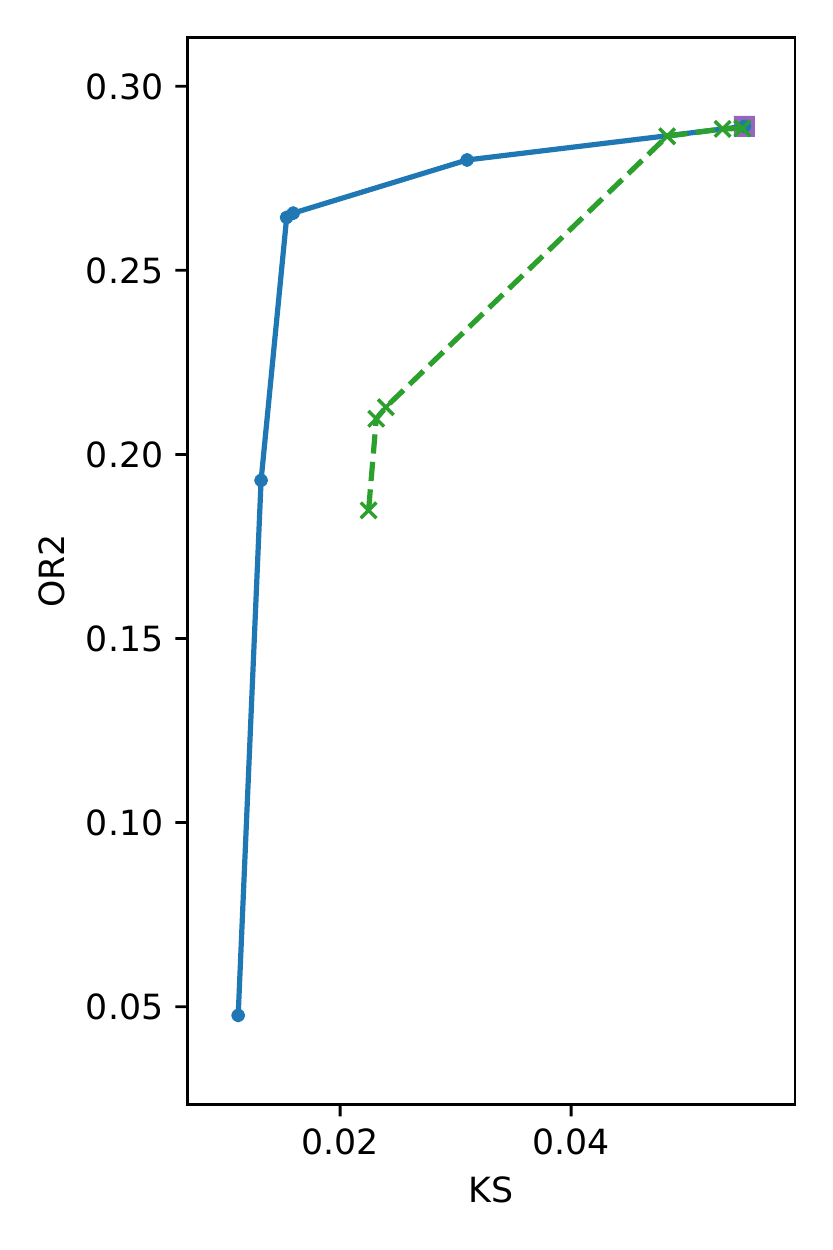}
			\caption{Wine Quality}
		\end{subfigure}\hfill\\
		\begin{subfigure}[t]{0.33\linewidth}
			\includegraphics[width=\linewidth]{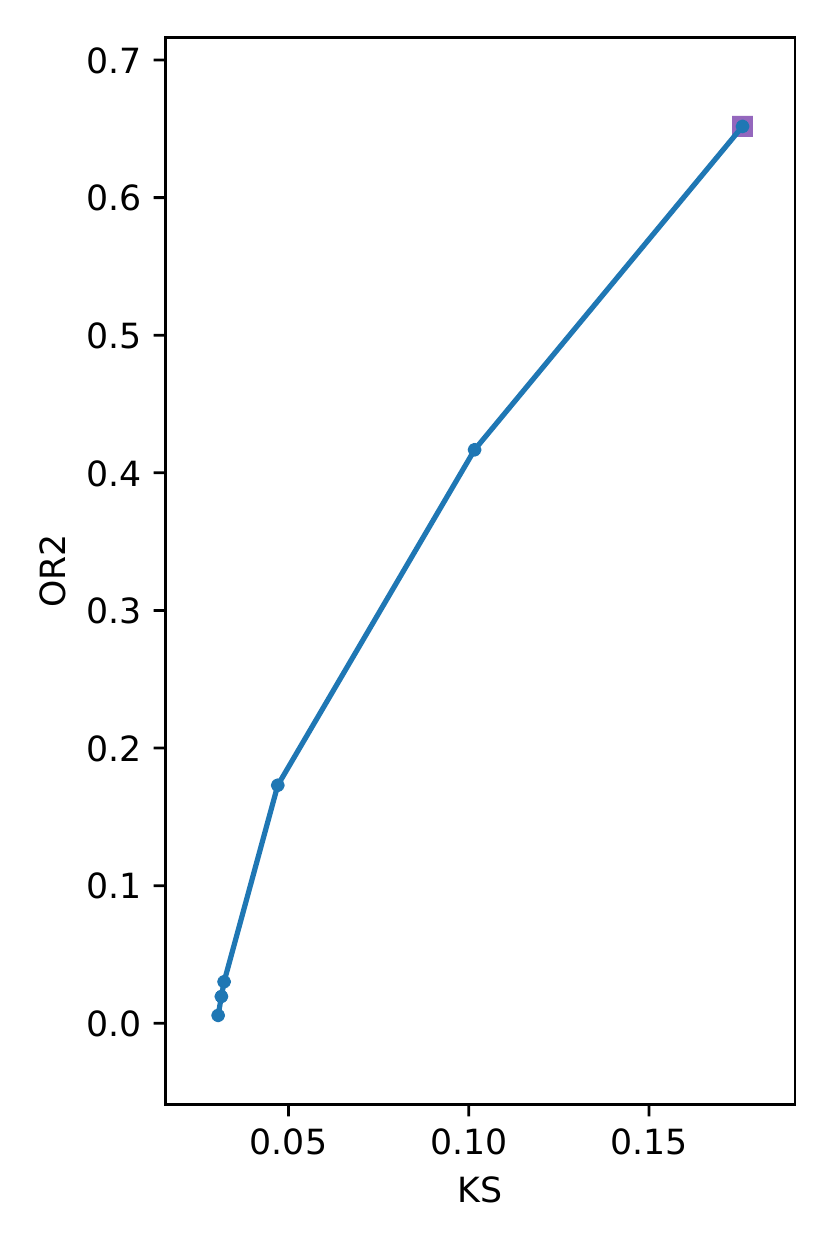}
			\caption{Communities}
		\end{subfigure}\hfill
		\begin{subfigure}[t]{0.33\linewidth}
			\includegraphics[width=\linewidth]{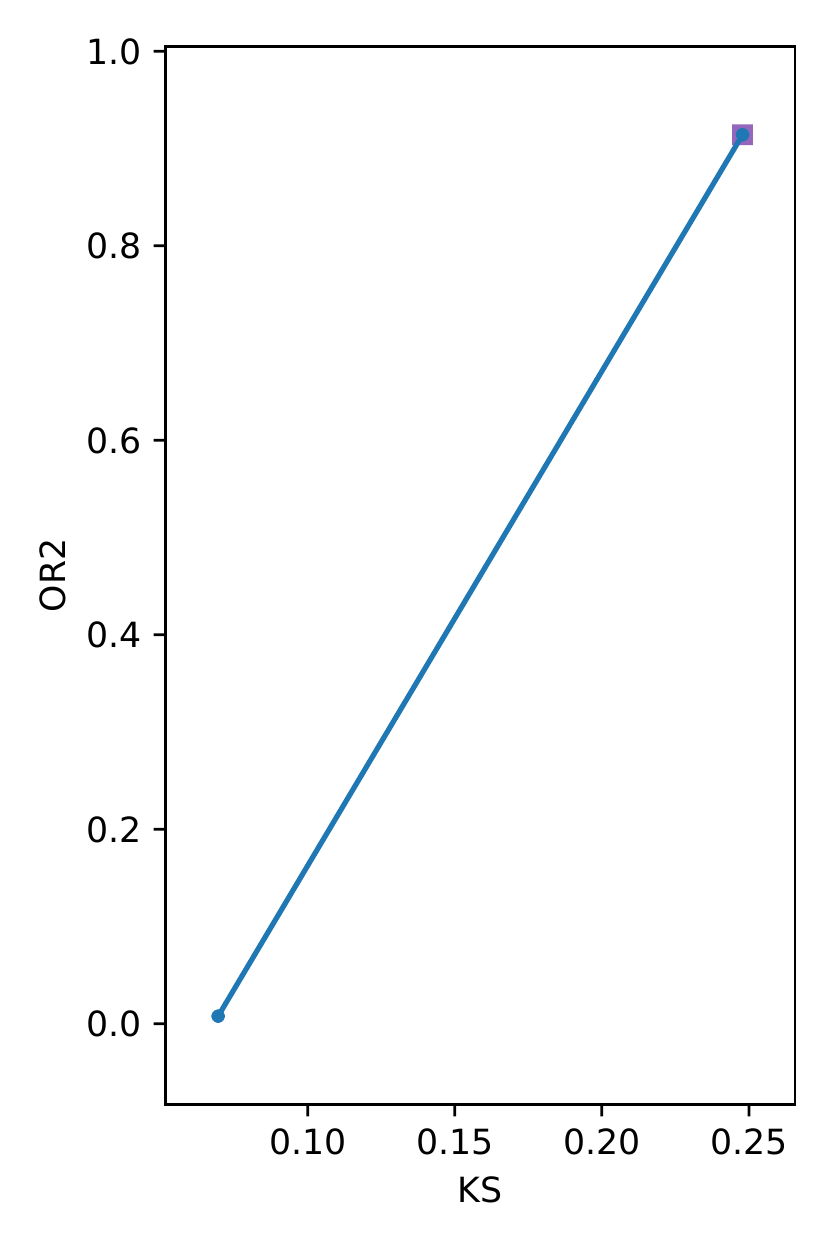}
			\caption{Energy}
		\end{subfigure}\hfill
		\begin{subfigure}[t]{0.33\linewidth}
			\includegraphics[width=\linewidth]{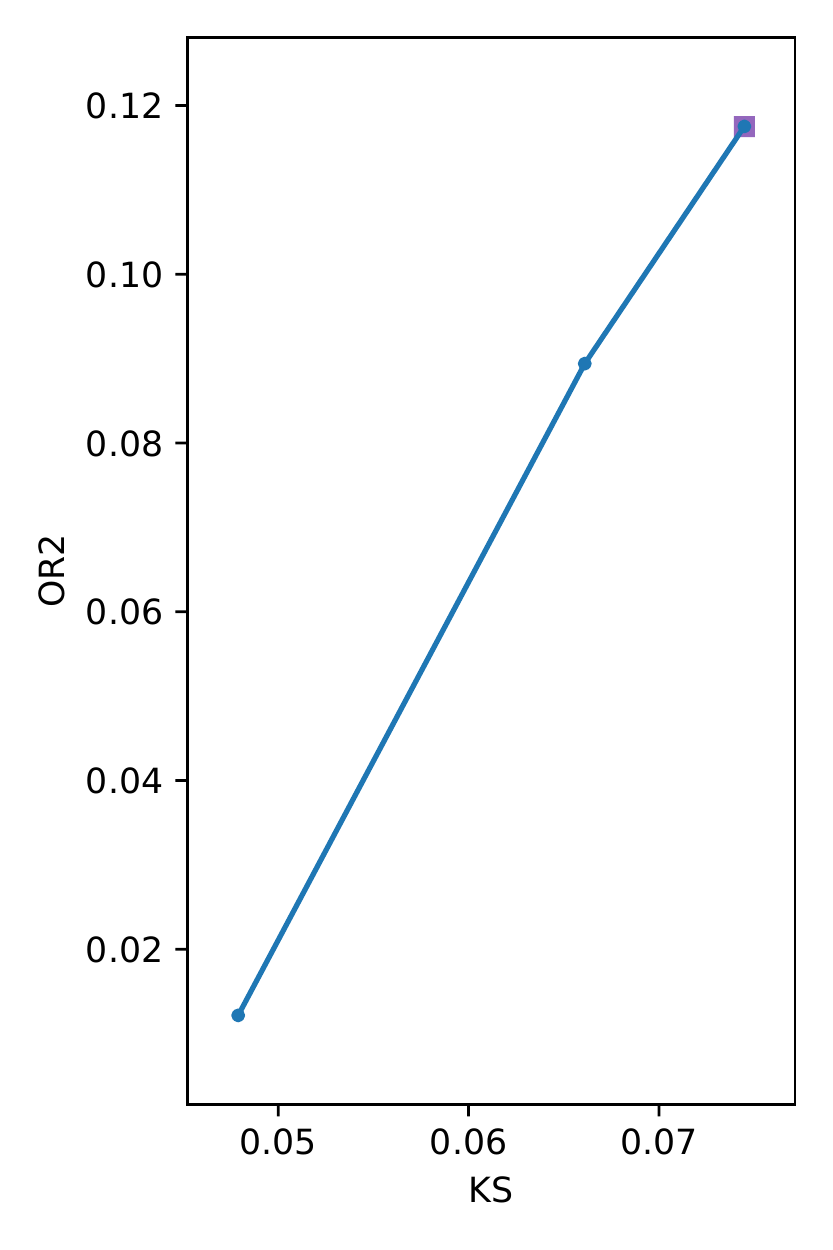}
			\caption{Music}
		\end{subfigure}\hfill		
		\caption{\label{fig:freg} Pareto frontiers for fair regression. The approaches compared are the FO formulation (solid line) and Berk et al. \cite{berk2017convex} (dashed line). The square mark denotes ridge regression without any fairness modifications.}
	\end{center}
\end{figure*}

\subsection{Fair Regression}

We next consider regression problems using another series of datasets {For the FO approach, we consider the formulation in Example \ref{ex:freg}. We perform five-fold cross validation repeated five times. The Pareto frontiers of different approaches are shown in Fig \ref{fig:freg}. Accuracy is measured by the out-of-sample $R^2$ (OR2), which means that higher values of OR2 implies better accuracy. Fairness is measured by the Kolmogorov-Smirnov distance between the joint and product distributions of the model prediction and the protected information.The variance of the results over the five repetitions is low, and so this is not plotted to make the results easier to visualize.} As the method of \cite{berk2017convex} is unable to accommodate non-binary protected attributes, we only provide results for the appropriate datasets. Again, we note that our method can reduce the bias of a typical linear regression problem. {Our method generally does better than \cite{berk2017convex} on datasets where \cite{berk2017convex} can be applied.}

\subsection{Case Study: Morphine Dosing}

Opioid overdoses, including from illicit heroine and synthetic fentanyl, have become the leading cause of death in Americans under 50 \cite{salam2017opioid}. Today, Americans comprise 4.6\% of the global population, but 51.2\% of global morphine usage. Hence there has been much recent interest in regulated and disciplined methods for dosing \cite{manchikanti2017responsible}. At the same time, recent reports have indicated that women and low-income patients are more likely to be under-diagnosed for pain or made to wait longer for a diagnosis \cite{billock2018pain,dusenbury2018everybody}. Thus, we seek to employ FO in order to train an individualized dosing policy that adapts to each patient's measurements and status, but can be made certifiably fair with regards to protected labels. 

%This has largely arisen due to misguided views in the 1990's on the danger of opioids \cite{lasagna1965addicting,mcquay1999opioids}. 

We extracted data for 7156 morphine prescriptions made to 4612 unique patients extracted from the publicly-available Multiparameter Intelligent Monitoring in Intensive Care (MIMIC III) database \cite{saeed2011multiparameter}. For each patient, we collected age (at the time of prescription), heart rate, breath rate, blood pressure (both systolic and diastolic), weight and temperature. In all cases, measurements are the latest possible within 48 hours of prescription. We also collect, as categorical variables, admission type (ER, urgent care or other), service type (surgery or medical), ethnicity (black, white or other), gender (male or female) and insurance type (private or governmental). We also note the presence of embolism or obesity amongst the diagnoses of the patients at admission. We exclude all patients who are not prescribed Morphine Sulfate to be taken intravenously, and all patients for whom the appropriate measurements were not available. Since there are medical justifications for the consideration of gender and ethnicity in opioid dosing, we decide to instead consider insurance type as our protected variable in this analysis. To begin, we conduct a standard linear regression to determine if insurance type does currently play a role in, or is at least highly correlated with, morphine dosage, conditional on all other variables considered. The results found that insurance type had a large magnitude coefficient with $p < 0.001$, which provides some statistical evidence that insurance type is correlated to dosing even after adjusting for the other predictor variables.

\begin{figure*}[!t]
	\begin{center}
		\begin{subfigure}[t]{0.33\linewidth}
			\includegraphics[width=\linewidth]{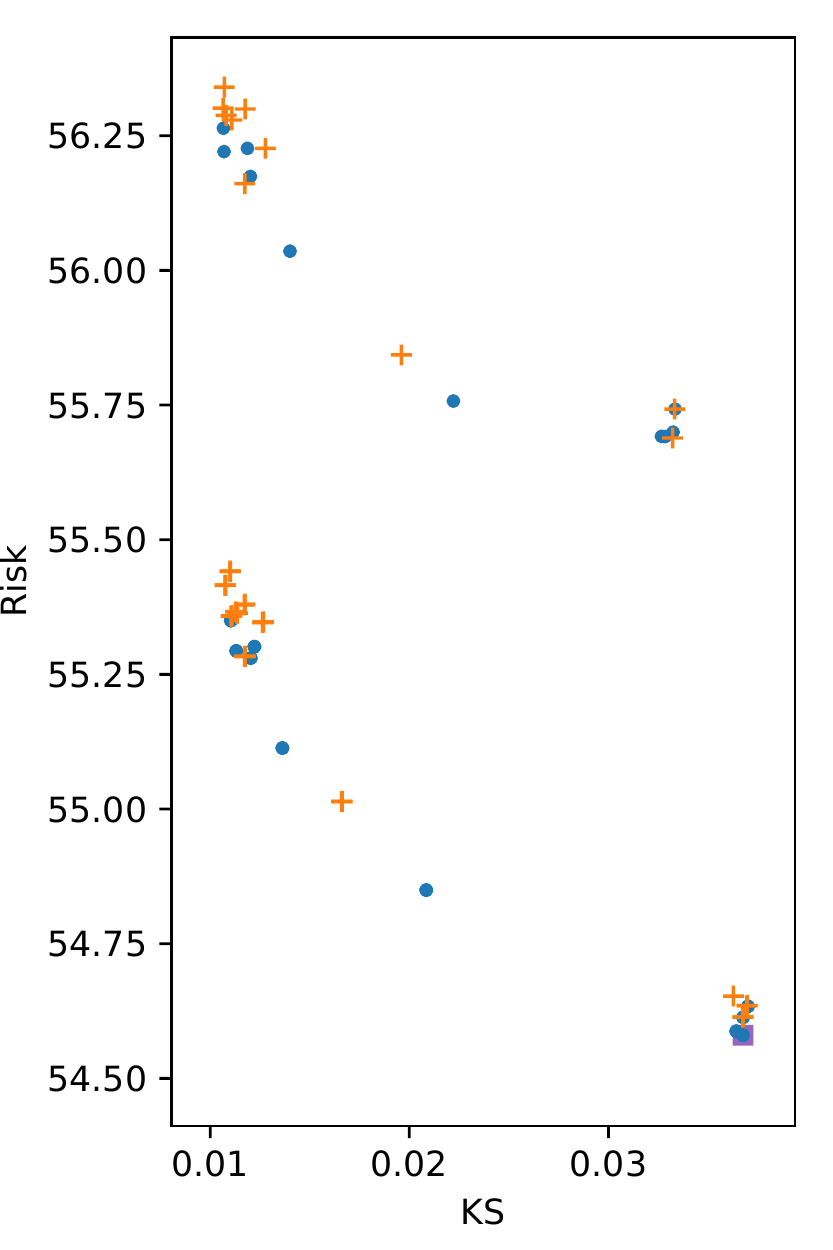}
			\caption{\label{fig:fqua_cv} Cross-Validation}
		\end{subfigure}\hfill
		\begin{subfigure}[t]{0.33\linewidth}
			\includegraphics[width=\linewidth]{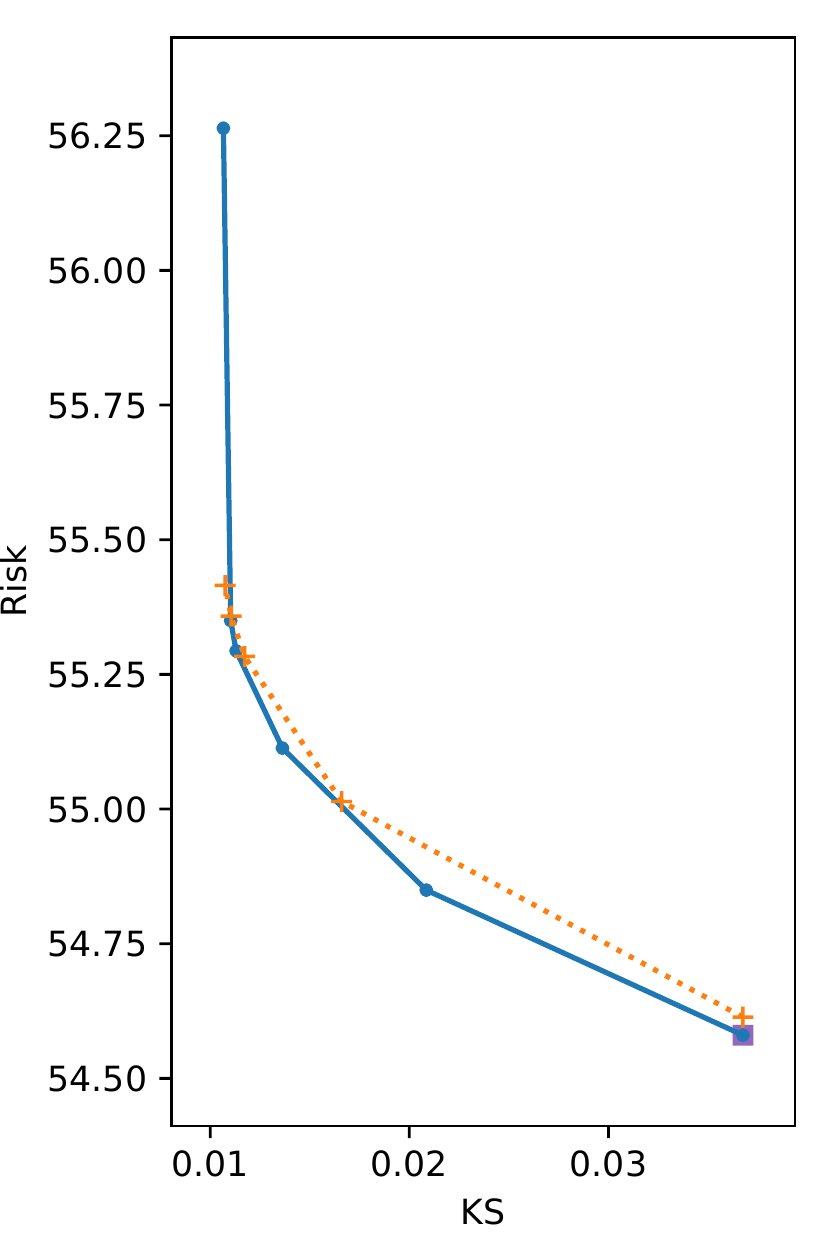}
			\caption{\label{fig:fqua_lpf} Level Pareto Frontiers }
		\end{subfigure}\hfill
		\begin{subfigure}[t]{0.33\linewidth}
			\includegraphics[width=\linewidth]{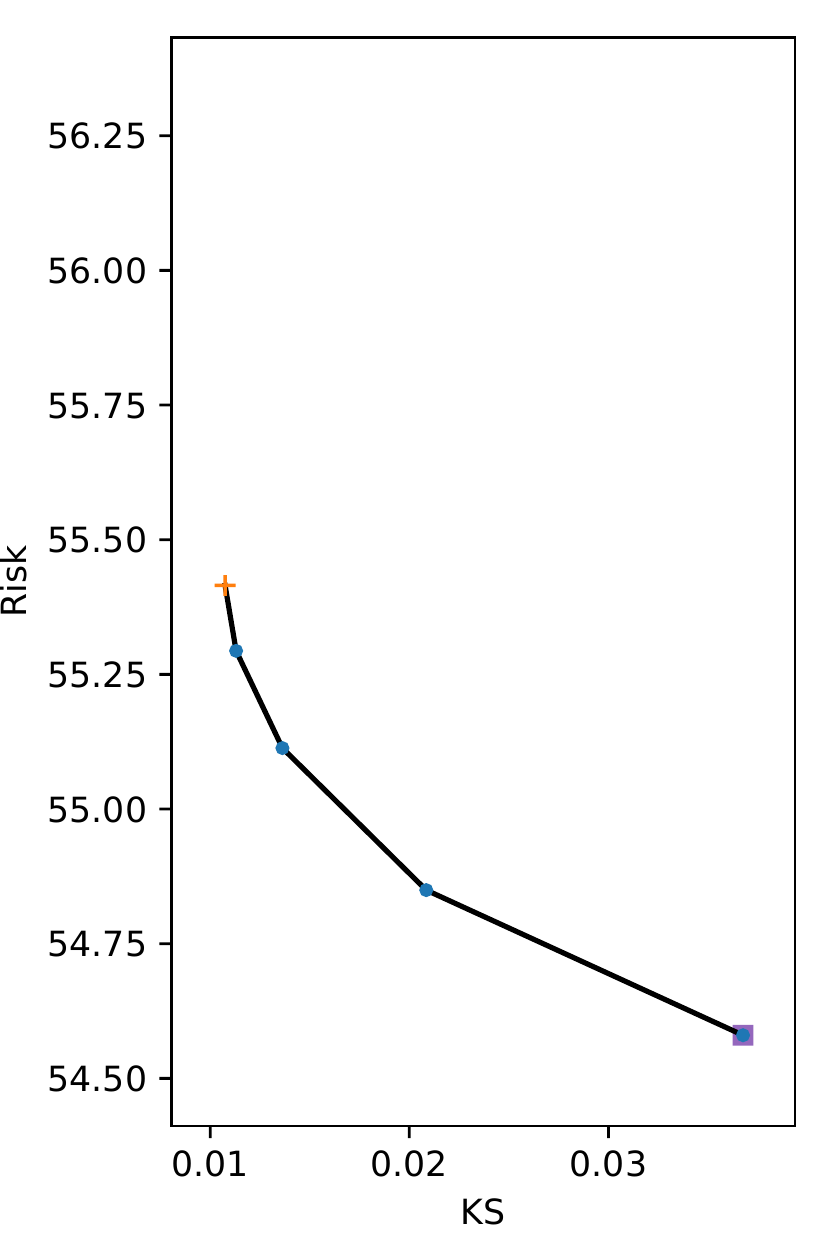}
			\caption{\label{fig:fqua_pf} Pareto Frontier}
		\end{subfigure}
		\caption{\label{fig:cv_fqua} Pareto frontier of learned morphine dosage rules. Five-fold cross-validation repeated five times identifies points of possible tradeoff between model accuracy (measured by risk) and fairness (measured by Kolmogorov-Smirnov distance between the joint and product distributions of the learned dosage and the insurance type) using the FO problem (left), Pareto frontiers can be constructed for each individual level of the FO problem (middle), and a single Pareto frontier can be constructed for all the levels of the FO problem (right). For the points, circles are level-(1,1), pluses are level-(1,2), and the square is quantile regression without fairness modifications.}
	\end{center}
\end{figure*}

%
%\begin{figure*}[t]
	%\begin{center}
		%\begin{subfigure}[t]{0.47\linewidth}
			%\includegraphics[width=\linewidth]{MorphineAUC}
			%\caption{\label{fig:morphauc} AUC vs. Disparate impact for morphine dosage.}
		%\end{subfigure}\qquad
		%\begin{subfigure}[t]{0.47\linewidth}
			%\includegraphics[width=\linewidth]{ParetoAUC}
			%\caption{\label{fig:paretoauc} Pareto optimal curve for morphine dosage. }
		%\end{subfigure}
	%\end{center}
	%\caption{\label{fig:morphinecurves} The accuracy vs. disparate impact of the learned dosage rule under varying orders of FO and for different hyper-parameters. In \Cref{fig:morphauc}, the red curve reflects the level-(1,1) FO, while the blue curves reflect the level-(1,2) FO for the $\mu_{1,2}$ parameters indicated. All curves represent $\Delta_{1,1}\in\{0,0.05,0.1,0.2\}$. The curve in \Cref{fig:paretoauc} reflects the Pareto optimal curve, showing the lowest loss achievable for all levels of disparate impact.}
%\end{figure*}

One possible risk function for dosing is analogous to the newsvendor problem from the operations research community, where supply must be chosen beforehand to meet random demand and undersupply/oversupply are penalized differently. Recent work formulated a data-driven newsvendor model, where demand is predicted via a quantile regression problem \cite{sachs2015data}. Similarly, we can treat dosage as a matter of supply, with demand being the amount of medication that a specific patient needs. In our case, we impose a linearly increasing cost to both under-prescription and over-prescription, with the cost to over-prescription increasing half as quickly as that of under-prescription. {This means we use the loss $R_n(\delta) = \frac{1}{n}\sum_{i=1}^n \max\{0,\delta(X_i) - Y_i\} + 2\max\{0,-(\delta(X_i) - Y_i)\}$ and consider decision rules of the form $\delta(x) = Bx$. The nondifferentiability of this loss is easily handled by introducing the slack variables $s_i,t_i$ and noting that $R_n(\delta) = \frac{1}{n}\sum_{i=1}^n (s_i + 2\cdot t_i)$ subject to the constraints $s_i \geq 0$, $s_i \geq \delta(X_i) - Y_i$, $t_i \geq 0$, and $t_i \geq -(\delta(X_i) - Y_i)$.} This reflects the short-term nature of the risks of under-prescription, and the long-term nature of the risks of over-prescription. Given the features described above (excluding insurance payer), we then formulate varying levels of our FO to solve the quantile regression problem that specifies dosing. 

\begin{figure*}[t]
	\begin{center}
		\begin{subfigure}[t]{0.3\linewidth}
			\includegraphics[width=\linewidth]{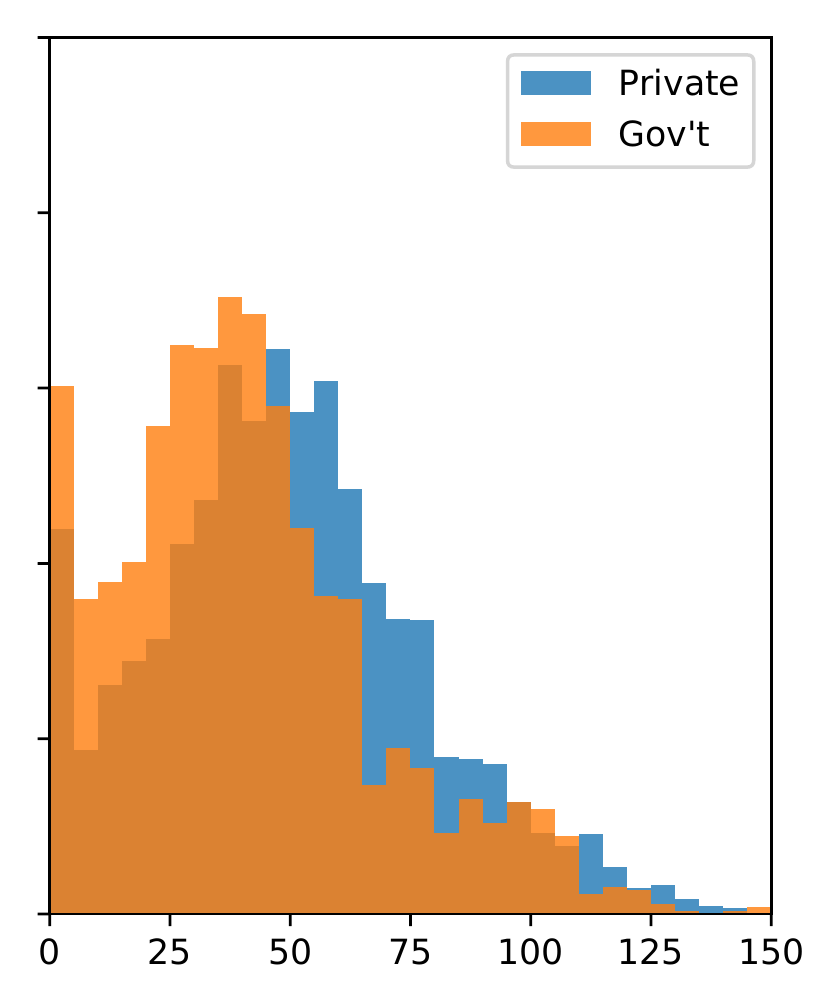}
			\caption{\label{fig:morphqr} \scriptsize Quantile regression}
		\end{subfigure}\hfill
		\begin{subfigure}[t]{0.3\linewidth}
			\includegraphics[width=\linewidth]{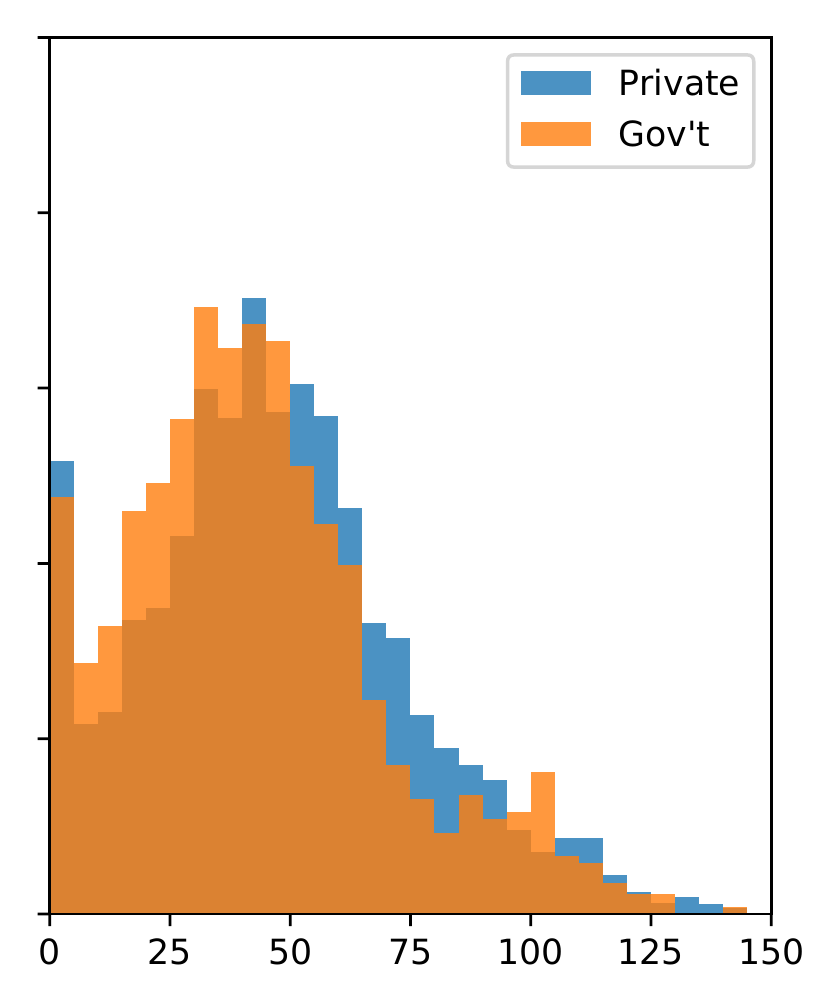}
			\caption{\label{fig:morphfo1} \scriptsize level-(1,1) FO}
		\end{subfigure}\hfill
		\begin{subfigure}[t]{0.3\linewidth}
			\includegraphics[width=\linewidth]{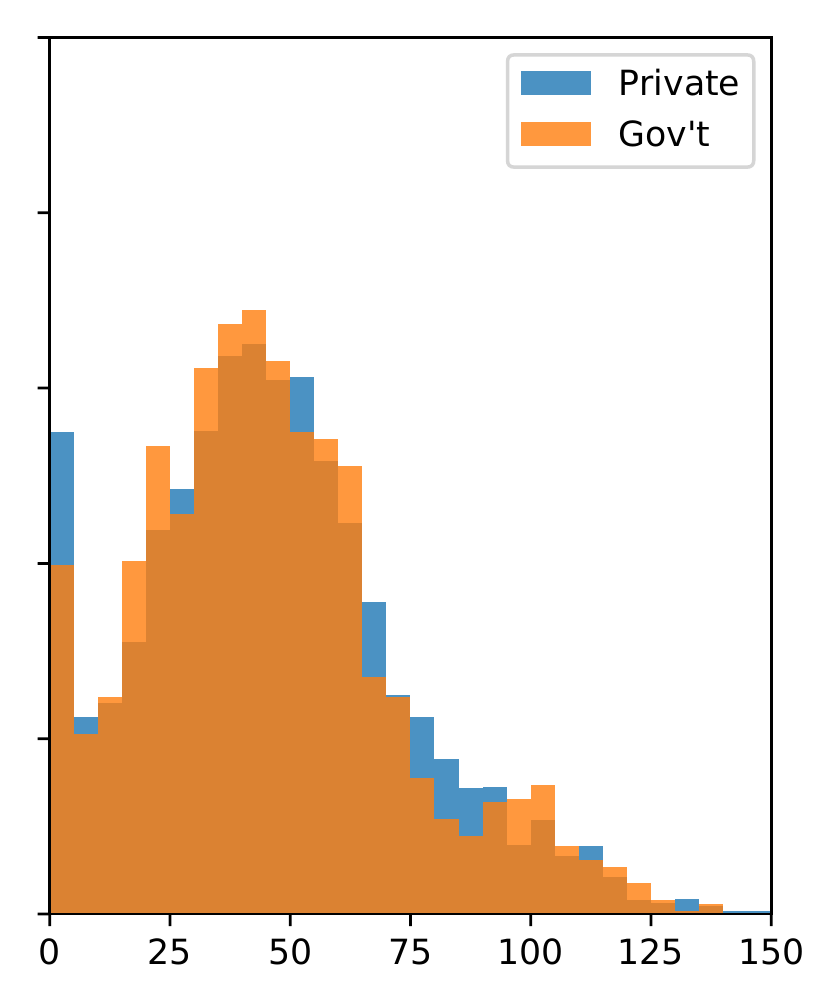}
			\caption{\label{fig:morphfo2} \scriptsize level-(1,2) FO}
		\end{subfigure}
	\end{center}
	\caption{\label{fig:morphdist} Distributions of morphine dosage, conditional on insurance type, for varying levels of FO. Histograms are shown of morphine dosages recommended by rules generated using quantile regression with no fairness modifications (left), the level-(1,1) FO problem (center), and the level-(1,2) FO problem. These histograms are generated by combining the recommended dosages for the hold-out data when doing five-fold cross-validation repeated five times. These results show how using increasing levels of the FO problem can yield more similar distributions. Note that all negative dosage recommendations from the respective models are replaced with zero.}
\end{figure*}

The results of our analysis are displayed in Fig \ref{fig:cv_fqua} and Fig \ref{fig:morphdist}. In Fig \ref{fig:cv_fqua}, the tradeoff between risk and fairness is displayed, as well as the range of best possible dosage rules. Visual evidence of the reduction in disparate impact is shown in Fig \ref{fig:morphdist}, which presents the difference in the distribution of dosage levels across insurance types for standard Quantile Regression (QR), the level-(1,1) FO { with hyperparameters that provide an intermediate tradeoff between risk and fairness},  and the level-(1,2) FO { with hyperparameters that provide the maximum level of fairness achievable}. There is a clear disparity between the distributions in Fig \ref{fig:morphqr}, but this difference is significantly reduced in Fig \ref{fig:morphfo1} and even more so in Fig \ref{fig:morphfo2}. {In fact, Fig \ref{fig:cv_fqua} shows that an intermediate tradeoff between risk and fairness using the level-(1,1) FO problem increases risk by 0.5\% while improving fairness by 45\%, whereas the maximum fairness achievable by the level-(1,2) FO problem increases the risk by only 1.5\% while improving fairness by 70\%.}

\section{Conclusion}

We proposed an optimization hierarchy for fair statistical decision problems, which provides a systematic approach to fair versions of hypothesis testing, decision-making, estimation, regression, and classification. We proved that higher levels of this hierarchy asymptotically impose independence between the output of the decision rule and the protected variable as a constraint in corresponding statistical decision problems. We demonstrated numerical effectiveness of our hierarchy using several data sets. An important question that remains to be answered is how to tune the hyperparameters in our hierarchy. Our theoretical results provide some guidance on how to choose the level of the hierarchy and how to reduce the number of tuning parameters to just one. However, further theoretical and empirical study is needed to better understand the tuning process.
%
%\begin{supplement}
%%\sname{Supplement A}\label{suppA}
%\stitle{Standard Errors of Pareto Frontiers}
%\slink[doi]{URL GOES HERE}
%\sdatatype{.zip}
%\sdescription{Excel spreadsheets containing 95\% standard errors of the points shown on the Pareto frontiers in Fig \ref{fig:fsvm}, Fig \ref{fig:fsvm2}, and Fig \ref{fig:freg}.}
%\end{supplement}

\bibliographystyle{imsart-number}
\bibliography{fopt}

\end{document}